\documentclass[10pt,a4paper]{article}

\usepackage[english]{babel}
\usepackage[T1]{fontenc}
\usepackage[dvips]{graphicx}
\usepackage{amsmath, amsfonts, amsthm}
\usepackage{verbatim}
\usepackage{amsmath}
\usepackage{amsfonts}
\usepackage{amssymb}
\usepackage[usenames,dvipsnames]{color}
\usepackage{mathrsfs}

\numberwithin{equation}{section}

\usepackage{enumerate}

\usepackage[latin1]{inputenc}
\usepackage{graphicx}
\usepackage[all]{xy}
\usepackage{geometry}
\usepackage{bbm}
\usepackage{fullpage}

\numberwithin{equation}{section}

\newcommand{\Z}{\mathbb{Z}}    
\newcommand{\R}{\mathbb{R}}    
\newcommand{\C}{\mathbb{C}}    
\newcommand{\T}{\mathbb{T}}
\newcommand{\F}{\mathcal{F}}

\renewcommand{\H}{\mathcal{H}}
\renewcommand{\L}{\mathcal{L}}

\newcommand{\scat}{\mathcal{R}}
\newcommand{\class}{\mathcal{Q}}

\renewcommand{\S}{\mathscr{S}}
\newcommand{\Hz}{H^M_{\zeta}}
\newcommand{\Hzc}{H^{M}_{\zeta,\C}}
\newcommand{\B}{\mathcal{B}}
\newcommand{\mmod}[1]{\left| #1 \right|}

\newcommand{\norm}[1]{\left\|#1 \right\|}
\newcommand{\K}{\mathcal{K}}
\newcommand{\G}{\mathcal{G}}
\newcommand{\HT}{\mathcal{H}}
\newcommand{\intx}[1]{\int\limits_{#1}^{+\infty}}

\newcommand{\Ltt}{L^{2}_{x\geq a} L^2}

\newcommand{\Czt}{C^{0}_{x\geq a} L^2}
\newcommand{\Czp}{C^{0}_{x \geq a} L^\beta}
\newcommand{\Lxyc}{L^2_{x\geq c} L^2_{y\geq 0}}

\newcommand{\Lc}{L^2_{x \geq c}}

\newcommand{\Loc}{L^1_{x \geq c}}

\newcommand{\Lic}{L^{\infty}_{x \geq c}}

\newcommand{\intzi}{\int\limits_0^{+\infty}}
\newcommand{\intiz}{\int\limits_{-\infty}^0}
\newcommand{\intii}{\int\limits_{-\infty}^{+\infty}}
\newcommand{\derk}{\partial_k}
\newcommand{\derx}{\partial_x}
\newcommand{\dery}{\partial_y}
\newcommand{\derz}{\partial_z}
\newcommand{\bF}{\mathbf{F}_\sigma}
\newcommand{\bH}{\mathbf{H}_\sigma}
\newcommand{\bG}{\mathbf{G}_\sigma}

\newcommand{\kdv}{U_{KdV}}
\newcommand{\airy}{U_{Airy}}
\newcommand{\la}{\langle}
\newcommand{\ra}{\rangle}
\newcommand{\rot}{\Omega}

\newcommand{\x}{\langle x \rangle}

\renewcommand{\Im}{\operatorname{Im}}
\renewcommand{\Re}{\operatorname{Re}}
\newcommand{\Lyz}{L^2_{y \geq 0}}
\newcommand{\Cyz}{C^0_{y \geq 0}}
\newcommand{\Cxy}{C^0_{x \geq c, y \geq 0}}

\newcommand{\CxLy}{C^0_{x \geq c} L^2_{y \geq 0}}
\newcommand{\Hoc}{H^1_{x \geq c}}

\DeclareMathOperator*{\esssup}{ess\,sup}

\newtheorem{theorem}{Theorem}[section]
\newtheorem{lemma}[theorem]{Lemma}
\newtheorem{cor}[theorem]{Corollary}
\newtheorem{proposition}[theorem]{Proposition}
\newtheorem{remark}[theorem]{Remark}
\newtheorem{definition}[theorem]{Definition}

\makeatletter
\renewcommand{\paragraph}{\@startsection{paragraph}{4}{0ex}%
   {-3.25ex plus -1ex minus -0.2ex}%
   {1.5ex plus 0.2ex}%
   {\normalfont\normalsize\bfseries}}
\makeatother

\stepcounter{secnumdepth}
\stepcounter{tocdepth}

\begin{document}
\title{One smoothing property of the scattering map of the KdV on $\R$}
\author{A. Maspero\footnote{Institut f\"ur Mathematik, Universit\"at Z\"urich, Winterthurerstrasse 190, 
CH-8057 Z\"urich, \texttt{alberto.maspero@math.uzh.ch}}, B. Schaad \footnote{Department of Mathematics, 
University of Kansas, 
405 Snow Hall, 
1460 Jayhawk Blvd, 
Lawrence, Kansas 66045-7594, \texttt{schaadbeat@gmail.com} }} 
\maketitle
\begin{abstract}
In this paper we prove that in appropriate weighted Sobolev spaces, in the case of no bound states, the scattering map of the Korteweg-de Vries (KdV) on $\R$ is  a perturbation of the Fourier transform by a regularizing operator. As an application of this result, we show that the difference of the KdV flow and the corresponding Airy flow is 1-smoothing.
\end{abstract}
\section{Introduction}
In the last decades the problem of a rigorous  analysis of the theory of infinite dimensional integrable Hamiltonian systems in 1-space dimension has been widely studied. 
These systems come up in two setups:  (i) on compact intervals (finite volume) and (ii) on
infinite intervals (infinite volume). The dynamical behaviour of the systems in the two setups have many similar
features, but also distinct ones, mostly due to the different manifestation of dispersion.

The analysis of the finite volume case is now quite well understood. Indeed, Kappeler  with collaborators introduced a series of methods in order to  construct rigorously Birkhoff coordinates (a cartesian version of action-angle variables) for 1-dimensional integrable Hamiltonian PDE's on $\T$. The program succeeded in many cases,  like Korteweg-de Vries (KdV) \cite{kamkdv},  defocusing and focusing Nonlinear Schr\"odinger (NLS) \cite{kappeler_grebert,kappeler_lohrmann_topalov_zung}. In each case considered, it has been proved that  there exists a real analytic symplectic diffeomorphism, the {\em Birkhoff map}, between two scales of Hilbert spaces which conjugate the nonlinear dynamics to a linear one.\\
An   important property of  the Birkhoff map $\Phi$  of the KdV on $\T$ and its inverse $\Phi^{-1}$ is the semi-linearity, i.e., the nonlinear part  of $\Phi$ respectively $\Phi^{-1}$ is $1$-smoothing. A local version of this result was first proved by  Kuksin and Perelman  \cite{kuksinperelman} and later extended globally by Kappeler, Schaad and Topalov  \cite{beat2}. It plays an important role in the perturbation theory of KdV -- see \cite{kuksin_damped_kdv} for   randomly perturbed KdV equations and  \cite{Erdogan_Tzirakis_forced} for forced and weakly damped problems. The semi-linearity of $\Phi$ and $\Phi^{-1}$  can be used to prove  $1$-smoothing properties of the KdV flow in the
periodic setup  \cite{beat2}.

The analysis of the  infinite volume case was developed mostly during the '60-'70 of the last century, starting from the pioneering works of Gardner, Greene,  Kruskal and Miura  \cite{gardner1,gardner6} on the KdV on the line.
 In these works  the authors showed that the KdV can be integrated by a {\em scattering transform} which maps a function $q$, decaying sufficiently fast at infinity, into the spectral data of the operator $L(q) := -\derx^2 + q$. Later, similar  results were obtained by Zakharov and Shabat for the NLS on $\R$ \cite{Zakharov_Shabat}, by Ablowitz, Kaup, Newell and
Segur for the Sine-Gordon equation \cite{AKNS}, and by Flaschka for the Toda lattice  with infinitely many particles \cite{flaschka}.
Furthermore, using the spectral data of the corresponding Lax operators,   action-angle variables were (formally) constructed for each of the equations above  \cite{zakharov_faddeev,zakharov_manakov,mcLaughlin,mcLaughlin_erratum}.  See also \cite{Novikov_book,Faddeev_Takhtajan,ablowitz_book} for monographs about the subject. Despite so much work, the analytic properties of the scattering transform and of the action-angle variables in the infinite volume setup are not yet completely understood. In the present paper we discuss these properties, at least for a special class of potentials.
\vspace{.5em}\\

The aim of this paper is to show that for the KdV on the line, the scattering map is  an analytic perturbation of the Fourier transform by a 1-smoothing nonlinear  operator. 
With the applications we have in mind, we choose a setup for the scattering map so that the spaces considered are left invariant under the KdV flow.
Recall that the KdV equation on $\R$
\begin{equation}
\label{KdV}
\begin{cases}
\partial_t u(t,x) = -\partial_x^3u(t,x) - 6u(t,x) \partial_x u(t,x)  \ , \\
u(0,x) = q(x) \ ,
\end{cases} 
\end{equation}
 is globally in time well-posed in various function spaces such as the Sobolev spaces
$H^N\equiv H^N(\R,\R), N\in \mathbb{Z}_{\geq 2}$ ( e.g. \cite{bona,kato79,kenig_ponce_vega0}), as well as on the weighted spaces
$H^{2N}\cap L^2_M,$ with integers $ N\geq M \geq 1$ \cite{kato}, endowed with the norm 
$\|\cdot\|_{H^{2N}}+\|\cdot\|_{L^2_M}$. Here $L^2_M \equiv L^2_M(\R, \C)$ denotes the space of complex 
valued $L^2$-functions satisfying 
$ \|q\|_{L^2_M}:=\left(\int_{-\infty}^\infty( 1+|x|^2)^{ M} |q(x)|^2dx\right)^{\frac{1}{2}}< \infty.$ 

Introduce for $q \in L^2_M$ with $M \geq 4$ the Schr\"odinger operator $L(q):= -\partial_x^2 + q$ with domain $H^2_\C$, where, for any integer $N \in \Z_{\geq 0}$,  $H^N_\C:= H^N(\R, \C)$. 
For 
$k \in \R$ denote by $f_1(q,x,k)$ and $f_2(q,x,k)$ the Jost solutions, i.e. solutions of
$L(q)f=k^2 f $ with asymptotics
$f_1(q,x,k)\sim e^{ikx},\; x\to \infty, \; f_2(q,x,k)\sim e^{-ikx},\; x\to -\infty$. 
As $f_i(q,\cdot, k), \; f_i(q,\cdot, -k)$, $i= 1,2$, are linearly independent 
for $k \in \R \setminus \{0\}$, one can find coefficients $S(q,k)$, $W(q,k)$ such that  for $k \in \R \setminus \{0\}$ one has
\begin{equation}
\begin{aligned}
\label{refl_tras_rel2}
 f_2(q,x,k)=&  \frac{S(q,-k)}{2ik}f_1(q,x,k) + \frac{W(q,k)}{2ik}f_1(q,x,-k) \ , \\ 
 f_1(q,x,k)=& \frac{S(q,k)}{2ik}f_2(q,x,k) + \frac{W(q,k)}{2ik}f_2(q,x,-k) \ .
\end{aligned} 
 \end{equation}
 It's easy to verify that  the functions $W(q,\cdot)$ and $S(q, \cdot)$ are given by the wronskian identities
 \begin{equation}
\label{wronskian_W}
W(q,k):=\left[f_2, f_1\right](q,k):= f_2(q, x,k)\derx f_1(q,x,k) - \derx f_2(q, x,k) f_1(q, x,k) \ ,
\end{equation}
and
\begin{equation}
\begin{aligned}
\label{wronskian}
&S(q, k):=\left[f_1(q,x,k), f_2(q,x,-k)\right],
\end{aligned}
\end{equation}
which are independent of $x \in \R$.  
The functions $S(q,k)$ and $W(q,k)$ are related to the more often used reflection coefficients $r_\pm(q,k)$ and transmission coefficient $t(q,k)$  by the formulas
\begin{equation}
\label{r.S.rel} 
r_+(q,k) = \frac{S(q, - k)}{W(q,k)}, \quad r_-(q,k) = \frac{S(q,k)}{W(q,k)}, \quad t(q,k) =  \frac{2ik}{W(q,k)} \quad \forall \, k \in \R \setminus \{0\} \  .
\end{equation}

It is well known that for $q$ real valued  the spectrum of
$L(q)$ consists of an absolutely continuous part, given by $[0, \infty)$, and a finite number of eigenvalues referred to as bound states, $-\lambda_n < \cdots < -\lambda_1<0$ (possibly none). Introduce the set 
\begin{equation}
\class:= \left\{ q:\R \to \R \ ,\  q \in L^2_4: W(q,0)\neq 0, \, q \mbox{ without bound states} \right\}. 
\end{equation} 
We remark that the property $W(q,0)\neq 0$ is generic. In the sequel we refer to elements in $\class$ as generic 
potentials without bound states. 
Finally we define 
$$\class^{N,M}:=\class\cap H^{N}\cap L^2_M, \quad N\in \mathbb{Z}_{\geq 0},\quad M\in \mathbb{Z}_{\geq 4}.$$ 
We will see in Lemma \ref{class_open} that for any integers $N\geq 0$, $M\geq 4$, $\class^{N,M}$ is open in $H^N \cap L^2_M$.

Our main theorem analyzes the properties of the scattering map
$q \mapsto S(q,\cdot)$
 which is known to linearize the KdV flow \cite{gardner6}.
To formulate our result on the scattering map in more details let $\S$ denote the set of all functions 
$\sigma: \R \to \C$ satisfying
\begin{enumerate}
\item[(S1)] $\sigma(-k) = \overline{\sigma(k)}, \quad \forall k \in \R$;
\item[(S2)] $\sigma(0) >0$.
\end{enumerate}
For $M \in \Z_{\geq 1}$ define the \textit{real} Banach  space
\begin{equation}
 \begin{aligned}
\label{H^N*}
&\Hz := \lbrace  f \in H^{M-1}_\C : \quad \overline{f(k)}= f(-k), \quad  \zeta \derk^M f \in L^2 \rbrace \ ,
\end{aligned}
\end{equation}
where $\zeta: \R \to \R$ is an odd monotone $C^\infty$ function with 
\begin{equation}
\label{zeta}
 \zeta(k)=k \ \mbox{ for }  \ |k|\leq 1/2  \quad \mbox{ and }  \quad \zeta(k)=1  \ \mbox{ for } \  k\geq 1 \ .
\end{equation}
 The norm on $\Hz $ is  given by
$$
 \norm{f}_{\Hz}^2 := \norm{f}_{H^{M-1}_\C}^2
+ \norm{\zeta \derk^M f}_{ L^2}^2.$$
For any $N, M \in \Z_{\geq 0}$ let
\begin{align}
\S^{M, N} := \S \cap \Hz \cap L^2_{N} \ .
\label{reflspaceNM0}
\end{align}
Different choices of $\zeta$, with $\zeta$ satisfying \eqref{zeta}, lead to the same Hilbert space with equivalent norms.
We will see in Lemma \ref{lem:S.open} that for any integers $N\geq 0$, $M \geq 4$, $\S^{M,N}$ is an open subset of $\Hz \cap L^2_N $.
Moreover let  $\F_{\pm}$ be the Fourier transformations 
defined by $\mathcal{F}_\pm(f) = \int_{-\infty}^{+\infty} e^{\mp 2ikx} f(x) \,dx.$
In this setup, the scattering map $S$ has the following properties -- see Appendix \ref{analytic_map} for a discussion of the notion 
of real analytic.
\begin{theorem}\label{reflthm} For any integers $N \geq 0$, $M \geq 4$, the following holds:
\begin{enumerate}[(i)]
\item The map $$S: \class^{N, M}\to \S^{M, N},\quad q \mapsto S(q, \cdot) $$ is a  real analytic diffeomorphism.
\item The maps $A := S - \mathcal{F}_{-} $ and $B := S^{-1} - \mathcal{F}^{-1}_{-} $ are  1-smoothing, i.e. 
$$A: \class^{N, M}\rightarrow \Hz \cap L^2_{N+1}\quad\text{and} \quad B : \S^{M, N}\rightarrow H^{N+1} \cap L^2_{M-1}   \ . $$
Furthermore they are  real analytic maps.
\end{enumerate}
\end{theorem}

As a first application of Theorem \ref{reflthm} we prove analytic properties of the action variable for the KdV on the line. For a potential $q \in \class$, the action-angle variable were formally defined for $k \neq 0$ by Zakharov and Faddeev  \cite{zakharov_faddeev}  as the densities
\begin{equation}
\label{action_angle} 
I(q,k) := \frac{k}{\pi} \log \left(1+\frac{|S(q,k)|^2}{4k^2} \right) \ , \quad \theta(q,k):= \arg \left(S(q,k) \right) , \quad k \in  \R\setminus\{0\}\ .
\end{equation}
We can write the action as
\begin{equation}
\label{action} 
I(q,k) := -\frac{k}{\pi} \log \left(\frac{4k^2}{4k^2 + S(q,k)S(q,-k)} \right) \ , \quad k \in \R\setminus\{0\} \ .
\end{equation}
By Theorem \ref{reflthm}, $S(q, \cdot) \in \S$, thus property (S2) implies that  $\lim_{k \to 0} I(q,k)$ exists and equals $0$. Furthermore, by (S1), the action $I(q, \cdot)$ is an odd function in $k$, and strictly positive for $k >0$. Thus  we will consider just the case $k \in [0, +\infty)$. The properties of $I(q,\cdot)$ for $k$ near $0$ and $k$ large are described separately.
\begin{cor}
\label{thm:actions}
For any integers $N \geq 0$, $M \geq 4$,
the maps
$$
\class^{N,M} \to L^1_{2N+1}([1,+\infty), \R)  \ , \quad q \mapsto \left.I(q, \cdot)\right|_{[1,\infty)}
$$
and
$$
\class^{N,M} \to H^M([0,1], \R) \ , \quad q \mapsto \left.I(q, \cdot)\right|_{[0,1]} + \frac{k}{\pi} \ln \left(\frac{4k^2}{4(k^2+1)} \right) 
$$
are  real analytic. Here $\left.I(q, \cdot)\right|_{[1,\infty)}$ (respectively $\left.I(q, \cdot)\right|_{[0,1]}$) denotes the restriction of the function $k\mapsto I(q,k)$ to the interval $[1, \infty)$ (respectively $[0,1]$).
\end{cor}

Finally we compare 
solutions of \eqref{KdV} to solutions of the Cauchy problem for the Airy equation on $\R$,
\begin{equation}
\label{Airy}
\begin{cases}
\partial_t v(t,x) = -\partial_x^3v(t,x) \\
v(0,x) = p(x)
\end{cases}
\end{equation}
Being a linear equation with constant coefficients, one sees  that the Airy equation is globally in time well-posed on $H^N$ and $H^{2N}\cap L^2_M,$ with integers $N \geq M \geq 1$ (see  Remark \ref{rem.airy.flow} below). Denote the  flows of \eqref{Airy} and \eqref{KdV} by $U_{Airy}^t(p):=v(t,\cdot)$ respectively $U_{KdV}^t(q):= u(t,\cdot)$. Our third result is to show that for $q\in H^{2N}\cap L^2_M$ with no bound states and $W(q,0)\neq 0$, the difference 
$U_{KdV}^t(q)-U_{Airy}^t(q)$ is 1-smoothing, i.e. it takes values in $H^{2N+1}$. 
More precisely we prove the following theorem.

\begin{theorem}\label{firstapprox}
Let $N$, $M$ be integers with $N \geq 2M \geq 8$. Then the following holds true:
\begin{enumerate}[(i)]
\item $\class^{N, M}$ is invariant under the KdV flow.
\item For any $q\in \class^{N, M}$ the difference $U_{KdV}^t(q)-U_{Airy}^t(q)$ takes values in $H^{N+1}\cap L^2_M$. Moreover the map 
\begin{align*}
\class^{N,M} \times \R_{\geq 0}\to& H^{N+1}\cap L_M^2, \qquad (q,t)\mapsto U_{KdV}^t(q)-U_{Airy}^t(q)
\end{align*} is continuous and for any fixed $t$ real analytic in $q$.
\end{enumerate}
\end{theorem}

\vspace{1em}
\noindent{\em Outline of the proof: } In Section 2 we study analytic properties of the Jost functions $f_j(q,x,k)$, $j=1,2$, in appropriate Banach spaces. We use these results in Section \ref{sec:dir.scat} to prove  the direct scattering part of Theorem \ref{reflthm}.  The inverse scattering part of Theorem \ref{reflthm} is proved in Section \ref{sec:inv.scat}. Finally in Section 5 we prove Corollary \ref{thm:actions} and Theorem \ref{firstapprox}.

\vspace{1em}
\noindent{\em Related works: } As we mentioned above, this paper is motivated in part from the study of the  $1$-smoothing property of  the KdV flow in the
periodic setup,   established recently in \cite{babin_ilyin_titi,erdogan_tzirakis2,beat2}. 
In \cite{beat2} the one smoothing property of the Birkhoff map has been exploited to prove that for $q \in H^N(\T, \R)$,
$N\geq 1$, the difference $\kdv^t(q) - \airy^t(q)$ is bounded in $H^{N+1}(\T, \R)$ with a bound which grows 
linearly in time.

Kappeler and Trubowitz  \cite{kapptrub, kapptrub2} studied analytic  properties of the scattering map $S$ between weighted Sobolev spaces. More precisely, define the spaces
\begin{align*}
& H^{n,\alpha} := \left\{ f \in L^2 : x^\beta \derx^j f \in L^2 , 0 \leq j \leq n, 0 \leq \beta \leq \alpha \right\} \ , \\
& H^{n,\alpha}_{\sharp}:= \left\{ f \in H^{n,\alpha} : x^\beta \derx^{n+1} f \in L^2 , 1 \leq \beta \leq \alpha \right\} \ .
\end{align*}
In \cite{kapptrub}, Kappeler and Trubowitz showed that the map $q \mapsto S(q, \cdot)$ is a real analytic diffeomorphism from $\class \cap H^{N,N}$ to $\S \cap H^{N-1,N}_{\sharp}$, $N \in \Z_{\geq 3}$.
They extend their results to potentials with finitely many bound states in \cite{kapptrub2}. Unfortunately, $\class \cap H^{N,N}$ is not left invariant under the KdV flow.

Results concerning the 1-smoothing property of the inverse scattering map  were obtained previously in  \cite{novikovR}, where it is shown that  for a potential $q$ in the space $ W^{n,1}(\R, \R)$ of real-valued functions with weak derivatives up to order $n$ in $L^1$
$$
q(x) - \frac{1}{\pi} \int_\R e^{-2ikx} \chi_c(k) 2ik r_+(q,k) dk \in W^{n+1,1}(\R, \R) \ . 
$$
Here  $c$ is an arbitrary number with $c > \norm{q}_{L^1}$ and $\chi_c(k)=0$ for $|k|\leq c\,$, $\chi_c(k)=|k|-c$ for $c\leq |k|\leq c+1$, and 1 otherwise. The main difference between the result in \cite{novikovR} and ours concerns the function spaces considered. For the application to the KdV we need to choose function spaces such as $H^N \cap L^2_M$ for which KdV is well posed. To the best of our knowledge it is not known if KdV is well posed in $W^{n,1}(\R, \R)$.
Furthermore in \cite{novikovR}  the question of  analyticity  of the map $q \mapsto r_+(q)$ and its inverse  is not addressed.

We remark that Theorem \ref{reflthm} treats just the case of  regular potentials. In \cite{perry1,perry2}   a special class of distributions  is considered. In particular the authors study  Miura potentials $q \in H^{-1}_{loc}(\R, \R)$ such that $q = u' + u^2$ for some $u \in L^1(\R, \R) \cap L^2(\R, \R)$, and prove that the map $q \mapsto r_+$ is bijective and locally bi-Lipschitz continuous  between appropriate  spaces.
Finally we point out the work of Zhou \cite{zhou}, in which  $L^2$-Sobolev space bijectivity for the scattering and inverse scattering transforms associated with the ZS-AKNS system are proved.

\section{Jost solutions}

In this section we assume that the potential $q$ is complex-valued. Often we will assume that   $q \in L^2_{M}$ with $M \in \Z_{\geq 4}$. 
Consider the normalized Jost functions  $m_1(q,x,k):= e^{-ikx}f_1(q,x,k)$ and $m_2(q,x,k):= e^{ikx}f_2(q,x,k)$ which satisfy the following integral equations
\begin{align}
\label{defm}
 &m_1(q, x,k)=1+\int_x^{+\infty} D_k(t-x)\, q(t) \, m_1(q,t,k) dt \\
 &m_2(q, x,k)=1+\int_{-\infty}^{x} D_k(x-t)\, q(t) \, m_2(q,t,k) dt 
 \label{defm2}
 \end{align}
where $D_k(y):= \int_0^y e^{2iks} ds$.

The purpose of this section is to  analyze the solutions of the integral equations \eqref{defm} and \eqref{defm2} in spaces needed for our application to KdV. We adapt the corresponding results of \cite{kapptrub} to these spaces. As \eqref{defm} and \eqref{defm2} are analyzed in a similar way we concentrate  on \eqref{defm} only. For simplicity we write $m(q,x,k)$ for $m_1(q,x,k)$. 

For $1 \leq p \leq \infty$,  $M \geq 1$ and  $ a \in \R, \ 1\leq \alpha<\infty$, $1 \leq \beta \leq \infty$ we introduce the spaces 
$$L^p_{M} := \left\{f: \R \to \C: \; \x^M f \in L^p \right\} \ , \quad  L^\alpha_{x\geq a} L^\beta:=\left\{ f: [a, +\infty) \times \R \rightarrow \C:  \norm{f}_{L^\alpha_{x\geq a} L^\beta} < +\infty \right\}$$ 
where  $\x:= (1+x^2)^{1/2}$, $L^p$ is the standard $L^p$ space, and 
$$\norm{f}_{L^\alpha_{x\geq a} L^\beta} := \Big(\int_{a}^{+ \infty} \norm{f(x, \cdot)}^\alpha_{L^\beta} \,dx \Big)^{1/\alpha} $$
whereas for $\alpha=\infty$, 
$\norm{f}_{L^{\infty}_{x\geq a} L^\beta}:= \sup_{ x \geq a } \norm{f(x, \cdot)}_{L^\beta}.$ 
We consider also the space $C^0_{x\geq a} L^\beta := C^0\left( [a, +\infty), L^\beta \right)$ with 
$\norm{f}_{C^0_{x\geq a} L^\beta} := \sup_{ x\geq a} \norm{f(x, \cdot)}_{L^\beta}< \infty$.
We will use also the space $ L^\alpha_{x\leq a} L^\beta$ of functions $ f: (-\infty, a]\times \R \to \C$ with finite norm 
$\norm{f}_{L^\alpha_{x\leq a} L^\beta} := \Big(\int_{-\infty}^{a} \norm{f(x, \cdot)}^\alpha_{L^\beta} \,dx \Big)^{1/\alpha} $.
Moreover given any  Banach spaces $X$ and $Y$ we denote by $\L(X,Y)$  the Banach space of linear bounded operators from $X$ to $Y$ endowed with the operator norm. If $X=Y$, we simply write  $\L(X)$. \\
For the notion of an analytic map between complex  Banach spaces we refer to Appendix \ref{analytic_map}.\\
 We begin by stating a well known result about the properties of $m$. 
\begin{theorem}[ \cite{deift}] \label{deift_jost} 
Let $q \in L^1_1$. For each $k, \, \Im k \geq 0$, the integral equation
$$
m(x,k) = 1 + \intx{x} D_k(t-x) q(t) m(t,k) dt \ , \qquad x \in \R 
$$
has a unique solution $m \in C^2(\R, \C)$ which solves the equation
$m'' + 2ik m' = q(x) m $ with $m(x,k) \to 1$ as $x \to +\infty$. 
If in addition $q$ is real valued   the function $m$ satisfies the reality condition $\overline{m(q,k)}= m(q,-k)$. 
Moreover, there exists a constant $K>0$  which can be chosen uniformly on  bounded subsets of $L^1_{1}$ such that
the following estimates hold for any $x \in \R$
\begin{enumerate}[(i)]
\item $|m(x,k) - 1| \leq e^{\eta(x) / |k|} \eta(x)/|k|, \quad k \neq 0  $;
\item $|m(x,k) - 1| \leq K \Big((1+ \max(-x,0))\intx{x}(1+|t|) |q(t)| dt \Big)/ (1+ |k|)$;
\item $|m'(x,k)| \leq K_1 \Big(\intx{x}(1+|t|) |q(t)| dt\Big)/(1+ |k|) $
\end{enumerate}
where $\eta(x) = \intx{x} |q(t)| dt$. For each $x$, $m(x,k)$ is analytic in $\Im k >0$ and continuous in $\Im k \geq 0$. In particular, for every $x$ fixed, $k \mapsto m(x,k) -1 \in H^{2+},$ where $H^{2+}$ is the Hardy space of functions analytic in the upper half plane such that $\sup_{y >0}\intii |h(k+iy)|^2 \,dk < \infty$.
\end{theorem}

\vspace{1em}
\noindent {\em  Estimates on the Jost functions.}
 \begin{proposition} 
\label{prop_minLit}
For any $q \in L^2_M$ with $M \geq 2$, $a \in \R$ and $2 \leq \beta \leq +\infty$, the solution   $m(q)$ of \eqref{defm} satisfies  $m(q)-1 \in \Czp \cap \Ltt$. The map $L^2_M \ni q  \mapsto m(q) -1 \in \Czp \cap \Ltt$ is analytic. Moreover there exist constants $C_1, C_2 >0$, only dependent on $a, \beta$,  such that
\begin{equation}
\norm{m(q)-1}_{\Czp}\leq C_1 e^{\norm{q}_{L^1_1}} \norm{q}_{L^2_1},  \quad 
   \norm{m(q)-1}_{\Ltt}\leq C_2 \norm{q}_{L^2_2}\left( 1+ \norm{q}_{L^2_{3/2}} e^{\norm{q}_{L^1_1}}\right).
\end{equation}
\end{proposition}
\begin{remark}
In comparison with \cite{kapptrub}, the novelty of Proposition \ref{prop_minLit} consists in the choice of spaces.
\end{remark}
To prove Proposition \ref{prop_minLit} we first need to establish some auxiliary results.
\begin{lemma} \begin{enumerate}
 \item[(i)] For any $q \in L^1_1$, $a \in \R$ and $1 \leq \beta \leq +\infty$, the linear operator
  \begin{equation}
\K(q): \Czp \to \Czp, \quad f \mapsto \K(q)[f](x,k) := \intx{x} D_k(t-x) q(t) f(t,k) dt
\label{operK}
\end{equation}
 is bounded. Moreover for any $n \geq 1$, the $n^{th}$ composition  $K(q)^n$ satisfies  $\norm{\K(q)^n}_{\L(\Czp)} \leq C^n \norm{q}^n_{L^1_1}/n!\,$ where $C >0$ is a constant depending only on $a$.
\item[(ii)] The map $\K: L^1_1 \to \L\left(\Czp\right), \; q \mapsto \K(q),$ is linear and bounded, and $Id - \K$ is invertible. More precisely,
\begin{align*}
 \left(Id - \K \right)^{-1}: \, L^1_1 & \to \L\left(\Czp\right), \quad  q   \mapsto  \left(Id- \K(q)\right)^{-1}
\end{align*}
is analytic and
$\norm{\left(Id- \K\right)^{-1}}_{\L\left(L^1_1, \Czp\right)}\leq e^{C \norm{q}_{L^1_1}}.$
\end{enumerate}
\label{KinLit}
\end{lemma}
\begin{proof} Let $h \in L^\alpha$ with $\frac{1}{\alpha}+ \frac{1}{\beta}=1$.  Using  $\mmod{D_k(t-x)}\leq |t-x|$, one has 
\begin{align*} 
 \mmod{\intii h(k) \K(q)[f](x,k) dk}& \leq \intx{x} dt \, |t-x| |q(t)| \norm{f(t, \cdot)}_{L^\beta} \norm{h}_{L^\alpha} \\
 &  \leq 
 \left(\intx{a} |t-a| |q(t)| dt \right)
\norm{f}_{\Czp} \norm{h}_{L^\alpha},
\end{align*}
and hence $\norm{\K(q)}_{\L(\Czp)}\leq \intx{a} |t-a| |q(t)| dt\leq C \norm{q}_{L^1_1}$, where $C>0$ is a constant depending just on $a$. To compute the norm of the iteration of the map $\K(q)$ it's enough to proceed as above and exploit the fact 
that the integration in $t$ is over a simplex, yielding 
 $\norm{\K(q)^n}_{\Czp}\leq C^n \norm{q}_{L^1_1}^n/n!$ for any $n \geq 1$. Therefore the Neumann series of the operator $ \Big(Id - \K(q) \Big)^{-1}=\sum_{n \geq 0} \K(q)^n$
converges absolutely in $\L\left(\Czp \right)$. Since $\K(q)$ is linear and bounded in $q$, the analyticity and, by item $(i)$, the claimed estimate for $(Id - \K)^{-1}$ follow.
\end{proof}

\begin{lemma} Let $a \in \R$. 
\begin{enumerate}[(i)]
 \item  For any $q \in L^2_{3/2}$, $\K(q)$ defines a bounded linear operator $\Ltt \to \Ltt$.  Moreover the $n^{th}$ composition $K(q)^n$ satisfies 
$$ \norm{\K(q)^n}_{\L(\Ltt)} \leq C^n \norm{q}_{L^2_{3/2}}\norm{q}^{n-1}_{L^1_1}/(n-1)!$$ 
where $C>0$ depends only  on $a$.
\item The map $\K: L^2_{3/2} \to \L\left(\Ltt\right), \quad q \mapsto \K(q)$ is linear and bounded; the map
\begin{align*}
  \left(Id - \K \right)^{-1}: \,L^2_{3/2} & \to \L\left(\Ltt\right) \quad   q   \mapsto \left(Id- \K(q)\right)^{-1}
\end{align*}
is analytic and $\norm{\left(Id- \K\right)^{-1}}_{\L(L^2_{3/2}, \Ltt)}\leq C \left( 1+ \norm{q}_{L^2_{3/2}}e^{\norm{q}_{L^1_1}}\right).$
\end{enumerate}
\label{KinLtt}
\end{lemma}
\begin{proof}
Proceeding as in the proof of the previous lemma, one gets for $x \geq a$  the estimate
$$\norm{\K(q)[f] (x, \cdot)}_{L^2} \leq \intx{x}|t-x| |q(t)| \norm{f(t, \cdot)}_{L^2} \,dt \leq 
\Big( \intx{x} (t-x)^2 |q(t)|^2 \,dt \Big)^{1/2} \norm{f}_{\Ltt},$$
from which it follows that
\begin{align*}
 \norm{\K(q)[f]}_{\Ltt}^2 & \leq \norm{\intx{x} (t-x)^2 |q(t)|^2 \,dt}_{L^1_{x\geq a}}^{1/2} \norm{f}_{\Ltt} \leq C
 \norm{ q}_{L^2_{3/2}} \norm{f}_{\Ltt}
 \end{align*}
proving item $(i)$. To estimate the composition  $\K(q)^n$ viewed as an operator on  $\Ltt$, remark that
\begin{align*}
& \norm{\K(q)^n [f](x, \cdot)}_{L^2}  \leq  \int\limits_{x\leq t_1 \leq \ldots \leq t_n} |t_1-x| |q(t_1)| \cdots |t_n - t_{n-1}| |q(t_n)| \norm{f(t_n, \cdot)}_{L^2} dt \\
& \quad \leq \int\limits_{x\leq t_1 \leq \ldots \leq t_n}  |t_1-x| |q(t_1)| \cdots |t_{n-1}-t_{n-2}| 
|q(t_{n-1})| \Big( \intx{t_{n-1}}dt_n\, (t_n-t_{n-1})^2\, |q(t_n)|^2\Big)^{1/2}  \norm{f}_{\Ltt} dt \\
& \quad \leq \Big( \intx{x} (t-x)^2 |q(t)|^2 \,dt \Big)^{1/2} \norm{f}_{\Ltt} \Big( \intx{x} |t-x| |q(t)| \,dt \Big)^{n-1}/(n-1)! \ .
\end{align*}
Therefore
\begin{align*}
 \norm{\K(q)^n [f]}_{\Ltt} \leq & \norm{\intx{x} (t-x)^2 |q(t)|^2 \,dt}_{L^1_{x\geq a}}^{1/2}
 \norm{f}_{\Ltt} \frac{C^{n-1} \norm{q}^{n-1}_{L^1_1}}{(n-1)!}\\
  \leq & \norm{q}_{L^2_{3/2}} \norm{f}_{\Ltt} C^n \frac{\norm{q}^{n-1}_{L^1_1}}{(n-1)!} 
 \end{align*}
from which item $(i)$ follows. Item $(ii)$ is then proved as in the previous Lemma.
\end{proof}
Note that for $f \equiv 1$, the expression in \eqref{operK} of $\K(q)[f]$, $\K(q)[1](x,k) = \intx{x} D_k(t-x)\, q(t)\,  dt$  is well defined.
\begin{lemma}
For any $2 \leq \beta \leq +\infty$ and $a \in \R$, the map $L^2_2 \ni q \mapsto \K(q)[1] \in \Czp \cap \Ltt$ is analytic. Furthermore
$$\norm{\K(q)[1]}_{\Czp} \leq C_1 \norm{q}_{L^2_2}  
,\quad \norm{\K(q)[1]}_{\Ltt} \leq C_2 \norm{q}_{L^2_2},$$
where $C_1, C_2 >0$ are constants depending  on $a$ and $\beta$.
\label{K1}
\end{lemma}
\begin{proof} Since the map $q \mapsto \K(q)[1]$ is linear in $q$, it suffices to prove its  continuity in $q$. 
Moreover, it is enough to prove the result for $\beta = 2$ and $\beta = +\infty$ as  the general case then follows by interpolation. For any $k \in \R$, the bound $|D_k(y)|\leq |y|$ shows that the map $k \mapsto D_k(y)$ is in $L^\infty$. Thus
$$
\norm{\K(q)[1](x, \cdot)}_{L^\infty} \leq \intx{x}(t-x) |q(t)| dt \leq \intx{a} |t-a| |q(t)| \, dt \leq C \norm{q}_{L^1_1},
$$
where $C>0$ is a constant depending only on  $a \in \R$. The claimed estimate follows by noting that $\norm{q}_{L^1_1}\leq C \norm{q}_{L^2_2}$.
\newline
Using that for $|k| \geq 1$, $|D_k(y)| \leq \frac{1}{|k|}$, one sees that $k \mapsto D_k(y)$ is $L^2$-integrable. Hence $k \mapsto D_k(t-x) D_{-k}(s-x)$ is integrable.
Actually, since the Fourier transform $\F_+ (D_k (y))$ in the $k$-variable of the function $k \mapsto D_k(y)$ is the function 
$\eta \mapsto \mathbbm{1}_{[0, y]} (\eta) $,               
by Plancherel's Theorem 
$$\int^{\infty}_{-\infty} D_k(t-x) \overline{D_{k}(s-x)} \;dk = 
\frac{1}{\pi}\int^{\infty}_{-\infty} \mathbbm{1}_{[0, t-x]} (\eta) \mathbbm{1}_{[0, s-x]} (\eta) \; d\eta = \frac{1}{\pi}\min(t-x, s-x).$$
For any $x \geq a$ one thus has 
\begin{align*} 
 \norm{\K(q)[1](x, \cdot)}^2_{L^2} & = \int^{\infty}_{-\infty} \K(q)[1](x, \cdot) \cdot\overline{\K(q)[1]}(x, \cdot)\, dk \\
 & =
 \iint\limits_{[x, \infty)\times [x, \infty)} dt\,ds\; q(t)\, \overline{q(s)} \intii D_k(t-x) D_{-k}(s-x)\, dk \ .
 \end{align*}
and hence
\begin{equation}
 \norm{\K(q)[1](x, \cdot)}^2_{L^2} \leq  \frac{2}{\pi} \intx{x} (t-x) |q(t)| \intx{t} |q(s)| ds \leq \frac{2}{\pi} \intx{a} ds\, |q(s)| \int_a^s  |t-a|\, |q(t)| \, dt \leq C \norm{q}_{L^2_1}^2 \ ,
\label{eq:l1}
\end{equation}
where the last inequality follows from the Hardy-Littlewood inequality. The continuity in $x$ follows from Lebesgue convergence Theorem.
\newline
To prove the second inequality, start from the second term in \eqref{eq:l1} and change the order of integration to obtain
$$ \norm{\K(q)[1]}_{\Ltt}^2 \leq \norm{\intx{x} |t-a| |q(t)| \intx{t} |q(s)| ds}_{L^1_{x \geq a}} \leq  \intx{a} |q(s)| \int\limits_a^s (s-a)^2 |q(s)| ds \leq C \norm{q}_{L^2_1} \norm{q}_{L^2_2}.$$
\end{proof}

{\em Proof of Proposition \ref{prop_minLit}.} 
 Formally, the solution of equation \eqref{defm} is given by 
\begin{equation} 
 m(q)-1 = \Big(Id - \K(q) \Big)^{-1}\K(q)[1].
\label{defmK}
\end{equation}
By Lemma \ref{KinLit}, \ref{KinLtt}, \ref{K1} 
it follows that the r.h.s. of \eqref{defmK} is  an element of $\Czp \cap \Ltt$, $2 \leq \beta \leq \infty$, and analytic as a function of $q$, since it is the composition of analytic maps.
\qed
\vspace{1em}\\
{\em Properties of $\derk^n m(q,x,k)$ for $1\leq n \leq M-1$.} In order to study $\derk^n m(q,x,k)$, we deduce from \eqref{defm} an integral equation for $\derk^n m(q, x, \cdot)$   and solve it. Recall that for any $M \in \Z_{\geq 0}$, $H^M_\C \equiv H^M(\R, \C)$ denotes the Sobolev space of functions $\{ f \in L^2 \vert \ \hat{f} \in L^2_M \} $. The result is summarized in the following
\begin{proposition} Fix $M \in \Z_{\geq 4}$ and $a \in \R$. For any integer $1\leq n \leq M-1$  the following holds: 
\begin{enumerate}[(i)]
\item for $q \in L^2_{M}$ and $x \geq a$ fixed, the function $k \mapsto m(q, x, k)-1$ is in $H^{M-1}_\C$;
\item the map
$L^2_{M} \ni q  \mapsto \derk^n m(q) \in \Czt$
is analytic. Moreover  $\norm{\derk^n m(q)}_{\Czt}\leq K \norm{q}_{L^2_{M}},$ where  $K$ can be chosen  uniformly on bounded subsets of $ L^2_{M}$.
\end{enumerate}
\label{prop_derminLit}
\end{proposition}
\begin{remark}
In \cite{KaCo2} it is proved that if $q \in L^1_{M-1}$ then for every $x \geq a$ fixed the map $k \mapsto m(q,x,k)$ is in $C^{M-2}$; note that since $L^2_M \subset L^1_{M-1}$, we obtain the same regularity result by Sobolev embedding theorem.
\end{remark}
To prove Proposition \ref{prop_derminLit} we first need to derive some auxiliary results. 
Assuming that $m(q, x, \cdot)-1$ has appropriate regularity and decay properties, the $n^{th}$ derivative $\derk^n m(q, x, k)$ satisfies the following integral equation
\begin{equation}
 \derk^n m(q, x,k) = \sum_{j=0}^n \binom{n}{j} \intx{x} \derk^j D_k(t-x)\, q(t)\, \derk^{n-j} m(q, t,k)\; dt \ .
\label{defderkm}
\end{equation}
To write \eqref{defderkm} in a more convenient form introduce for $1 \leq j \leq n$ and $q \in L^2_{n+1}$  the operators
\begin{equation}
\K_j(q): \Czt \to \Czt, \quad f \mapsto \K_j(q)[f](x,k):= \intx{x} \derk^j D_k(t-x)\, q(t)\, f(t,k)\; dt
\label{defkj}
\end{equation}
leading to
\begin{equation}
\Big(Id - \K(q) \Big) \derk^n m(q) = \left( \sum_{j=1}^{n-1}\binom{n}{j} \K_j(q)[\derk^{n-j} m(q)] + \K_n(q)[m(q)-1] + \K_n(q)[1] \right).
\label{defmKj}
\end{equation}
In order to prove the claimed properties for $\derk^n m(q)$ we must show in particular that the r.h.s. of \eqref{defmKj} is  in $\Czt$. This is accomplished by the following
\begin{lemma}  Fix $M \in \Z_{\geq 4}$ and $a \in \R$. Then  there exists a constant $C>0$, depending only on  $a, M$,  such that the following holds: 
\begin{enumerate}[(i)]
\item for any integers $1 \leq n \leq M-1$ 
\begin{enumerate}
\item[(i1)]  the map
$L^2_{M}\ni q \mapsto \K_n(q)[1] \in \Czt$
is analytic, and $ \norm{\K_n(q)[1]}_{\Czt} \leq C \norm{q}_{L^2_{M}}$.
\item[(i2)]   the map
$ L^2_{M}\ni q \mapsto \K_n(q)\in \L\left(\Ltt, \Czt\right)$
is analytic. Moreover  
$$\norm{\K_n(q)[f]}_{\Czt} \leq \norm{q}_{L^2_{M}}\norm{f}_{\Ltt} \ .$$
\end{enumerate}
\item  For any $1\leq n \leq M-2$, the map
$ L^2_M \ni q \mapsto \K_n(q) \in \L\left(\Czt\right)$ is analytic. Moreover one has $ \norm{\K_n(q)[f]}_{\Czt}\leq C \norm{q}_{L^2_{M}}\norm{f}_{\Czt}.$
\item  As an application of item $(i)$ and $(ii)$, for any integers $1 \leq n \leq M-1$ the map $ L^2_{M}\ni q  \mapsto \K_n(q)[m(q)-1] \in \Czt$
is analytic, and
$$\norm{\K_n(q)[m(q)-1]}_{\Czt} \leq K_0' \norm{q}_{L^2_{M}}^2 \ ,$$ where $K_0'>0$ can be chosen uniformly on bounded subsets of $L^2_{M}$.
\end{enumerate} 
\label{Kn1}
\end{lemma}
\begin{proof}First, remark that  all the  operators $q \mapsto \K_n(q)$ are linear in $q$, therefore the continuity in $q$ implies the analyticity in $q$. We begin proving item $(i)$.
\begin{enumerate}
\item[$(i1)$]  Let $ \varphi (x,k):= \intx{x} \derk^n D_k(t-x)\, q(t) \; dt$ 
and compute the Fourier transform  $\F_+(\varphi(x, \cdot))$ with respect to the $k$ variable for $x\geq a$ fixed, which we denote by $ \hat{\varphi}(x,\xi) \equiv \int_{-\infty}^{\infty} dk \, e^{ik\xi}\varphi(x,k)$. Explicitly 
$$ \hat{\varphi} (x,\xi)= \intx{x} dt\,  q(t)  \int\limits_{-\infty}^{+\infty} dk \, e^{2ik\xi}\, \derk^n D_k(t-x) = 
\intx{x}  q(t) \,\xi^n \, \mathbbm{1}_{[0, t-x]}(\xi)\, dt.$$
By Parseval's Theorem $\norm{\varphi(x,\cdot)}_{L^2}=\frac{1}{\sqrt{\pi}}\norm{\hat{\varphi}(x,\cdot)}_{L^2}$. By changing the order of integration one has  
\begin{align*}
 \norm{\hat{\varphi}(x,\cdot)}^2_{L^2}& = 
\intii \hat{\varphi}(x,\xi)\, \overline{\hat{\varphi}(x,\xi)} \; d\xi
 = \iint\limits_{[x,\infty) \times [x,\infty)} dt\, ds\; q(t)\, \overline{q(s)} \intii |\xi|^{2n}\, \mathbbm{1}_{[0, t-x]}(\xi)\, \mathbbm{1}_{[0, s-x]}(\xi) d\xi\leq \\
& \leq 2 \intx{x} dt\; |q(t)|\, |t-x|^{2n+1} \intx{t} |q(s)| \; ds \leq 
\norm{(t-a)^{n+1} q}_{L^2_{t\geq a}}\norm{(t-a)^n \intx{t} |q(s)| ds}_{L^2_{t\geq a}} \\ 
& \leq C \norm{q}_{L^2_{n+1}}^2,
\end{align*}
where we used that by $(A3)$ in Appendix \ref{techLemma}, $\norm{(t-a)^n \intx{t} |q(s)| \, ds}_{L^2_{t \geq a}} \leq C \norm{q}_{L^2_{n+1}}.$
\item[$(i2)$]  Let $f\in \Ltt$, and using
 $\mmod{\derk^n D_k(t-x)} \leq 2^{n}|t-x|^{n+1}$ it follows that
\begin{align*}
 \norm{\K_n(q)[f](x,\cdot)}_{L^2} \leq C \intx{x} |q(t)| \,|t-x|^{n+1}\, \norm{f(t, \cdot)}_{L^2}\; dt \leq C \norm{q}_{L^2_{n+1}} \norm{f}_{\Ltt};
\end{align*}
by taking the supremum in the $x$ variable one has  $\K_n(q) \in \L\left(\Ltt, \Czt\right)$, where the continuity in $x$ follows by Lebesgue's convergence theorem.
The map $q \mapsto \K_n(q)$ is  linear and continuous, therefore also analytic.
\end{enumerate}
We prove now item $(ii)$.
 Let $g \in \Czt$. From $\norm{\K_n(q)[g](x,\cdot)}_{L^2}  \leq \intx{x} |q(t)|\, |t-x|^{n+1}\, \norm{g(t, \cdot)}_{L^2}\; dt $ it follows that  
\begin{align*}
\sup_{x \geq a} \norm{\K_n(q)[g](x, \cdot)}_{L^2} & 
\leq \norm{g}_{\Czt} \intx{a} |q(t)|\, |t-a|^{n+1}\; dt \leq C \norm{g}_{\Czt} \norm{q}_{L^2_{n+2}} \ ,
\end{align*}
which implies the claimed estimate. The analyticity follows from the linearity and continuity of the map $q \mapsto \K_n(q)$.

Finally we prove item $(iii)$. By Proposition \ref{prop_minLit}, the map $L^2_{n+1} \ni q\mapsto m(q)-1 \in \Ltt$ is analytic. 
By item $(i2)$ above the bilinear map $L^2_{n+1}\times \Ltt \ni (q, f)\mapsto K_n(q)[f] \in \Czt$ is analytic;  since the composition of analytic 
maps is analytic, the map $L^2_{n+1} \ni q \mapsto K_n(q)[m(q)-1] \in \Czt$ is analytic.  By $(i2)$ and Proposition \ref{prop_minLit} one has
$$
\norm{\K_n(q)[m(q)-1]}_{\Czt} \leq C \norm{q}_{L^2_{n+1}} \norm{m(q) - 1}_{\Ltt} \leq K_0' \norm{q}_{L^2_{n+1}}^2
$$
where $K_0'$ can be chosen uniformly on bounded subsets of $L^2_{M}$.
\end{proof}

{\em Proof of Proposition \eqref{prop_derminLit}.} The proof is carried out by a recursive argument in $n$.
We assume that $q \mapsto \derk^r m(q)$ is analytic as a map from $L^2_M$ to $\Czt$ for $0 \leq r \leq n-1$, and  prove that  $L^2_M \to \Czt : $ $q \mapsto \derk^{n} m(q)$ is analytic, provided that $n \leq M-1$. The case $n=0$ is proved in Proposition \ref{prop_minLit}. 
\newline
We begin by showing that for every $x \geq a$ fixed  $k \mapsto \derk^{n-1} m(q, x,k)$ is a function in $H^1$, therefore it has one more (weak) derivative in the $k$-variable.
We use the following characterization of $H^1$ function \cite{brezis}:  
\begin{equation}
\label{H1_char}
f \in H^1 \mbox{ iff
there exists a constant } C >0 \mbox{ such that } \norm{\tau_h f - f }_{L^2} \leq C |h|, \quad  \forall h \in \R,
\end{equation}
where $(\tau_h f)(k) := f(k+h)$ is the translation operator. Moreover the constant $C$ above can be chosen to be $C = \norm{\derk u}_{L^2}$. 
Starting from \eqref{defmKj} (with $n-1$ instead of $n$), an easy computation  shows that for every $x\geq a$ fixed  $(\tau_h) \derk^{n-1} m(q) \equiv \derk^{n-1} m(q, x, k+h)$ satisfies  the integral equation
\begin{equation}
\label{2.10bis} 
\begin{aligned}
\left( Id - \K(q) \right)&(\tau_h \derk^{n-1} m(q) - \derk^{n-1} m(q) ) \\
 = &  \int_x^{+\infty}
(\tau_h \derk^{n-1} D_k(t-x) - \derk^{n-1} D_k(t-x)) q(t) ( m(q, t, k+h)-1) \, dt  \\
& + \int_x^{+\infty}
(\tau_h \derk^{n-1} D_k(t-x) - \derk^{n-1} D_k(t-x)) q(t)  \, dt \\
& + \int_x^{+\infty} (\derk^{n-1} D_k(t-x)) \, q(t) \, \left(m(q,t, k+h) - m(q,t,k) \right) \, dt \\
& +\sum_{j=1}^{n-2} \binom{n-1}{j} \Big(  \int_x^{+\infty}
(\tau_h \derk^j D_k(t-x) - \derk^j D_k(t-x)) q(t) \derk^{n-1-j} m(q, t, k+h) \, dt \\
&  + \int_x^{+\infty} \derk^j D_k(t-x) \,q(t)\, (\tau_h \derk^{n-1-j} m(q,t,k) - \derk^{n-1-j} m(q,t,k)) \, dt \Big) \\
& + \int_x^{+\infty} (\tau_h D_k(t-x)- D_k(t-x)) \, q(t) \, \derk^{n-1} m(q,t, k+h) \, dt.
\end{aligned}
\end{equation}
In order to estimate the term in the fourth line on the right hand side of the latter identity, use item $(i1)$ of Lemma \ref{Kn1} and the characterization  \eqref{H1_char} of $H^1$. To estimate all the remaining lines, use the induction  hypothesis,  the estimates of Lemma \ref{Kn1},  the fact that the operator norm  of  $(Id - \K(q))^{-1}$ is bounded uniformly in $k$ and the estimate 
$$
\mmod{\tau_h \derk^j D_k(t-x) - \derk^j D_k(t-x)} \leq C |t-x|^{j+2} \, |h|, \quad \forall h \in \R,
$$
 to deduce that for every $n \leq M-1$
$$
\norm{\tau_h \derk^{n-1} m(q) - \derk^{n-1} m(q)}_{L^2} \leq C |h|, \quad \forall h \in \R,
$$ which is
exactly condition \eqref{H1_char}.
This shows that $k \mapsto \derk^{n-1} m(q,x,k)$ admits a weak derivative in $L^2$. Formula \eqref{defderkm} is therefore justified.
We  prove now that the map $ L^2_{M}\ni q \mapsto \derk^{n} m(q)\in \Czt $ is analytic for $1 \leq n \leq M-1$. Indeed equation \eqref{defmKj} and Lemma \ref{Kn1} imply that
$$ \norm{\derk^{n} m(q)}_{\Czt} \leq K' \Big( \norm{q}_{L^2_{M}}+ \norm{q}_{L^2_{M}}^2 +  \sum_{j=1}^{n-1}\norm{q}_{L^2_{M}}
\norm{\derk^{n-j}m(q)}_{\Czt} \Big)$$
where $K'$ can be chosen uniformly on bounded subsets of $q$ in $L^2_M$. Therefore  $\derk^n m(q) \in \Czt$ and one gets recursively $\norm{\derk^n m(q)}_{\Czt}\leq K \norm{q}_{L^2_{M}} $, where $K$ can be chosen uniformly on bounded subsets of $q$ in $L^2_M$.   The analyticity of the map $q \mapsto \derk^{n} m(q)$ follows by formula \eqref{defmKj} and the fact that composition of analytic maps is analytic.
\qed
\newline \vspace{1em} \newline
{\em Properties of  $k\derk^n m(q,x,k)$ for $1\leq n \leq M$.} The analysis of the $M^{th}$ $k$-derivative of $m(q,x,k)$ requires a separate attention. It turns out that the distributional derivative $\derk^M m(q,x,\cdot)$ is not necessarily $L^2$-integrable near $k=0$ but the product  $k\derk^M m(q,x,\cdot)$ is. This is  due to the fact that $\derk^M D_k(x)q(x) \sim x^{M+1}q(x)$ which might not be $L^2$-integrable. However, by integration by parts, it's easy to see that $k \derk^M D_k(x)q(x) \sim x^{M}q(x) \in L^2$. The main result of this section is the following
\begin{proposition} Fix $M \in  \Z_{\geq 4}$ and $a \in \R$. Then for every integer $1 \leq n \leq M$ the following holds:
\begin{enumerate}[(i)]
\item for every $q \in L^2_M$ and $x \geq a$ fixed, the function $k \mapsto k \derk^n m(q,x,k)$ is in $L^2$;
\item the map
$ L^2_{M}\ni q \mapsto k\derk^n m(q) \in \Czt$
is analytic. 
Moreover 
  $\norm{k \derk^n m}_{\Czt}\leq K_1 \norm{q}_{L^2_{M}}$ where   $K_1$ can be chosen  uniformly on bounded subsets of $L^2_M$.
  \end{enumerate}
\label{prop_dermiNLit}
\end{proposition}

Formally, multiplying equation \eqref{defderkm} by $k$, the function  $k\derk^n m(q)$ solves \begin{equation}
\left(Id - \K(q) \right)( k\derk^n m(q))=  \left( \sum_{j=1}^{n-1}
\binom{n}{j} \tilde{\K}_j(q)[\derk^{n-j} m(q)] + \tilde{\K}_n(q)[m(q)-1] + 
\tilde{\K}_n(q)[1] \right)
\label{kderk^Nm_formula}
\end{equation}
where we have introduced for  $0 \leq j \leq M$  and $q \in L^2_M$ the operators
\begin{equation}
\label{Ktilde.def} 
\tilde{\K}_j(q): \Czt \to \Czt, \quad f \mapsto \tilde{\K}_j(q)[f](x,k):=  \intx{x} k \derk^j D_k(t-x)\, q(t)\, f(t,k)\; dt.
\end{equation}
We begin by proving  that each term of the r.h.s. of \eqref{Ktilde.def} is well defined and analytic as a function of $q$. The following lemma is analogous to Lemma \ref{Kn1}:
\begin{lemma}
\label{lem:derk^Nm(q)}
 Fix $M \in \Z_{\geq 4}$ and $a \in \R$.  There exists a constant $C>0$ such that  the following holds: 
\begin{enumerate}[(i)]
\item for any integers $1 \leq n \leq M$
\begin{enumerate}
\item[(i1)] the map $ L^2_{M}\ni q \mapsto \tilde{\K}_n(q)[1]\in \Czt$
is analytic, and $\norm{\tilde{\K}_n(q)[1] }_{\Czt} \leq C \norm{q}_{L^2_M}$;
 \item[(i2)] the map
 $L^2_{M} \ni q  \mapsto \tilde{\K}_n(q)\in  \L\left(\Ltt, \, \Czt\right)$  is  analytic. Moreover 
 $$\norm{\tilde{\K}_n(q)[f]}_{\Czt} \leq C \norm{q}_{L^2_{M}}\norm{f}_{\Ltt} \ ;$$
 \end{enumerate}
 \item for any $ 1 \leq j \leq M-1$ the map  $L^2_{M}\ni q  \mapsto \tilde{\K}_j(q)\in \L\left(\Czt \right)$ is analytic, and 
$$\norm{\tilde{\K}_j(q)[f]}_{\Czt} \leq C \norm{q}_{L^2_{M}} \norm{f}_{\Czt} \ .$$
\item As an application of item $(i)$ and $(ii)$ we get
\begin{enumerate}
\item[$(iii1)$] for any $1 \leq n \leq M$, the map $L^2_M \ni q \mapsto \tilde{\K}_n(q)[m(q)-1] \in \Czt$ is analytic with \begin{equation}
\norm{\tilde{\K}_n(q)[m(q)-1]}_{\Czt}  \leq K_1' \norm{q}_{L^2_M}^2, 
\label{eq:derk^Nm(q)}
\end{equation}
 where $K_1'$ can be chosen uniformly on bounded subsets of $L^2_M$;
 \item[$(iii2)$] for any $1 \leq j \leq n-1$, the map $L^2_M \ni q \mapsto \tilde{\K}_j(q)[\derk^{n-j}m(q)] \in \Czt $ is analytic with 
\begin{equation}
 \norm{\tilde{\K}_j(q)[\derk^{n-j}m(q)]}_{\Czt} \leq K_2' \norm{q}_{L^2_M}^2, 
\label{eq:derk^Nm(q).2}
\end{equation}
 where $K_2'$ can be chosen uniformly on bounded subsets of $L^2_M$.
\end{enumerate}
\end{enumerate} 
\end{lemma}
\begin{proof} 
\begin{enumerate}[$(i)$]
\item Since the maps $q \mapsto \tilde{\K}_n(q)$, $0\leq n \leq M$, are linear, it is enough to prove that these maps  are continuous.
\begin{enumerate}
\item[$(i1)$] Introduce 
$ \varphi (x,k):= \intx{x} k\derk^n D_k(t-x)\, q(t)  \; dt$.  The Fourier transform $\F_+(\varphi(x, \cdot))$ of $\varphi$ with respect to the $k$-variable is given by $\F_+(\varphi(x, \cdot))\equiv \hat{\varphi} (x,\xi)$, where 
$$\hat{\varphi} (x,\xi)= \intx{x} dt\,  q(t)  \int\limits_{-\infty}^{+\infty} dk\; e^{-2ik\xi}\, k\derk^n D_k(t-x) = -(2i)^{n-1}
\intx{x} dt\;  q(t)\,\partial_\xi( \xi^n\, \mathbbm{1}_{[0, t-x]}(\xi)),$$
where $\partial_\xi (\xi^n \mathbbm{1}_{[0, t-x]}(\xi))$ is to be understood  in the distributional sense. 
By Parseval's Theorem 
$\norm{\varphi(x,\cdot)}_{L^2}=\frac{1}{\sqrt{\pi}} \norm{\hat{\varphi}(x,\cdot)}_{L^2}.$
Let $C^\infty_0$ be the space of  smooth, compactly supported functions. Since
$$\norm{\hat{\varphi}(x,\cdot)}_{L^2_\xi}=\sup_{\substack{\chi \in  C^{\infty}_0 \\ \norm{\chi}_{L^2}\leq 1}}\mmod{\int\limits_{-\infty}^{\infty} \chi(\xi)\,\hat{\varphi}(x,\xi)  \;d\xi}, $$
one computes
\begin{align*}
\mmod{\int\limits_{-\infty}^{\infty} \chi(\xi)\, \hat{\varphi}(x,\xi)\; d\xi} &=  \mmod{\intx{x} dt \;  q(t) \int\limits_{-\infty}^{\infty} \chi(\xi)\,\partial_\xi \left( \xi^n \mathbbm{1}_{[0, t-x]}(\xi)\right)\; d\xi} =  \mmod{\intx{x} dt\;  q(t) \int\limits_0^{t-x}d\xi\; \xi^n  \partial_\xi \chi(\xi) } \\
& \leq \mmod{\intx{x} dt \;  q(t) \chi(t-x) (t-x)^n} + n \mmod{\intx{x} dt\;  q(t) \int\limits_0^{t-x}d\xi \, \chi(\xi) \xi^{n-1}} \\
&\leq \norm{q}_{L^2_M} \norm{\chi}_{L^2} + n \mmod{\intx{x} dt\;  |q(t)| |t-x|^{n-1} \int\limits_0^{t-x}d\xi \, |\chi(\xi)|} \\
& \leq \norm{q}_{L^2_M} \norm{\chi}_{L^2} + n \mmod{\intx{x} dt\;  |q(t)| |t-x|^{n} \frac{\int\limits_0^{t-x}d\xi \, |\chi(\xi)|}{|t-x|}} \leq C \norm{q}_{L^2_M} \norm{\chi}_{L^2}
\end{align*}
where the last inequality follows from Cauchy-Schwartz and  Hardy inequality, and $C >0$ is a constant depending  on $a$ and $M$.
\item[$(i2)$] As $\mmod{k\derk^n D_k(t-x)} \leq 2^n |t-x|^{n}$ by integration by parts, it follows that for some constant $C>0$ depending only on $a$ and $M$, 
\begin{align*}
 \norm{\tilde{\K}_n(q)[f](x,\cdot)}_{L^2} & \leq C \intx{x} |t-x|^n \,|q(t)|\, \norm{f(t,\cdot)}_{L^2}\; dt 
\leq C \norm{q}_{L^2_{M}} \norm{f}_{\Ltt} \ .
\end{align*}
Now take the supremum over $x\geq a$ in  the expression above and use Lebesgue's dominated convergence theorem to prove item $(i2)$.
\end{enumerate}
\item The claim follows by the estimate  
\begin{align*}
 \norm{\tilde{\K}_j(q)[f](x,\cdot)}_{L^2} & \leq C \intx{x} |t-x|^j \,|q(t)|\, \norm{f(t,\cdot)}_{L^2}\; dt 
\leq C \norm{q}_{L^1_{j}} \norm{f}_{\Czt}
\end{align*}
 and the remark  that $\norm{q}_{L^1_{j}} \leq C \norm{q}_{L^2_M}$ for $0 \leq j\leq M-1$.
\item  By Propositions \ref{prop_minLit} and \ref{prop_derminLit} the maps $L^2_M \ni q \mapsto m(q)-1 \in \Czt \cap \Ltt$ and $L^2_M \ni q \mapsto \derk^{n-j}m(q) \in \Czt$ are analytic;  by item $(ii)$ for any $1 \leq n \leq M-1$, the bilinear map $(q, f) \mapsto \tilde{\K}_n(q)[f] $ is analytic from $L^2_M \times \Czt$ to $\Czt$. Since the composition 
of two analytic maps is again analytic, item $(iii)$ follows. Moreover $\tilde{\K}_n(q)[m(q)-1]$,  $\tilde{\K}_j(q)[\derk^{n-j} m(q)] \in \Czt$ since  $m(q,x,k)$ and $\derk^n m(q,x,k)$ are continuous in the $x$-variable. 
The estimate \eqref{eq:derk^Nm(q)} follows from  item $(ii)$ and Proposition \ref{prop_minLit}, \ref{prop_derminLit}.
\end{enumerate}
\end{proof}

{\em Proof of Proposition \ref{prop_dermiNLit}.} One proceeds in the same way as in the proof of Proposition \ref{prop_derminLit}. Given any $1 \leq n \leq M$, we assume that $q \mapsto k \derk^r m(q)$ is analytic as a map from $L^2_M$ to $\Czt$ for $1 \leq r \leq n-1$, and deduce that   $q \mapsto k\derk^{n} m(q)$ is analytic as a map from $L^2_M$ to $\Czt$ and satisfies equation \eqref{kderk^Nm_formula} (with $r$ instead of $n$). 
\newline
We begin by showing that for every $x \geq a$ fixed,  $k \mapsto k\derk^{n-1} m(q, x,k)$ is a function in $H^1$. Our argument uses again the characterization \eqref{H1_char} of $H^1$. Arguing as for the derivation of \eqref{2.10bis} one gets  the integral equation
\begin{align*}
\left( Id - \K(q) \right)&(\tau_h (k\derk^{n-1} m(q)) - k\derk^{n-1} m(q) ) = \\ 
& = \intx{x}
\left(\tau_h (k \derk^{n-1} D_k(t-x)) - k \derk^{n-1} D_k(t-x)\right) q(t) ( m(q, t, k+h)-1) \, dt  \\
&+\intx{x}
\left(\tau_h (k\derk^{n-1} D_k(t-x)) - k \derk^{n-1} D_k(t-x)\right) q(t)  \, dt \\
& + \intx{x} (k \derk^{n-1} D_k(t-x) ) q(t) \left( m(q,t,k+h) - m(q, t, k) \right) \, dt \\
&+\sum_{j=1}^{n-2} \binom{n-1}{j} \Big( \intx{x}
\left(\tau_h (k\derk^j D_k(t-x)) - k\derk^j D_k(t-x)\right) q(t)\, \derk^{n-1-j} m(q, t, k+h) \, dt \\
& + \intx{x} k \derk^j D_k(t-x) \,q(t)\, \left(\tau_h \derk^{n-1-j} m(q,t,k) - \derk^{n-1-j} m(q,t,k)\right) \, dt \Big) \\
& + \intx{x} \left(\tau_h D_k(t-x)- D_k(t-x)\right) \, q(t) \, (k+h)\derk^{n-1} m(q,t, k+h) \, dt \ .
\end{align*}
Using the estimates 
$$
\mmod{\tau_h D_k(t-x) - D_k(t-x)} \leq C |t-x|^2 |h| 
$$
 and
 $$ \mmod{\tau_h (k\derk^j D_k(t-x)) - k\derk^j D_k(t-x)} \leq C |t-x|^{j+1} \, |h|, \quad \forall h \in \R \ ,
$$
obtained by integration by parts, the characterization  \eqref{H1_char} of $H^1$, the inductive hypothesis, estimates of Lemma \ref{Kn1} and Lemma \ref{KinLtt} one deduces that for every $n \leq M$
$$
\norm{\tau_h (k\derk^{n-1} m(q)) - k\derk^{n-1} m(q)}_{L^2} \leq C |h|, \quad \forall h \in \R.
$$
This shows that $k \mapsto k\derk^{n-1} m(q,x,k)$ admits a weak derivative in $L^2$. Since 
$$k\derk^n m(q,x,k) = \derk (k\derk^{n-1}m(q,x,k)) - \derk^{n-1} m(q,x,k) \ , $$
 the estimate above and Proposition \ref{prop_derminLit} show that $k \mapsto k\derk^n m(q,x,k)$ is an $L^2$ function.
Formula \eqref{defderkm} is therefore justified.
\newline
The proof of the analyticity of the map $q \mapsto  k\derk^n m(q)$ is analogous to the one of Proposition \ref{prop_derminLit} and it is omitted.
\qed
\vspace{1em}

{\em Analysis of $\derx m(q,x,k)$.}
Introduce a odd smooth monotone function $\zeta: \R \to \R$  with $\zeta(k)=k$ for $|k|\leq 1/2$ and $\zeta(k)=1$ for $k\geq 1$. We prove the  following
\begin{proposition}
Fix $M \in \Z_{\geq 4}$ and $a \in \R$. Then  the following holds:
\begin{enumerate}[(i)]
\item for any integer $0 \leq n \leq M-1$, the map $ L^2_M \ni q \mapsto \derk^n \derx  m(q) \in \Czt$ is analytic, and 
$\norm{ \derk^n \derx m(q)}_{\Czt} \leq K_2 \norm{q}_{L^2_{M}}$ where $K_2$ can be chosen uniformly on bounded subsets of $L^2_M$.
\item the map $L^2_M \ni q \mapsto \zeta \derk^M \derx  m(q) \in \Czt$ is analytic, and
$\norm{ \zeta \derk^M \derx m(q)}_{\Czt} \leq K_3 \norm{q}_{L^2_{M}}$ where $K_3$ can be chosen uniformly on bounded subsets of $L^2_M$.
\end{enumerate}
 \label{prop:derxderkm}
\end{proposition}
The integral equation for $\derx m(q,x,k)$ is obtained by taking the derivative in the $x$-variable of \eqref{defm}:
\begin{equation}
 \derx m(q, x,k) = -\intx{x} e^{2ik(t-x)}\, q(t)\, m(q, t,k)\; dt.
 \label{derxm.equation}
\end{equation}
Taking the derivative with respect to the $k$-variable one obtains, for $0 \leq n \leq M-1$, 
\begin{equation}
 \derk^n \derx m(q, x,k) = -\sum_{j=0}^n \binom{n}{j} \intx{x}  e^{2ik(t-x)} \,(2i(t-x))^j \, q(t)\, \derk^{n-j} m(q, t,k) \;dt.
 \label{eq:derkderxm}
\end{equation}
For $0 \leq j  \leq M$ introduce the integral operators
\begin{equation}
\G_j(q):\Czt \to \Czt, \quad q \mapsto \G_j(q)[f](x,k) := - \intx{x} e^{2ik(t-x)}\, (2i(t-x))^j\, q(t)\, f(t,k)\; dt 
\end{equation}
and rewrite \eqref{eq:derkderxm}  in the more compact form
\begin{equation}
\label{2.17bis} 
\derk^n \derx m(q) = \sum_{j=0}^{n-1} \binom{n}{j} \G_j(q) [\derk^{n-j} m(q)] + \G_n(q)[m(q)-1] + \G_n(q)[1].
\end{equation}
Proposition \ref{prop:derxderkm}  $(i)$ follows from  Lemma \ref{lem:derxderkm} below.

The  $M^{th}$ derivative requires a separate treatment, as $\derk^M m$ might not be well defined at $k=0$. Indeed for $n=M$  the  integral   $ \intx{x}e^{2ik(t-x)}\, q(t)\,  \derk^M m(q,t,k)\; dt$ in \eqref{eq:derkderxm} might not be well defined near $k=0$ since we only know that  $k\derk^M m(q, x, \cdot)\in L^2$. To deal with this issue we use the function $\zeta$ described above. Multiplying \eqref{2.17bis} with $n=M$ by $\zeta$ we formally obtain
\begin{align*}
\zeta \derk^M \derx m(q)& =  \sum_{j=1}^{M-1}\binom{M}{j} \zeta \, \G_j(q)[\derk^{M-j} m(q)] + \zeta\, \G_M(q)[m(q)-1]+ \zeta\, \G_M(q)[1] + \G_0(q)[\zeta \derk^{M} m(q)].
\end{align*}
Proposition \ref{prop:derxderkm}  $(ii)$ follows from item $(iii)$ of  Lemma \ref{lem:derxderkm} and the fact that $\zeta \in L^\infty$: 
\begin{lemma}
\label{lem:derxderkm}
Fix $M \in \Z_{\geq 4}$ and $a \in \R$. There exists a constant $C>0$ such that 
\begin{enumerate}[(i)]
\item for any integer $0\leq n \leq M$ the following holds: 
\begin{enumerate}
\item[(i1)] the map $L^2_M \ni q \mapsto \G_n(q)[1] \in \Czt $ is analytic. Moreover
 $\norm{\G_n(q)[1]}_{\Czt}\leq C \norm{q}_{L^2_{M}}.$
 \item[(i2)] The map
$L^2_{M}\ni q \mapsto  \G_n(q)\in \L\left(\Ltt, \Czt\right)$ is analytic and $$\norm{\G_n(q)[f]}_{\Czt} \leq C \norm{q}_{L^2_{M}}\norm{f}_{\Ltt} \ .$$
\end{enumerate} 
\item For any  $0 \leq j \leq M-1$, the map $L^2_{M}\ni q \mapsto \G_j(q)\in \L\left(\Czt \right)$  is  analytic, and 
$$\norm{\G_j(q)[f]}_{\Czt} \leq C \norm{q}_{L^2_{M}} \norm{f}_{\Czt} \ .$$
\item For any $1 \leq n  \leq M-1$, $0 \leq j \leq n-1$ and $\zeta: \R \to \R$ odd smooth monotone function   with $\zeta(k)=k$ for $|k|\leq 1/2$ and  $\zeta(k)=1$ for $k\geq 1$, the following holds:
\begin{enumerate}
\item[(iii1)]  the maps $L^2_M \ni q \rightarrow \G_j(q)[\derk^{n-j}m(q)] \in \Czt $
and $L^2_M \ni q \rightarrow \G_n(q)[m(q)-1] \in \Czt$ are analytic.  Moreover 
$$
\norm{\G_j(q)[\derk^{n-j}m(q)]}_{\Czt}, \quad \norm{\G_n(q)[m(q)-1]}_{\Czt}   \leq K_2' \norm{q}_{L^2_M}^2,
$$
where $K_2'$ can be chosen  uniformly on bounded subsets of $L^2_M$.
\item[(iii2)] The map $L^2_M \ni q \rightarrow \G_0(q)[\zeta \derk^M m(q)] \in \Czt$ is analytic and $\norm{\G_0(q)[\zeta \derk^M m(q)]}_{\Czt} \leq K_3' \norm{q}_{L^2_M}^2$
where $K_3'$ can be chosen  uniformly on bounded subsets of $L^2_M$.
\end{enumerate}
\end{enumerate}
\end{lemma}

\begin{proof} As before it's enough to prove the continuity in $q$ of the  maps considered to conclude that they are analytic.
\begin{enumerate}
\item[$(i1)$] For $x \geq a$ and any $0 \leq n \leq M$ one has  
$\norm{\G_n(q)[1](x, \cdot)}_{L^2}^2 \leq C \intx{x} |t-x|^{2n} |q(t)|^2 dt  \leq C \norm{q}_{L^2_M}^2.$ The claim follows by taking the supremum over $x \geq a$ in the  inequality above.
\item[$(i2)$] For $x \geq a $ and $0 \leq n \leq M$ one has the bound $\norm{\G_n(q)[f]}_{\Czt} \leq C \norm{q}_{L^2_{n}} \norm{f}_{\Ltt}$, which implies the claimed estimate.
\item[$(ii)$] For $x \geq a$ and $0 \leq j \leq M-1$ one has the bound 
$$\norm{\G_j(q)[f]}_{\Czt} \leq C \norm{q}_{L^1_{M-1}} \norm{f}_{\Czt} \leq C \norm{q}_{L^2_{M}} \norm{f}_{\Czt} \ .$$
\item[$(iii1)$] By Proposition \ref{prop_derminLit} one has that for any $1 \leq n \leq M-1$ and $0 \leq j \leq n-1$ the map  $L^2_M \ni q \mapsto \derk^{n-j} m(q) \in \Czt$ is analytic. Since composition of analytic maps is again an analytic map, the claim regarding the analyticity follows.  The first estimate follows from item $(ii)$. A similar  argument can be used to prove the second estimate.
\item[$(iii2)$] By Proposition \ref{prop_dermiNLit},  the map $L^2_M \ni q \mapsto \zeta \derk^M m(q) \in \Czt$ is analytic, implying the claim regarding the analyticity. The estimate follows from 
$ \norm{\G_0[\zeta \derk^{M} m(q)]}_{\Czt} \leq \norm{q}_{L^2_M} \norm{\zeta\derk^M m(q)}_{\Czt}.$
\end{enumerate}
\end{proof}

The following corollary follows from the  results obtained so far:
\begin{cor}
\label{m(q,0,k)} Fix $M \in \Z_{\geq 4}$. Then the normalized Jost functions $m_j(q,x,k)$, $j=1,2$,  satisfy:
\begin{enumerate}[(i)]
\item the maps  $L^2_{M} \ni q \mapsto m_j(q,0, \cdot) -1 \in L^2$ and $L^2_{M} \ni q \mapsto k^\alpha \derk^n m_j(q,0,\cdot) \in L^2 $ are  analytic 
for $1 \leq n \leq M-1 $  $[1 \leq n \leq M]$ if $\alpha=0$ $[\alpha = 1]$. Moreover
$$
\norm{m_j(q,0, \cdot) -1}_{L^2}, \, 
\norm{k^\alpha \derk^n m_j(q,0,\cdot)}_{L^2}  \leq K_1 \norm{q}_{L^2_{M}},
$$
where $K_1 >0$ can be chosen uniformly on bounded subsets of $L^2_{M}$.
\item For $0 \leq n \leq M-1$, the maps $L^2_{M} \ni q \mapsto \derk^n \derx m_j(q,0, \cdot) \in L^2$ and $L^2_{M} \ni q \mapsto \zeta \derk^M \derx m_j(q,0, \cdot) \in L^2$ are analytic. Moreover
$$
\norm{\derk^n \derx m_j(q,0, \cdot)}_{L^2}, \, 
\norm{\zeta \derk^M \derx m_j(q,0, \cdot)}_{L^2}  \leq K_2 \norm{q}_{L^2_{M}},
$$
where $K_2 >0$ can be chosen uniformly on bounded subsets of $L^2_{M}$.
\end{enumerate}
\end{cor}
\begin{proof} The Corollary follows by evaluating  formulas \eqref{defm}, \eqref{defderkm}, \eqref{eq:derkderxm} at $x=0$ and using the results of Proposition \ref{prop_minLit}, \ref{prop_derminLit}, \ref{prop_dermiNLit} and \ref{prop:derxderkm}.
\end{proof}


\section{One smoothing properties of the scattering map.}
\label{sec:dir.scat}
The aim of this section is to prove the part of Theorem \ref{reflthm} related to the direct problem. To begin, note that  by Theorem \ref{deift_jost}, for $q \in L^2_4$ real valued one has $\overline{m_1(q,x,k)}= m_1(q,x,-k)$ and $\overline{m_2(q,x,k)}= m_2(q,x,-k)$; hence
\begin{equation}
\label{S.conj}
\overline{S(q,k)} = S(q,-k) \ , \qquad \overline{W(q,k)}= W(q, -k) \ .
\end{equation}
Moreover one has for any $q \in L^2_{4}$
\begin{equation}
\label{W&S}
W(q,k) W(q,-k) = 4k^2 + S(q,k) S(q,-k) \qquad \forall \, k \in \R\setminus \{0 \}
\end{equation}
 which by continuity holds for $k=0$ as well. In the case where $q \in \class$, the latter identity  implies that $S(q,0) \neq 0$. \\ 
Recall that for $q \in L^2_{4}$ the Jost solutions $f_1(q,x,k)$ and $f_2(q,x,k)$ satisfy the following integral equations
\begin{align}
\label{duhamelformula}
& f_1(x,k)= e^{i k x} + \intx{x}
\frac{\sin k(t-x)}{k}q(t)f_1(t,k)dt \ , \\
\label{duhamelformula2}
& f_2(x,k) = e^{-ikx} + \int\limits_{-\infty}^x
\frac{\sin k(x-t)}{k}q(t)f_2(t,k)dt \ .
\end{align}
Substituting \eqref{duhamelformula} and \eqref{duhamelformula2} into \eqref{wronskian}, \eqref{wronskian_W}, one verifies that  $S(q,k), \, W(q,k)$ satisfy for $k \in \R$ and $q \in L^2_{4}$
\begin{align}
& S(q, k)=  \intii e^{ikt} q(t) f_1(q, t,k) dt \ , \label{S.1}\\
& W(q, k) = 2ik - \intii e^{-ikt} q(t) f_1(q, t,k) dt \ .
\label{W.1}
\end{align}
Note that the integrals above are well defined thanks to the estimate in item $(ii)$ of Theorem  \ref{deift_jost}.
\vspace{0.5em}\newline
Inserting  formula \eqref{duhamelformula} into \eqref{S.1}, one gets that
$$
S(q,k) = \F_-(q,k) +O\left( \tfrac{1}{k}\right) \ .
$$
The main result of this section is an  estimate  of 
\begin{equation}
\label{map.A}
A(q,k) := S(q,k) - \F_-(q,k) \ ,
 \end{equation}
saying   that $A$ is  $1$-smoothing. 
To formulate the result in a precise way, we need to introduce the  following  Banach  spaces. For $M \in \Z_{\geq 1}$ define
\begin{equation*}
 \begin{aligned}
&H^M_* := \lbrace  f \in H^{M-1}_\C : \quad \overline{f(k)}= f(-k), \quad k \derk^M f \in L^2 \rbrace \ ,
\end{aligned}
\end{equation*}
endowed with the norm
$$
\norm{f}_{H^M_*}^2 := \norm{f}_{H^{M-1}_\C}^2
+ \norm{k\derk^M f}_{ L^2}^2 \ .$$
Note that $H^M_{*}$ is a  \textit{real} Banach space.
We will use also the complexification of the Banach spaces $H^M_*$ and $\Hz$ (this last defined in \eqref{H^N*}), in which the reality condition $\overline{f(k)} = f(-k)$ is dropped:
\begin{equation*}
 \begin{aligned}
&H^M_{*,\C} := \lbrace  f \in H^{M-1}_\C :  \quad k \derk^M f \in L^2 \rbrace, 
\qquad \Hzc := \lbrace  f \in H^{M-1}_\C:  \quad \zeta \derk^M f \in L^2 \rbrace.
\end{aligned}
\end{equation*}
 Note that for any $M\geq 2$
\begin{equation}
\label{lemHN*}
(i)\, H^M_\C  \subset \Hzc \mbox{ and }  H^M_{*,\C} \subset \Hzc, \quad (ii)\, fg \in \Hzc \qquad \forall \, f \in  H^M_{*, \C}, \, g \in \Hzc.
\end{equation}
We can now state the main theorem of this section.
 Let $L^2_{M, \R} := \left\{ f \in L^2_M \ \vert \ f \mbox{ real valued } \right\} $.

\begin{theorem}
\label{A.prop} Let $N \in \Z_{\geq 0}$ and $M \in \Z_{\geq 4}$. Then one has:
\begin{enumerate}[(i)]
\item The map $q \mapsto A(q, \cdot)$ is  analytic as a map from $L^2_{M}$ to $\Hzc$.
\item  The map $q \mapsto A(q, \cdot)$ is  analytic as a map from $H^N_\C \cap L^2_4$ to $L^2_{N+1}$. Moreover 
 $$\norm{A(q, \cdot)}_{L^2_{N+1}} \leq C_A \norm{q}^2_{H^N_\C \cap L^2_4}$$
where the constant $C_A>0$ can be chosen uniformly on bounded subsets of $H^N_\C \cap L^2_4$. 

Furthermore for $q \in L^2_{4, \R}$  the map  $A(q, \cdot)$ satisfies  $\overline{A(q,k)}= A(q,-k)$ for every $k \in \R$. Thus its restrictions $A: L^2_{M,\R} \to \Hz$ and $A: H^N \cap L^2_4 \to L^2_{N+1}$ are real analytic.
\end{enumerate}
\end{theorem}
The following corollary follows immediately from identity \eqref{map.A}, item $(ii)$ of Theorem \ref{A.prop} and   the properties of the Fourier transform:
\begin{cor} 
\label{S.decay}
Let $N \in \Z_{\geq 0}$. Then the map $q \mapsto S(q, \cdot)$ is  analytic as a map from $H^N_\C \cap L^2_4$ to $L^2_{N}$. Moreover 
 $$\norm{S(q, \cdot)}_{L^2_{N}} \leq C_S \norm{q}_{H^N_\C \cap L^2_4}$$
where the constant $C_S>0$ can be chosen uniformly on bounded subsets of $H^N_\C \cap L^2_4$.
\end{cor}

In \cite{beat2}, it is shown that in the periodic setup, the Birkhoff map of KdV is 1-smoothing. As the map $q \mapsto S(q, \cdot)$ on the spaces considered can be viewed as a version of the Birkhoff map in the scattering setup of KdV, Theorem \ref{A.prop} confirms that a result analogous to the one on the circle holds also on the line.

The proof of  Theorem \ref{A.prop} consists of several steps. We begin by proving item $(i)$. Since $\F_-: L^2_{M} \to H^M_\C$ is bounded, item $(i)$ will follow  from the following proposition:
\begin{proposition}
\label{prop:scatt}
Let $M \in \Z_{\geq 4}$, then  the map $L^2_{M} \ni q  \mapsto S(q, \cdot) \in  \Hzc$ is analytic and 
$$\norm{ S(q, \cdot)}_{ \Hzc}\leq K_S \norm{q}_{L^2_{M}},$$
 where $K_S>0$ can be chosen uniformly  on bounded subsets of $ L^2_{M}$.
\end{proposition}
\begin{proof} Recall that $f_1(q,x,k) = e^{ikx}\,m_1(q,x,k)$ and $f_2(q,x,k) = e^{-ikx} \, m_2(q,x,k)$. The $x$-independence of $S(q, k)$ implies that
\begin{equation}
\label{SWwronskian}
 S(q, k) =  [m_1(q, 0,k),\, m_2(q, 0,-k)] \ . 
 \end{equation}
As by Corollary \ref{m(q,0,k)},   $m_j(q,0, \cdot)-1 \in H^M_{*, \C}$ and $\derx m_j(q,0, \cdot) \in \Hzc$, $j=1,2$, the identity \eqref{SWwronskian} yields
\begin{align*}
 S(q,k) = & (m_1(q,0,k)-1)\, \derx m_2(q,0,-k) - (m_2(q,0,-k)-1)\, \derx m_1(q,0,k) \\
 & + \derx m_2(q,0,-k) - \derx m_1(q,0,k) \ ,
 \end{align*}
thus $S(q, \cdot) \in \Hzc $  by \eqref{lemHN*}. The estimate on the norm $\norm{ S(q, \cdot)}_{ \Hzc}$ follows by Corollary \ref{m(q,0,k)}.
\end{proof}
\vspace{0.5em}
{\em Proof of Theorem \ref{A.prop} $(i)$.}
The claim is a direct consequence of Proposition \ref{prop:scatt} and the fact that  for any real valued potential $q$, $\overline{S(q,k)} = S(q, -k)$, $\overline{\F_-(q,k)}= \F_-(q,-k)$ and hence $\overline{A(q,k)}= A(q,-k)$ for any $k \in \R$.
\qed
\vspace{1em}\newline
In order to prove the second item of Theorem \ref{A.prop}, we  expand the map $q \mapsto A(q)$ as a power series of $q$. More precisely, iterate formula \eqref{duhamelformula}  and insert the formal expansion  obtained in this way in the integral term of \eqref{S.1}, to get
\begin{equation}
\label{b0series}
S(q,k)= \F_-(q,k) + \sum_{n \geq 1} \frac{s_n(q,k)}{k^n}
\end{equation}
where, with $dt = dt_0 \cdots dt_n$,
\begin{equation}
\label{expansionSn}
s_n(q,k):=\int_{\Delta_{n+1}} e^{ikt_0} q(t_0) \prod_{j=1}^n \Big(q(t_j) \, \sin k(t_j-t_{j-1})\Big)e^{ikt_n} \, dt
\end{equation}
 is a polynomial of degree  $n+1$ in $q$ (cf Appendix \ref{analytic_map}) and 
$\Delta_{n+1}$ is given by $$
\Delta_{n+1} := \left\{(t_0, \cdots, t_n) \in \R^{n+1}: \quad t_0\leq \cdots \leq 
t_n \right\} . 
$$
Since by Proposition \ref{prop:scatt}  $S(q, \cdot)$ is in $L^2$,  it remains to control the decay  of $A(q, \cdot) $ in $k$ at infinity. Introduce a cut off function $\chi$ with $\chi(k)=0$ for $|k|\leq 1$ and $\chi(k)=1$ for $|k| > 2$ and consider the series 
\begin{equation}
\label{b1series}
\chi(k) S(q,k)= \chi(k) \F_-(q,k) + \sum_{n \geq 1} \frac{\chi(k) s_n(q,k)}{k^n}.
\end{equation}
Item $(ii)$ of Theorem \ref{A.prop} follows once  we show that  each  term $\tfrac{\chi(k) s_n(q,k)}{k^n}$ of the series is bounded  as a map from $H^N_\C \cap L^2_{4}$ into $L^2_{N+1}$ and the series  has an infinite radius of convergence in  $L^2_{N+1}$. Indeed the analyticity of the map then  follows from general properties of analytic maps in complex Banach spaces, see  Remark \ref{entire.func}.\\
In order to estimate the terms of the series, we need estimates on the maps $k \mapsto s_n(q,k)$. A first trivial bound   is given by 
\begin{equation}
\label{sn.l.inf}
\norm{s_n(q, \cdot)}_{L^\infty} \leq \tfrac{1}{(n+1)!} \norm{q}_{L^1}^{n+1}.
\end{equation}
However, in order to prove convergence of \eqref{b1series}, one needs more refined estimates of the norm of $k\mapsto s_n(q,k)$ in $L^2_{N}$. In order to derive such estimates, we begin with  a preliminary lemma about oscillatory integrals:
\begin{lemma}
\label{hyp_red}
Let $f \in L^1(\R^n, \C)\cap L^2(\R^n, \C)$. Let $\alpha \in \R^n$, $\alpha \neq 0$ and
$$
g: \R \rightarrow \C, \quad
g(k):= \int_{\R^n}e^{ik \alpha \cdot t} f(t)\; dt.
$$
Then $g \in L^2$ and for any component $\alpha_i \neq 0$ one has
\begin{equation}
\norm{g}_{L^2} \leq \int\limits_{\R^{n-1}} 
\Big(\intii |f(t)|^2 \, d t_i \Big)^{1/2} dt_1 \ldots \widehat{d t_i} \ldots dt_n.
\end{equation}
\end{lemma}
\begin{proof} The lemma is a  variant of   Parseval's theorem for the Fourier transform; indeed
\begin{equation}
\label{hr1}
\norm{g}_{L^2}^2 = \int_{\R} g(k)\,  \overline{g(k)}\, dk  = \int\limits_{\R \times \R^n \times \R^n}  e^{ik \alpha \cdot (t-s)} f(t) \overline{f(s)} \, dt \, ds \, dk.
\end{equation}
Integrating first in the $k$ variable and using the distributional identity $\int_{\R}e^{ikx} \, dk= \frac{1}{2\pi} \delta_{0}$, where $\delta_0$ denotes the Dirac delta function, one gets 
\begin{equation}
\norm{g}_{L^2}^2  =\frac{1}{2\pi} \int\limits_{\R^{n}\times \R^{n}} f(t)\, \overline{f(s)}\, \delta(\alpha \cdot (t-s))  \,dt \, ds
\end{equation}
Choose an index $i$ such that $\alpha_i \neq 0$; then $\alpha \cdot (t-s) = 0 $ implies  that
$s_i = t_i + c_i / \alpha_i$, where $c_i = \sum_{j \neq i} \alpha_j(t_j - s_j)$. Denoting $d\sigma_i=  dt_1 \cdots \widehat{dt_i} \cdots dt_n$ and  $d\tilde{\sigma_i}=  ds_1 \cdots \widehat{ds_i} \cdots ds_n$, one has,  integrating first in the  variables $s_i$ and $t_i$,
\begin{equation}
\begin{aligned}
\norm{g}_{L^2}^2 & =\frac{1}{2\pi} \int\limits_{\R^{n-1}\times \R^{n-1} } d\sigma_i \, d\tilde{\sigma_i} \, \int_\R f(t_1, \ldots, t_i, \ldots, t_n) \overline{f(s_1, \ldots, t_i + c_i/\alpha_i, \ldots, s_n)} dt_i  \\
& \quad \leq \int\limits_{\R^{n-1}\times \R^{n-1} } d\sigma_i \, d\tilde{\sigma_i} \Big(\intii |f(t)|^2 \, dt_i \Big)^{1/2} \cdot \Big(\intii |f(s)|^2 \, ds_i \Big)^{1/2} \\
& \quad \leq \Big(\int\limits_{\R^{n-1} } d\tilde\sigma_i  \, \Big(\intii |f(s)|^2 \, ds_i \Big)^{1/2} \Big)^2
\end{aligned}
\end{equation}
where in the second line we have used the Cauchy-Schwarz inequality and the invariance of the integral $\intii |f(s_1, \ldots, t_i + c_i/\alpha_i, \ldots, s_n)|^2 $ by translation. 
\end{proof}

To get bounds on the norm of the polynomials $k \mapsto s_n(q,k)$ in $L^2_{N}$ it is convenient to study the  multilinear maps associated with  them:
\begin{align*}
\label{multisn}
\tilde s_n \ : \ & \left(H^N_\C \cap L^1 \right)^{n+1} \to L^2_N \ , \\
& (f_0, \cdots, f_n) \mapsto \tilde{s}_n(f_0, \cdots, f_n):= \int_{\Delta_{n+1}} e^{ikt_0}f_0(t_0) 
\prod_{j=1}^{n} \Big(f_j(t_j) \, \sin (k(t_j-t_{j-1})) \Big)\,e^{ikt_n} \; dt \ .
\end{align*}

The boundedness of these multilinear maps is given by the following 
\begin{lemma}
\label{tilde_s_n}
 For each $n \geq 1$ and $N \in \Z_{\geq 0}$,  $\tilde{s}_n: (H^N_\C \cap L^1)^{n+1} \to L^2_{N}$ is bounded. In particular  there exist  constants $C_{n,N}>0$ such that
\begin{equation}
\label{s_n_tilde_estim}
 \norm{\tilde{s}_n(f_0, \ldots, f_n)}_{L^2_{N}} \leq C_{n,N} \norm{f_0}_{H^N_\C \cap L^1} \cdots \norm{f_n}_{H^N_\C \cap L^1}.
 \end{equation}
\end{lemma}
For the proof, introduce the operators $I_j : L^1 \to L^\infty$, $j= 1,2$, defined by
\begin{equation}
I_1(f)(t):= \intx{t} f(s)\, ds \qquad I_2(f)(t):= \int\limits^t_{-\infty} f(s)\, ds.
\end{equation}
It is easy to prove that if $u, v \in H^N_\C\cap L^1$, then $u\,I_j(v) \in H^N_\C\cap L^1$ and the estimate $\norm{u\, I_j(v)}_{H^N_\C}\leq \norm{u}_{H^N_\C \cap L^1} \norm{v}_{H^N_\C \cap L^1}$ holds for $j=1,2$.
\vspace{0.5em}\\
{\em Proof of Lemma \ref{tilde_s_n}.} 
As $\sin x = (e^{ix}- e^{-ix})/2i$ we can write 
$e^{ikt_0} \Big( \prod_{j=1}^n \sin k(t_j-t_{j-1})\Big) e^{ik t_n} $ as a sum of complex exponentials. 
Note that the arguments of the exponentials are obtained by taking all the possible combinations of $\pm$ in the expression $t_0 \pm (t_1-t_0) \pm \ldots \pm (t_n-t_{n-1}) + t_n$.
To handle this combinations, define  the set
\begin{equation}
\begin{aligned}
\Lambda_n := \Big\{ \sigma=(\sigma_j)_{1 \leq j \leq n}: \, \sigma_j \in \{\pm 1 \} \Big\}
\label{index_set}
\end{aligned}
\end{equation}
and introduce 
$$
\delta_\sigma := \# \{ 1 \leq j \leq n : \, \sigma_j = -1 \}.
$$
For any $\sigma \in \Lambda_n$, define $\alpha_\sigma = (\alpha_j)_{0 \leq j \leq n}$ as
$$
\alpha_0 = (1- \sigma_1), \quad \alpha_j= \sigma_j - \sigma_{j+1} \mbox{ for } 1 \leq j \leq n-1, \quad \alpha_n = 1 + \sigma_n.
$$
Note that for any $t=(t_0, \ldots, t_n)$, one has $\alpha_\sigma \cdot t = t_0 + \sum_{j=1}^n \sigma_j (t_j - t_{j-1}) + t_n$.\\
 For every $\sigma \in \Lambda_n$, $\alpha_\sigma$ satisfies the following properties: 
\begin{equation}
\label{set.ind.prop}
(i) \; \alpha_0,\, \alpha_n \in \left\{2, 0 \right\}, \;  \alpha_j \in \left\{ 0, \pm 2 \right\}\, \forall 1 \leq j \leq n-1; \quad (ii)\;\#\left\{j \middle| \alpha_j \neq 0 \right\}  \mbox{ is odd.}
\end{equation}
Property $(i)$ is obviously true;  we prove now $(ii)$ by induction. For $n=1$, property $(ii)$ is trivial. To prove  the induction step $n \leadsto n+1$, let $\alpha_0= 1-\sigma_1, \ldots, \alpha_n=\sigma_n - \sigma_{n+1},\, \alpha_{n+1}=1+\sigma_{n+1}$, and define $\tilde \alpha_n:= 1+ \sigma_n \in \{0, 2 \}$.  By the induction hypothesis the vector $\tilde \alpha_\sigma = (\alpha_0, \ldots, \alpha_{n-1}, \tilde{\alpha}_n)$ has an odd number of elements non zero. Case $\tilde{\alpha}_n =0 $: in this case the vector $(\alpha_0, \ldots, \alpha_{n-1})$ has an odd number of non zero elements. Then, since $\alpha_n = \sigma_n - \sigma_{n+1} = \tilde{\alpha}_n- \alpha_{n+1}= - \alpha_{n+1}$, one has that $(\alpha_n, \alpha_{n+1}) \in \{ (0,0), \, (-2, 2)\}$. Therefore the vector $\alpha_\sigma$ has an odd number of non zero elements.
Case  $\tilde{\alpha}_n =2$: in this case the vector $(\alpha_0, \ldots, \alpha_{n-1})$ has an even number of non zero elements. As  $\alpha_n = 2 - \alpha_{n+1}$, it follows that
$(\alpha_n, \alpha_{n+1}) \in \{ (2,0), \, (0, 2)\}$. Therefore the vector $\alpha_\sigma$ has an odd number of non zero elements. This proves \eqref{set.ind.prop}.\\
As
$$e^{ikt_0} \Big( \prod_{j=1}^n \sin k(t_j-t_{j-1})\Big) e^{ik t_n} = \sum_{\sigma \in \Lambda_n} \frac{(-1)^{\delta_\sigma}}{(2i)^n}e^{ik \alpha \cdot t}$$
  $\tilde{s}_n$ can be written as a sum of complex exponentials, $\tilde{s}_n(f_0, \ldots, f_n)(k)= \sum_{\sigma \in \Lambda_n } \frac{(-1)^{\delta_\sigma}}{(2i)^n} \tilde{s}_{n,\sigma}(f_0, \ldots, f_n)(k)$ where 
\begin{equation}
\tilde{s}_{n,\sigma}(f_0, \ldots, f_n)(k)= \int_{\Delta_{n+1}}e^{ik \alpha \cdot t} f_0(t_0) \cdots f_n(t_n) dt.
\end{equation}
 The case $N=0$ follows directly from Lemma \ref{hyp_red}, since for each $\sigma \in \Lambda_n$ one has by \eqref{set.ind.prop} that there exists $m$ with $\alpha_m \neq 0$ implying 
$\norm{\tilde{s}_{n,\sigma}(f_0, \ldots, f_n)}_{L^2} \leq C \norm{f_m}_{L^2} \prod_{j \neq m}  \norm{f_j}_{L^1}$, which leads to \eqref{s_n_tilde_estim}.

We now prove by induction that $\tilde{s}_n: (H^N_\C \cap L^1)^{n+1} \to L^2_{N} $ for any $N \geq 1$.  We start with $n=1$.
Since we have already proved that $\tilde{s}_1$ is a bounded map from $(L^2 \cap L^1)^2$ to $L^2$,   it is enough to establish the stated  decay at $\infty$. One verifies that 
\begin{align*}
\tilde{s}_1(f_0, f_1) & = \frac{1}{2i}\intii e^{2ikt}\, f_0(t) \, I_1(f_1)(t) \,dt - \frac{1}{2i} 
\intii e^{2ikt}\, f_1(t)\, I_2(f_0)(t) \;dt \\
&  = \frac{1}{2i}\F_-(f_0\, I_1(f_1)) - \frac{1}{2i}\F_-(f_1\, I_2(f_0)).
\end{align*}
Hence, for each $N \in \Z_{\geq 0}$, $(f_0, f_1) \mapsto \tilde{s}_1(f_0,f_1)$ is bounded as a map from $(H^N_\C \cap L^1)^2$ to $L^2_{N}$. Moreover 
$$\norm{\tilde{s}_1(f_0, f_1)}_{L^2_{N}}\leq C_1 \left(\norm{f_0\, I_1(f_1)}_{H^N_\C} + \norm{f_1\, I_2(f_0)}_{H^N_\C}\right) \leq C_{1,N} \norm{f_0}_{H^N_\C \cap L^1 }\norm{f_1}_{H^N_\C\cap L^1}.$$
We prove the induction step $n \leadsto n+1$ with $n \geq 1$ for any $N \geq 1$ (the case $N=0$ has been already treated). The term $\tilde{s}_{n+1}(f_0, \ldots, f_{n+1})$ equals
\begin{align*}
\int_{\Delta_{n+2}} e^{ikt_0}f_0(t_0) 
\prod_{j=1}^{n} \Big(\sin k(t_j-t_{j-1}) f_j(t_j) \Big)e^{ikt_n} \sin k(t_{n+1}-t_n)e^{ik(t_{n+1}-t_n)} f_{n+1}(t_{n+1})\, dt
\end{align*}
where we multiplied and divided by the factor $e^{ikt_n}$. 
Writing $$\sin k(t_{n+1}-t_n) = (e^{ik(t_{n+1}- t_n)} - e^{-ik(t_{n+1}- t_n)})/2i \ , $$ the integral term  $\intx{t_{n}} \, e^{ik(t_{n+1}-t_n)} \sin k(t_{n+1}-t_{n})\, f_{n+1}(t_{n+1})\; dt_{n+1}$ equals 
\begin{align*}
 \frac{1}{2i}\intx{t_n}e^{2ik(t_{n+1}-t_n)}f_{n+1}(t_{n+1}) \,dt_{n+1} - \frac{1}{2i} I_1(f_{n+1})(t_{n}).
\end{align*}
Since $f_{n+1}\in H^N_\C$, for $0 \leq j \leq N-1$ one gets $f_{n+1}^{(j)}\rightarrow 0$ when $x\to \infty$, where  we wrote $f_{n+1}^{(j)} \equiv \derk^j f_{n+1}$. Integrating  by parts $N$-times in the integral expression displayed above one has
\begin{align*}
\frac{1}{2i}\sum_{j=0}^{N-1} \frac{(-1)^{j+1} }{(2ik)^{j+1}}\, f_{n+1}^{(j)}(t_n) + 
\frac{(-1)^N}{2i (2ik)^N}\intx{t_n}e^{2ik(t_{n+1}-t_n)}  f_{n+1}^{(N)}(t_{n+1}) \, dt_{n+1} - \frac{1}{2i} I_1(f_{n+1})(t_{n}).
 \end{align*}
Inserting the formula above in the expression for $\tilde{s}_{n+1}$, and using the multilinearity of $\tilde{s}_{n+1}$ one gets
\begin{align}
& \tilde{s}_{n+1}(f_0, \ldots, f_{n+1})=\frac{1}{2i}\sum_{j=0}^{N-1}\frac{(-1)^{j+1}}{(2ik)^{j+1}} \tilde{s}_n(f_0, \ldots, f_n \cdot f_{n+1}^{(j)}) - \frac{1}{2i}\tilde{s}_n(f_0, \ldots, f_n\, I_1(f_{n+1})) \label{line_s_n_1} \\
& \qquad +  \frac{(-1)^N}{2i(2ik)^N} 
\int_{\Delta_{n+2}} e^{ikt_0} f_0(t_0) \prod_{j=1}^n \Big(\sin k(t_j-t_{j-1}) \, f_j(t_j)\Big)  e^{2ik t_{n+1}} f_{n+1}^{(N)}(t_{n+1}) \; dt_{n+1}.
\label{line_s_n_2}
\end{align}
We analyze the first term in the r.h.s. of  \eqref{line_s_n_1}. For $0 \leq j \leq N-1$, the function $f_{n+1}^{(j)} \in H^{N-j}_\C $ is in $ L^\infty$ by the Sobolev embedding theorem. Therefore 
$f_n \cdot f_{n+1}^{(j)} 	\in H^{N-j}_\C \cap L^1$. By the inductive assumption applied to $N-j$, 
$\tilde{s}_n(f_0, \ldots,f_n \cdot f_{n+1}^{(j)}) \in L^2_{N-j} $. Therefore 
$\frac{\chi}{(2ik)^{j+1}} \tilde{s}_n(f_0, \ldots, f_n \cdot f_{n+1}^{(j)})\in L^2_{N}$ , where $\chi$ is chosen as in \eqref{b1series}.
For the second term in \eqref{line_s_n_1} it is enough to note that  $f_n \, I_1(f_{n+1}) \in H^N_\C \cap L^1$ and  by the  inductive assumption it follows that $\tilde{s}_n(f_0, \ldots, f_n \, I_1(f_{n+1})) \in L^2_{N}$.
\newline
We are left with \eqref{line_s_n_2}. Due to the factor $(2ik)^N$ in the denominator, we need just to prove that the integral term is $L^2$ integrable in the $k$-variable. Since the oscillatory factor  $e^{2ik t_{n+1}}$ doesn't get canceled when we express the sine functions with exponentials,  we can apply  Lemma \ref{hyp_red},  integrating first in $L^2$ w.r. to the variable $t_{n+1}$, getting 
$$\norm{\chi \cdot \eqref{line_s_n_2}}_{L^2_{N}} \leq C_{n+1,N} \norm{f_{n+1}^{(N)}}_{L^2} \prod_{j=0}^n \norm{f_j}_{L^1}.$$

Putting all together, it follows that $\tilde{s}_{n+1}$ is bounded as a map from $(H^N_\C \cap L^1)^{n+2}$ to $L^2_{N}$ for each $N \in \Z_{\geq 0}$ and the estimate \eqref{s_n_tilde_estim} holds.
\qed
\vspace{0.5em}\\
By evaluating the multilinear map $\tilde s_n$ on the diagonal, Lemma \ref{tilde_s_n} says that for any $ N \geq 0$,
\begin{equation}
\label{sn.est.2}
\norm{s_n(q, \cdot)}_{L^2_{N}} \leq C_{n,N} \norm{q}_{H^N_\C \cap L^1}^{n+1}, \qquad \forall n \geq 1.
\end{equation}
Combining the $L^\infty$ estimate \eqref{sn.l.inf} with \eqref{sn.est.2} we can now  prove
 item $(ii)$ of Theorem \ref{A.prop}:
 \vspace{1em}\\
{\em Proof of Theorem \ref{A.prop} $(ii)$.} Let $\chi$ be the cut off function introduced in \eqref{b1series} and set
\begin{equation}
\label{A_n.tilde}
\tilde A(q,k) := \sum_{n = 1}^{\infty} \frac{\chi(k)  s_n(q,k)}{k^n}.
\end{equation}
We now show that for any $\rho >0 $,  $\tilde A(q, \cdot)$  is an absolutely and uniformly convergent series in $L^2_{N+1}$ for $q$ in $ B_\rho(0)$, where $B_\rho(0)$ is the ball in $H^N_\C \cap L^1$ with center $0$ and radius $\rho$. By \eqref{sn.est.2} the map $q \mapsto \sum_{n = 1}^{N+1} \frac{\chi(k)  s_n(q,k)}{k^n} $ is analytic as a map from $H^N_\C \cap L^1$ to $ L^2_{N+1}$, being a finite sum of polynomials - cf. Remark \ref{entire.func}.
It remains to estimate the sum
$$\tilde{A}_{N+2}(q,k) := \tilde A(q,k) - \sum_{n = 1}^{N+1} \frac{\chi(k)  s_n(q,k)}{k^n} \ .$$ 
It   is absolutely convergent since by the $L^\infty$ estimate \eqref{sn.l.inf} 
\begin{equation}
\label{A_N+1_norm}
\norm{ \sum_{n\geq N+2}  \frac{\chi s_n(q,\cdot)}{k^n} }_{L^2_{N+1}} \leq  \sum_{n \geq N+2} \norm{\frac{\chi(k) }{k^n}}_{L^2_{N+1}} \norm{s_n(q,\cdot)}_{L^\infty} \leq C \sum_{n \geq N+2} \frac{\norm{q}_{L^1}^{n+1}}{(n+1)!} 
\end{equation}
for an absolute constant $C>0$. 
 Therefore the series in \eqref{A_n.tilde} converges absolutely and uniformly in $B_\rho(0)$ for every $\rho >0$. The absolute and uniform convergence  implies that for any $N \geq 0$,   $q \mapsto \tilde A(q, \cdot)$ is  analytic as a map from $H^N_\C \cap L^1$ to $L^2_{N+1}$.\\
 It remains to show that identity \eqref{b1series} holds, i.e.,  for every $q \in H^N_\C \cap L^1$ one has  $\chi A(q,\cdot) =  \tilde A(q,\cdot)$ in $L^2_{N+1}$. Indeed, fix $q \in H^N_\C \cap L^1$ and choose $\rho$ such that $\norm{q}_{H^N_\C \cap L^1} \leq \rho$.  Iterate formula \eqref{duhamelformula} $N'\geq 1$ times and insert the result in \eqref{S.1} to get for any $k \in \R \setminus \{0 \}$, 
 $$
 S(q,k) = \F_-(q,k) + \sum_{n = 1}^{N'} \frac{  s_n(q,k)}{k^n} + S_{N'+1}(q,k) \ ,
 $$
 where 
 $$
 S_{N'+1}(q,k) := \frac{1}{k^{N'+1}}\int_{\Delta_{N'+2}} e^{ikt_0} q(t_0) \prod_{j=1}^{N'+1} \Big(q(t_j) \, \sin k(t_j-t_{j-1})\Big)f_1(q, t_{N'+1}, k)  \, dt \ .
 $$
By the definition \eqref{map.A} of $A(q,k)$  and the expression of $S_{N'+1}$ displayed above 
 $$
 \chi(k) A(q,k) - \sum_{n = 1}^{N'} \frac{\chi(k)  s_n(q,k)}{k^n} = \chi(k)  S_{N'+1}(q,k), \qquad \forall N' \geq 1 \ .
 $$
Let now $N' \geq N$, then by Theorem \ref{deift_jost} $(ii)$ there exists a constant $K_\rho$, which can be chosen uniformly on $B_\rho(0)$ such that 
$$
\norm{\chi  S_{N'+1}(q,\cdot)}_{L^2_{N+1}} \leq K_\rho \frac{\norm{q}_{L^1_1}^{N'+2}}{(N'+2)!}\leq K_\rho \frac{\rho^{N'+2}}{(N'+2)!} \to 0, \qquad \mbox{ when } N'\to \infty \ , $$
where for the last inequality we used that $\norm{q}_{L^1_1} \leq C \norm{q}_{L^2_2}$ for some absolute constant $C>0$.
Since $\lim_{N' \to 0} \sum_{n = 1}^{N'} \frac{\chi(k)  s_n(q,k)}{k^n} = \tilde{A}(q,k)$ in $L^2_{N+1}$, it follows that $\chi(k) A(q,k) = \tilde A(q,k)$ in $L^2_{N+1}$.
 \qed

\vspace{1em}
For later use we study regularity and decay properties  of the map $k \mapsto W(q,k)$. For $q\in L^2_{4}$ real valued with no bound states it follows that $ W(q,k)\neq 0, \, \forall \, \Im k \geq 0$ 
by classical results in scattering theory. We define 
\begin{equation}
\label{class.C.def} 
\class_\C:= \left\{ q \in L^2_{4}: W(q,k)\neq 0, \,  \forall \, \Im k \geq 0 \right\}, \quad \class^{N,M}_\C:= \class_\C \cap H^N_\C \cap L^2_{M} \ . 
\end{equation} 
We will prove in Lemma \ref{class_open} below that $\class^{N,M}_\C$ is open in $ H^N_\C \cap L^2_{M}$. 
Finally consider  the Banach space $W^M_\C$ defined for $M \geq 1$ by
\begin{equation}
W^M_\C := \lbrace  f \in L^\infty :  \quad \derk f \in H^{M-1}_\C \rbrace \ ,
\end{equation}
endowed with the norm $\norm{f}_{W^M_\C}^2 = \norm{f}_{L^\infty}^2 + \norm{\derk f}_{H^{M-1}_\C}^2 $.\\
Note that $H^M_\C \subseteq W^M_\C$  for any $M\geq 1$ and   
\begin{equation}
\label{lemHN*2} 
gh \in \Hzc \quad \forall \,  g \in \Hzc, \ \forall \, h \in W^M_\C \ .
\end{equation}
 The properties of the map $W$ are summarized in the following Proposition:
\begin{proposition}
\label{Wlem} 
For  $M \in \Z_{\geq 4}$ the following  holds:
\begin{enumerate}[(i)]
		\item The map $L^2_{M} \ni q \mapsto W(q,\cdot)-2ik+ \F_-(q,0)  \in \Hzc$ is  analytic and 
	$$\norm{W(q, \cdot)-2ik + \F_-(q,0)}_{\Hzc} \leq C_W \norm{q}_{L^2_M},$$ where the constant $C_W>0$ can be chosen 	uniformly on bounded subsets 
	of $L^2_M$.
	\item The map $\class^{0,M}_\C \ni q \mapsto 1 / W(q, \cdot) \in L^\infty $ is  analytic. 
	\item The maps 
	$$ \class^{0,M}_\C \ni q \mapsto \frac{\derk^j W(q,\cdot)}{W(q,\cdot)} \in L^2 \ \mbox{ for } 0 \leq j \leq M-1 \quad \mbox{and} \quad \class^{0,M}_\C \ni q \mapsto \frac{\zeta \derk^M W(q,\cdot)}{W(q,\cdot)} \in L^2$$
	 are analytic. Here  $\zeta$ is a function as in \eqref{zeta}.
\end{enumerate}
\end{proposition}

\begin{proof}
 The $x$-independence of the Wronskian function \eqref{wronskian_W}  implies that
\begin{equation}
\label{Wwronskian}
\begin{aligned}
 W(q, k) =  2ik \, m_2(q, 0,k)\,m_1(q, 0,k) + [m_2(q, 0,k),\, m_1(q, 0,k)].
\end{aligned}
\end{equation}
Introduce for $j=1,2$ the functions  
$\grave{m}_j(q,k) := 2ik \, (m_j(q,0,k) -1)$. 
By the integral formula \eqref{defm} one verifies that
\begin{equation}
\begin{aligned}
\label{grave.m}
\grave{m}_1(q,k) & = \intzi \left( e^{2ikt}-1 \right)\, q(t)\, (m_1(q,t,k) - 1) \, dt + \intzi  e^{2ikt}\, q(t) \, dt - \intzi q(t)\, dt;\\
\grave{m}_2(q,k) & = \int\limits_{-\infty}^0 \left( e^{-2ikt}-1 \right)\, q(t)\, (m_2(q,t,k) - 1) \, dt + \int\limits_{-\infty}^0  e^{-2ikt} \, q(t) \, dt - \int\limits_{-\infty}^0 q(t)\, dt.
\end{aligned}
\end{equation}
 A simple computation using \eqref{Wwronskian} shows that $W(q,k)-2ik + \F_-(q,0) = I + II + III$ where
\begin{equation}
\label{W_wronskian}
\begin{aligned}
 & I  :=  \grave{m}_1(q,k) + 
 \grave{m}_2(q,k) +\F_-(q,0), \\
& II  := \grave{m}_1(q,k) (m_2(q,0,k)-1) \quad \mbox{ and } \quad III :=  [m_2(q,0,k), m_1(q,0,k)].
 \end{aligned}
 \end{equation}
 We prove now that each of the terms $I, II$ and $III$ displayed above is an element of $\Hzc$. We begin by discussing the smoothness of the functions $k \mapsto \grave{m}_j(q,k)$, $j=1,2$. For any $1 \leq n \leq M,$ 
$$
\derk^n \grave{m}_j(q,k) = 2in \, \derk^{n-1} (m_j(q,0,k)-1) + 2ik \,\derk^n m_j(q,0,k) \ .
$$
Thus by Corollary  \ref{m(q,0,k)} $(i)$, $\grave{m}_j(q, \cdot) \in W^M_\C$ and  $ q \mapsto \grave{m}_j(q, \cdot) $, $j=1,2$,  are analytic as maps from $L^2_{M}$ to $W^M_\C$. Consider first the term $III $ in \eqref{W_wronskian}. By Corollary \ref{m(q,0,k)},
$\norm{III(q,\cdot)}_{\Hzc} \leq K_{III} \norm{q}_{L^2_{M}} $, where $K_{III}>0$ can be chosen uniformly on bounded subsets of $L^2_{M}$. Arguing as in the proof of Proposition \ref{prop:scatt}, one shows that it is an element of $\Hzc$ and it is analytic as a map $L^2_M \to \Hzc$. Next consider the term $II$. Since  $ \grave{m}_1(q,\cdot)$ is in $W^M_\C$ and $m_2(q,0,\cdot)-1$ is in $\Hzc$, it follows by \eqref{lemHN*2} that their product is in $\Hzc$. It is left to the reader to show that  $L^2_M \to \Hzc\,$, $q \mapsto II(q)$ is analytic and furthermore $ \norm{II(q,\cdot)}_{\Hzc} \leq K_{II} \norm{q}_{L^2_M}$, where $K_{II}>0$ can be chosen uniformly on bounded subsets of $L^2_{M}$.\\
Finally let us consider term $I$. By summing the identities for $\grave{m}_1$ and $\grave{m}_2$ in  equation \eqref{grave.m}, one gets that 
\begin{equation}
\begin{aligned}
\grave{m}_1(q,k) + \grave{m}_2(q,k)+ \F_-(q,0) &=  \intzi e^{2ikt}\, q(t)\, m_1(q,t,k)  \, dt - \intzi q(t)\, (m_1(q,t,k) -1)\, dt \\
 &\quad   + \intiz e^{-2ikt}\, q(t)\, m_2(q,t,k) - \intiz q(t)\, (m_2(q,t,k) -1)\, dt   .
\end{aligned}
\end{equation}
We study just the first line displayed above, the second being treated analogously. By equation \eqref{derxm.equation} one has that $\intzi e^{2ikt}\, q(t)\, m_1(q,t,k)  \, dt = \derx m(q,0,k)$, which by Corollary \ref{m(q,0,k)} is an element of $\Hzc$ and analytic as a function $L^2_{M} \to \Hzc $. 
Furthermore, by Proposition \ref{prop_derminLit} and Proposition \ref{prop_dermiNLit} it follows that $k \mapsto \intzi q(t)\, (m_1(q,t,k) -1)\, dt$ is an element of $\Hzc$ and it is analytic as a function $L^2_{M} \to \Hzc $.
This proves item $(i)$. By Corollary \ref{m(q,0,k)}, it follows that $ \norm{I(q, \cdot)}_{\Hzc} \leq K_{I} \norm{q}_{L^2_{M}}$, where $K_{I}>0$ can be chosen uniformly on bounded subsets of $L^2_{M}$. \\
We prove now item $(ii)$.  By the  definition of $\class_\C$, for  $q \in \class^{0,4}_\C$ the function $W(q,k) \neq 0$ for any $k$ with $ \Im k \geq 0$.
By Proposition  \ref{prop:scatt} $(ii)$ and the condition  $M \geq 4$,  it follows that $W(q,k) = 2ik + L^\infty$; therefore  the map
$\class^{0,M}_\C \ni q \mapsto 1/W(q) \in  L^2 $ is analytic. \\
Item $(iii)$ follows immediately from item $(i)$ and $(ii)$.
\end{proof}

\begin{lemma}
\label{W0<0}
For any $q \in \class^{0,4}$,   $W(q,0) < 0$.
\end{lemma}

\begin{proof}
Let $q\in \class^{0,4}$ and $\kappa \geq 0$.  By formulas \eqref{duhamelformula} and \eqref{duhamelformula2} with $k=i\kappa$,  it follows that   $f_j(q,x,i\kappa)$ ($j=1,2$) is real valued (recall that $q$ is real valued). By the definition $W(q,i\kappa) = \left[ f_2, f_1\right](q,i\kappa)$ it follows that for $\kappa \geq 0$,  $W(q,i\kappa)$ is real valued. As $q$ is generic,  $W(q,i\kappa)$ has no zeroes for $\kappa \geq 0$. Furthermore for large $\kappa$ we have $W(q,i\kappa) \sim 2i(i\kappa) = -2\kappa$. Thus $W(q, i\kappa) < 0 $ for $\kappa \geq 0$.
\end{proof}
\vspace{1em}

We are now able to prove the direct scattering part of Theorem \ref{reflthm}. 
\vspace{1em}\\
\noindent{\em Proof of Theorem \ref{reflthm}: direct scattering part.}
Let $N \geq 0$,  $M \geq 4$ be fixed integers. 
First we remark that $S(q, \cdot)$ is an element of $\S^{M,N}$ if $q \in \class^{N,M}$. By \eqref{S.conj},  $S(q,\cdot)$ satisfies (S1). To see that  $S(q,0) >0$ recall that $S(q,0) = -W(q,0)$, and by Lemma \ref{W0<0}  $W(q,0)<0$. Thus $S(q,\cdot)$ satisfies (S2). Finally by Corollary \ref{S.decay} and Proposition \ref{prop:scatt} it follows that $S(q,\cdot) \in \S^{M,N}$. The analyticity properties of the map $q \mapsto S(q,\cdot)$ and $q \mapsto A(q,\cdot)$ follow by Corollary \ref{S.decay}, Proposition \ref{prop:scatt} and  Theorem \ref{A.prop}.
\qed 
\\\\
We conclude this section with a lemma  about the openness of $\class^{N,M}$ and $\S^{M,N}$.
\begin{lemma} 
\label{class_open}
For any integers $N \geq 0, \, M \geq 4$, $\class^{N,M}$ $[ \class^{N,M}_\C]$ is open in $H^N \cap L^2_M$ $[ H^N_\C \cap L^2_{M} ]$.
\end{lemma}
\begin{proof}
The proof can be found in \cite{kapptrub2}; we sketch it here for the reader's convenience.
By a classical result in scattering theory \cite{deift}, $W(q,k)$ admits an analytic extension to the upper plane $\Im k \geq 0$. 
By definition  \eqref{class.C.def} one has
$\class_\C = \{ q \in L^2_4: \, W(q,k) \neq 0\;  \quad \forall \, \Im k \geq 0 \}$. Using that $(q,k)\mapsto W(q,k)$ is continuous on $L^2_4\times \R$ and that
by Proposition \ref{Wlem}, $\| W(q, \cdot)-2ik\|_{L^{\infty}}$ is bounded locally uniformly in $q\in L^2_4$ one sees that $\class_\C$ is open in $L^2_4$. The remaining statements follow in a similar fashion.
\end{proof}

Denote by $\Hzc$ the  complexification of the Banach space $\Hz$, in which the reality condition $\overline{f(k)} = f(-k)$ is dropped:
\begin{equation}
 \begin{aligned}
\label{H^N*C}
&\Hzc  := \lbrace  f \in H^{M-1}_\C:  \quad \zeta \derk^M f \in L^2 \rbrace.
\end{aligned}
\end{equation}
On $\Hzc \cap L^2_{N}$ with $M \geq 4$, $N \geq 0$, define the linear functional
$$\Gamma_0 :\Hzc \cap L^2_{N} \to \C, \quad h \mapsto h(0).$$
By the Sobolev embedding theorem  $\Gamma_0$ is  a  linear analytic map on $\Hzc \cap L^2_{N} $. 
 In view of the definition \eqref{reflspaceNM0}, $\S^{M,N} \subseteq \Hz$. 
Furthermore denote by $\S^{M,N}_\C$ the complexification of  $\S^{M,N}$. It consists of functions  $\sigma:\R \to \C$ with $\Re (\sigma(0)) > 0$ and $\sigma \in \Hzc \cap L^2_{N}$.\\
In the following   we denote by $C^{n, \gamma}(\R, \C)$, with  $n \in \Z_{\geq 0}$ and  $0 < \gamma \leq 1$, the space of complex-valued functions with $n$ continuous derivatives such that the $n^{th}$ derivative is H\"older continuous with exponent $\gamma$.

\begin{lemma}
\label{lem:S.open}
For any integers $M\geq 4$, $N\geq 0$ the subset $\S^{M,N}$ $\ [\S^{M,N}_\C]$ is   open in $\Hz \cap L^2_{N}$ $\ [\Hzc \cap L^2_{N}]$.
\end{lemma}
\begin{proof}
Clearly $H^4_{\zeta,\C}  \subseteq H^{3}_\C$, and by the Sobolev embedding theorem $H^3_{\C} \hookrightarrow C^{2, \gamma}(\R, \C)$ for any $0 < \gamma < 1/2$. It follows that   $\sigma \to \sigma(0)$ is a continuous functional on $H^4_{\zeta,\C}$.
In view of the definition of $\S^{M,N}$, the claimed statement follows.
\end{proof}

\section{Inverse scattering map}
\label{sec:inv.scat}

The aim of this section is to prove the inverse scattering part of Theorem \ref{reflthm}. 
More precisely we prove the following theorem.

\begin{theorem}
\label{thm:inv.scat}
Let $N \in \Z_{\geq 0}$ and $M \in \Z_{\geq 4}$ be fixed. Then the scattering map $S:  \class^{N,M} \to \S^{M,N}$ is bijective. Its inverse $S^{-1}:\S^{M,N} \to \class^{N,M}$ is real analytic.
\end{theorem}

The smoothing and analytic properties of $B:= S^{-1} - \F_{-}^{-1}$ claimed in Theorem \ref{reflthm} follow now in a straightforward way from Theorem \ref{thm:inv.scat} and \ref{A.prop}.
 
 \begin{proof}[Proof of Theorem \ref{reflthm}: inverse scattering part.]
 By Theorem  \ref{thm:inv.scat},  $S^{-1}:\S^{M,N} \to \class^{N,M}$ is well defined and real analytic. 
 As by definition $B = S^{-1} - \F_-^{-1}$ and $S = \F_- + A$ one has $B \circ S = Id - \F_-^{-1} \circ S = - \F_-^{-1}\circ  A$ or 
 $$B=-\F_-^{-1}\circ A \circ S^{-1} \ .$$ 
 Hence, by Theorem \ref{A.prop} and Theorem \ref{thm:inv.scat},  for any   $M \in \mathbb{Z}_{\geq 4}$ and $N \in \mathbb{Z}_{\geq 0}$ the restriction
$B:\S^{M,N}\to H^{N+1} \cap L^2_{M-1} $
is a  real analytic map. 
\end{proof}

The rest of the section is devoted to the proof of  Theorem \ref{thm:inv.scat}. 
By the direct scattering part of Theorem  \ref{reflthm} proved in Section \ref{sec:dir.scat},   $S(\class^{N,M}) \subseteq \S^{M,N}$.
Furthermore, the map $S: \class \to \S$ is 1-1, see \cite[Section 4]{kapptrub}. Thus also its restriction  $\left.S\right|_{\class^{N,M}}: \class^{N,M} \to \S^{M,N}$ is 1-1.

Let us denote by $\H: L^2 \to L^2$ the Hilbert transform
\begin{equation}
\label{def.hilbert.trans}
 \HT(v)(k):=-\frac{1}{\pi} \operatorname{p.v.} \int_{-\infty}^{\infty}\frac{v(k')}{k'-k} dk' \ .
\end{equation}
We collect in Appendix \ref{Hilbert.transf} some well known properties of the Hilbert transform which will be exploited in the following.

In order to prove that  $S:\class^{N,M} \to \S^{M,N}$ is onto, we need some preparation. Following \cite{kapptrub} define for  $\sigma \in \S^{M,N}$,
\begin{equation}
\omega(\sigma, k) :=  \exp\left(\frac{1}{2} l(\sigma,k) + \frac{i}{2} \H(l(\sigma, \cdot))(k) \right) \ , \quad l(\sigma, k) := \log\left( \frac{4(k^2+1)}{4k^2 + \sigma(k) \sigma(-k)} \right) \ , \quad k \in \R
\end{equation}
and
\begin{equation}
\label{inv.scatt.elem}
\begin{aligned}
&\frac{1}{w(\sigma,k)} := \frac{\omega(\sigma,k)}{2i(k+i)}  \ , \qquad \tau(\sigma,k) := \frac{2ik}{w(\sigma,k)} \ , \\
&\rho_+(\sigma,k) := \frac{\sigma(-k)}{w(\sigma,k)}\ , \qquad \rho_-(\sigma, k) := \frac{\sigma(k)}{w(\sigma,k)} \ .
\end{aligned}
\end{equation}
The aim is to show that $\rho_+(\sigma, \cdot)$, $\rho_-(\sigma, \cdot)$ and $\tau(\sigma, \cdot)$ are the scattering data $r_+, r_-$ and $t$ of a potential $q \in \class^{N,M}$. 

In the next proposition we discuss the properties of the map $\sigma \to l(\sigma, \cdot)$. To this aim we introduce, for $M \in \Z_{\geq 2}$ and $\zeta$ as in \eqref{zeta}, the auxiliary Banach space 
\begin{equation}
W^M_{\zeta} := \lbrace  f \in L^\infty : \overline{f(k)} = f(-k) ,  \quad \derk^{n} f \in L^2 \mbox{ for } 1 \leq n \leq M-1 \ , \ \ \ \zeta \derk^M f \in L^2 \rbrace \ 
\end{equation}
and its complexification
\begin{equation}
W^M_{\zeta,\C} := \lbrace  f \in L^\infty :  \quad \derk^{n} f \in L^2 \mbox{ for } 1 \leq n \leq M-1 \ , \ \ \ \zeta \derk^M f \in L^2 \rbrace \ ,
\end{equation}
both  endowed with the norm $\norm{f}_{W^M_{\zeta,\C}}^2 := \norm{f}_{L^\infty}^2 + \norm{\derk f}_{H^{M-2}_\C}^2 + \norm{\zeta \derk^M f}_{L^2}^2$.
Note that $W^M_{\zeta}$ differs from $\Hz$ since we require that $f$ lies just in $L^\infty$ (and not in $L^2$ as in $\Hz$). 
\begin{proposition}
\label{prop:inv.l}
Let $N \in \Z_{\geq 0}$ and $M \in \Z_{\geq 4}$ be fixed. The map $\S^{M,N} \to H^M_{\zeta}$, $\sigma \to l(\sigma, \cdot)$ is real analytic.
\end{proposition}
\begin{proof}
Denote by 
$$
h(\sigma, k):= \frac{4(k^2+1)}{4k^2 + \sigma(k) \sigma(-k)} \ .
$$
We show that the map $\S^{M,N} \to W^M_{\zeta}$, $\sigma \to h(\sigma, \cdot)$ is real analytic. 
First note that the map $\S^{M,N}_\C \to L^\infty$, assigning to $\sigma$ the function $\sigma(k)\sigma(-k)$ is analytic by the Sobolev embedding theorem.
 For $\sigma \in  \S^{M,N}_\C$ write $\sigma = \sigma_1 + i \sigma_2$, where $\sigma_1:= \Re \sigma$, $\sigma_2 := \Im \sigma$. Then 
\begin{equation}
\label{re.sigma}  
 \Re (\sigma(k)\sigma(-k)) = \sigma_1(k)\sigma_1(-k) - \sigma_2(k)\sigma_2(-k) \ . 
 \end{equation}
Now fix $\sigma^0 \in \S^{M,N}$ and recall that $\S^{M,N} = \S \cap \Hz \cap L^2_N$.
Remark that  $\sigma_2^0 := \Im \sigma^0 = 0$, while  $\sigma_1^0:= \Re \sigma^0$ satisfies $\sigma_1^0(k)\sigma_1^0(-k) \geq 0$ and  $\sigma_1^0(0)^2 > 0$. Thus,  by formula \eqref{re.sigma} and the Sobolev embedding theorem, there exists  $V_{\sigma^0}\subset \S^{M,N}_\C$  small complex neighborhood of $\sigma^0$ and a constant $C_{\sigma^0}>0$ such that 
$$\Re (4k^2 + \sigma(k)\sigma(-k)) > C_{\sigma^0} \ , \quad \forall \sigma \in V_{\sigma^0} \ .$$
It follows that there exist constants $C_1, C_2 >0$ such that
\begin{equation}
\label{h.bound}
\Re h(\sigma, k) \geq C_1 \ , \qquad |h(\sigma,k)| \leq C_2 \ , \qquad \forall k \in \R, \ \forall \sigma \in V_{\sigma^0} \ ,
\end{equation}
implying that the map $V_{\sigma^0} \to L^\infty$, $\sigma \to h(\sigma,\cdot)$ is analytic. In a similar way one proves that $V_{\sigma^0} \to  W^M_{\zeta,\C}$, $\sigma \mapsto h(\sigma, \cdot)$ is analytic (we omit the details).
If $\overline{\sigma(k)} = \sigma(-k)$, the function $h(\sigma, \cdot)$ is real valued. Thus it follows that $\S^{M,N} \to W^M_{\zeta}$, $\sigma \to h(\sigma, \cdot)$ is real analytic.

We consider now the map $\sigma \to l(\sigma, \cdot)$. By  \eqref{h.bound},  $l(\sigma, k) = \log(h(\sigma,k))$ is well defined for every $k \in \R$. Since the logarithm is a real analytic function on the right half plane,  the map $\S^{M,N} \to L^\infty$, $\sigma \to l(\sigma, \cdot)$ is real analytic as well.
Furthermore for $|k| >1$ one finds a constant $C_3 >0$ such that $|l(\sigma,k)| \leq C_3/|k|^2$, $\forall \sigma \in V_{\sigma^0}$. Thus $\sigma \to l(\sigma, \cdot)$ is real analytic as a map from $\S^{M,N}$ to $L^2$.
One verifies that  $\derk \log(h(\sigma,\cdot)) = \frac{\derk h(\sigma, \cdot) }{h(\sigma, \cdot)}$ is in $L^2$ and one shows by induction that the map
$\S^{M,N} \to \Hz$, $\sigma \mapsto l(\sigma,\cdot)$
is real analytic. 
\end{proof}

In the next proposition we discuss the properties of the map $\sigma \to \omega(\sigma, \cdot)$.
\begin{proposition}
\label{prop:inv.omega}
Let $N \in \Z_{\geq 0}$ and $M \in \Z_{\geq 4}$ be fixed. The map $\S^{M,N} \to W^M_{\zeta}$, $\sigma \to \omega(\sigma, \cdot)$ is real analytic. Furthermore $\omega(\sigma, \cdot)$ has the following properties:
\begin{enumerate}[(i)]
\item $\omega(\sigma, k)$ extends analytically in the upper half plane $\Im k > 0$, and it has no zeroes in $\Im k \geq 0$.
\item $\overline{\omega(\sigma, k)} = \omega(\sigma, -k)$ $\, \forall k \in \R$.
\item For every $k \in \R$  
$$\omega(\sigma, k) \omega(\sigma,-k) = \frac{4(k^2+1)}{4k^2 + \sigma(k) \sigma(-k)} \ .$$
\end{enumerate}
\end{proposition} 
\begin{proof}
 By Lemma \ref{lem:hilb.zeta}, the Hilbert transform is a bounded linear operator from $\Hzc$ to $\Hzc$. By Proposition \ref{prop:inv.l} it then follows that  the map 
$$\S^{M,N} \to \Hz \ , \quad \sigma \mapsto \H(l(\sigma,\cdot))$$
 is real analytic as well. Since the exponential function is real analytic and    $\derk \omega(\sigma, \cdot) = \frac{1}{2}\derk (l(\sigma,\cdot) + i \H(l(\sigma, \cdot))) \omega(\sigma,\cdot)$,  one proves by induction that  $\S^{M,N} \to W^M_{\zeta}$, $\sigma \to \omega(\sigma, \cdot)$ is real analytic.
Properties $(i)$--$(iii)$ are proved in \cite[Section 4]{kapptrub}.
\end{proof}
Next we consider the map $\sigma \to \frac{1}{w(\sigma, \cdot)}$. The following proposition follows immediately from Proposition \ref{prop:inv.omega} and the definition $\frac{1}{w(\sigma,k)}=\frac{\omega(\sigma,k)}{2i(k+i)}$.
\begin{proposition}
\label{prop:w.sigma}
The map $\S^{M,N} \to H^{M-1}_\C$, $\sigma \to \frac{1}{w(\sigma, \cdot)}$ is real analytic. Furthermore the maps
$$
\S^{M,N} \to L^2 \ , \qquad \sigma \to  \derk^n \frac{2ik}{w(\sigma, \cdot)} \ , \quad 1 \leq n \leq M 
$$
are real analytic. The function $\frac{1}{w(\sigma, \cdot)}$ fulfills
\begin{enumerate}[(i)]
\item $\overline{\left(\frac{1}{w(\sigma, k)}\right)} = \frac{1}{w(\sigma, -k)}$ for every $k \in \R$.
\item $\mmod{ \frac{2ik}{w(\sigma, k)}} \leq 1 $ for every $k \in \R$.
\item For every $k \in \R$
$$
w(\sigma, k) w(\sigma, -k) = 4k^2 + \sigma(k) \sigma(-k) \ .
$$
In particular $\mmod{w(\sigma, k)} > 0$ for every $k \in \R$ and $\sigma \in \S^{M,N}$.
\end{enumerate}

\end{proposition}

Now we study the properties of $\rho_+(\sigma, \cdot)$ and $\rho_-(\sigma,\cdot)$ defined in formulas \eqref{inv.scatt.elem}.
\begin{proposition}
\label{prop:inv.r}
Let $N \in \Z_{\geq 0}$ and $M \in \Z_{\geq 4}$ be fixed. Then the maps  $\S^{M,N} \to \Hz \cap L^2_N$, $\sigma \to \rho_\pm(\sigma,\cdot)$ are real analytic. There exists $C >0$ so that $\norm{\rho_{\pm}(\sigma, \cdot)}_{\Hzc\cap L^2_N} \leq C \norm{\sigma}_{\Hz\cap L^2_N} $, where $C$ depends locally uniformly on $\sigma \in \S^{M,N}$.  Furthermore the following holds:
\begin{enumerate}[(i)]
\item unitarity: $\tau(\sigma, k) \tau(\sigma,-k) + \rho_\pm(\sigma, k)\rho_\pm(\sigma,-k) = 1$ and  $\rho_+(\sigma,k) \overline{\tau(\sigma,k)} + \overline{\rho_{-}(\sigma,k)} \tau(\sigma,k) = 0$ for every $ k\in \R$ . 
\item reality: $\tau(\sigma,k)= \overline{\tau(\sigma,-k)}, 
\, \rho_\pm(\sigma,k)= \overline{\rho_\pm(\sigma,-k)}$;
\item analyticity: $\tau(\sigma,k)$ admits an analytic extension to $\{ \Im k > 0 \}$; 
\item asymptotics: $\tau(\sigma, z) = 1 + O(1/|z|)$ as $|z| \to \infty$, $\, \Im z \geq 0$, and $\rho_\pm(\sigma,k) = O(1/k)$, as $|k| \to \infty$,  $k$ real;
\item rate at $k=0$: $|\tau(\sigma,z)|>0$ for $z \neq 0$, $\Im z \geq 0$ and  $|\rho_\pm(\sigma,k)|< 1$ for $k  \neq 0$. Furthermore 
\begin{align*}
 \tau(\sigma,z) =& \alpha z + o(z), \quad \alpha\neq 0, \quad \Im z \geq 0 \\ 1+\rho_\pm(\sigma,k) =& \beta_\pm k + o(k),\quad k\in \R; \end{align*} 
\end{enumerate}
\end{proposition}
\begin{proof}
The real analyticity of the maps  $\S^{M,N} \to \Hz \cap L^2_N$, $\sigma \to \rho_\pm(\sigma,\cdot)$ follows from Proposition \ref{prop:w.sigma} and the definition $\rho_\pm(\sigma,k) = \sigma(\mp,k)/w(\sigma,k)$ (see also the proof of Proposition \ref{prop:inv.R}).
Since $\sigma \mapsto \frac{1}{w(\sigma,\cdot)}$ is real analytic, it is  locally bounded, i.e.,   there exists $C >0$ so that $\norm{\rho_{\pm}(\sigma, \cdot)}_{\Hzc\cap L^2_N} \leq C \norm{\sigma}_{\Hz\cap L^2_N} $, where $C$ depends locally uniformly on $\sigma \in \S^{M,N}$.
Properties $(i),(ii), (v)$ follow now by simple computations. 
Property $(iii)-(iv)$ are proved in \cite[Lemma 4.1]{kapptrub}. 
\end{proof}

Finally define the functions
\begin{equation}
R_\pm(\sigma, k) := 2ik \rho_\pm(\sigma,k) \ .
\end{equation}
\begin{proposition}
\label{prop:inv.R}
Let $N \in \Z_{\geq 0}$ and $M \in \Z_{\geq 4}$ be fixed. Then the maps  $\S^{M,N} \to H^M_\C \cap L^2_N$, $\sigma \to R_\pm(\sigma,\cdot)$ are real analytic.
 There exists $C >0$ so that $\norm{R_{\pm}(\sigma, \cdot)}_{H^M_\C\cap L^2_N} \leq C \norm{\sigma}_{\Hz\cap L^2_N} $, where $C$ depends locally uniformly on $\sigma \in \S^{M,N}$. Furthermore the following holds:
\begin{enumerate}[(i)]
\item $\overline{R_\pm(\sigma, k)} = R_\pm(\sigma, -k)$ for every $k \in \R$.
\item $|R_\pm(\sigma, k)| < 2|k|$ for any $k \in \R \setminus \{0 \}$.
\end{enumerate}
\end{proposition}

\begin{proof}
In order to prove the statements, we will use that $R_\pm(\sigma, k) = 2ik \frac{\sigma(\mp k)}{w(\sigma,k)}$. We will consider just $R_-$, since the analysis for $R_+$ is identical.
 To simplify  the notation, we will denote $R_-(\sigma, \cdot) \equiv R(\sigma,\cdot)$.

By Proposition \ref{prop:w.sigma}$(ii)$,   
$\mmod{R(\sigma,k)}\leq \mmod{\sigma(k)}\,$, thus $R(\sigma,\cdot) \in L^2_{N}$.
In order to prove that $R(\sigma, \cdot) \in H^M_\C$,   take  $n$ derivatives ($1 \leq n \leq M$) of $R(\sigma,\cdot)$ to get   the identity  \begin{equation}
\derk^n R(\sigma, k) = \frac{2ik}{w(\sigma,k)} \derk^n \sigma(k) + \sum_{j=1}^{n-1} \binom{n}{j} \left(\derk^j\frac{2ik}{w(\sigma,k)}\right) \ \derk^{n-j} \sigma(k) + \left(\derk^n\frac{2ik}{w(\sigma,k)}\right)  \sigma(k) \ .
\label{derRrecursion}
\end{equation}
We show now that each term of the r.h.s. of the identity above is in $L^2$. Consider first the term $I_1:= \frac{2ik}{w(\sigma,k)} \derk^n \sigma(k)$. If $1 \leq n < M$, then $\derk^n \sigma \in L^2$ and $|2ik/w(\sigma,k)| \leq 1$, thus proving that $I_1 \in L^2$. If $n=M$, let  $\chi$ be a smooth cut-off function  with $\chi(k)\equiv 1$ in $[-1,1]$ and $\chi(k)\equiv 0$ in $\R \setminus [-2,2]$. Then one has 
$$
I_1 = \frac{1}{w(\sigma,k)}\chi(k) 2ik \derk^M \sigma(k) + \frac{ 2ik}{w(\sigma,k)}(1-\chi(k)) \derk^M \sigma(k) \ . $$
As  $\sigma \in \S^{M,N}$ it follows that $k \mapsto \chi(k) 2ik \derk^M \sigma(k)$ and $k \mapsto (1-\chi(k))\derk^M\sigma(k)$ are in $L^2$. By Proposition \ref{prop:w.sigma}, $ \frac{1}{w(\sigma,\cdot)}$ and $\frac{2ik}{w(\sigma, \cdot)}$ are in $L^\infty$. Altogether it follows that $I_1 \in L^2$ for any $1 \leq n \leq M$.

Consider now $I_2 := \sum_{j=1}^{n-1} \binom{n}{j} \left(\derk^j\frac{2ik}{w(\sigma,k)}\right) \ \derk^{n-j} \sigma(k)$. By Proposition \ref{prop:w.sigma}, $\left(\derk^{j} \frac{2ik}{w(\sigma,k)}\right) \in H^{1}_\C$ for every $1 \leq j \leq M-1$, thus by the Sobolev embedding theorem $\left(\derk^j\frac{2ik}{w(\sigma,k)}\right) \in L^{\infty}$ for every $1 \leq j \leq M-1$. As $\derk^{n-j} \sigma \in L^2$ for $1 \leq j \leq n-1 < M$, it follows that $I_2 \in L^2$ for any $1 \leq n \leq M$.

Finally consider $I_3:=\left(\derk^n\frac{2ik}{w(\sigma,k)}\right)  \sigma(k)$. By Proposition \ref{prop:w.sigma},  $\left(\derk^n\frac{2ik}{w(\sigma,k)}\right) \in L^2$ for any $1 \leq n \leq M$. Since  $\sigma \in L^\infty$,  $I_3 \in L^2$ for any $1 \leq n \leq M$.

Altogether we proved that $R(\sigma, \cdot) \in H^M_\C \cap L^2_N$. The claimed estimate on $\norm{R(\sigma,\cdot)}_{H^M_\C \cap L^2_N}$ and item  $(i)$ and $(ii)$ follow in a straightforward way. The real analyticity of the map $\S^{M,N} \to H^M_\C \cap L^2_N$, $\sigma \to R(\sigma,\cdot)$ follows by Proposition \ref{prop:w.sigma}.
\end{proof}

For $\sigma \in \S^{M,N}$, define  the Fourier transforms
\begin{equation}
\label{F.four}
F_{\pm}(\sigma, y) := \F_\pm^{-1}(\rho_\pm(\sigma, \cdot))(y) =  \frac{1}{\pi}\int_{\R} \rho_{\pm}(\sigma, k) e^{\pm 2ik y} dk \ . 
\end{equation}
 Then
 \begin{equation}
\label{FR2}
\pm \dery F_{\pm} (\sigma, y) = \frac{1}{\pi} \intii 2ik \rho_\pm(\sigma, k) e^{\pm 2iky} \; dk = \F_\pm^{-1}(R_\pm(\sigma, \cdot ))(y) \ .
\end{equation}
In the next proposition we analyze the properties of the maps $\sigma \mapsto F_{\pm}(\sigma, \cdot)$.
\begin{proposition}
\label{rem:dec_rel}
Let $N \in \Z_{\geq 0}$ and $M \in \Z_{\geq 4}$ be fixed. Then the following holds true:
\begin{enumerate}
\item[(i)] $\sigma \mapsto F_{\pm}(\sigma, \cdot)$ are real analytic as  maps from $\S^{4,0}$ to $H^1 \cap L^2_3$. Moreover there exists $C >0$ so that $\norm{F_{\pm}(\sigma, \cdot)}_{H^1\cap L^2_3} \leq C \norm{\sigma}_{\Hz} $, where $C$ depends locally uniformly on  $\sigma \in \S^{M,N}$. 
\item[(ii)]  $\sigma \mapsto F'_{\pm}(\sigma, \cdot)$ are real analytic as  maps from $\S^{M,N}$ to $H^N \cap L^2_M$. Moreover there exists $C' >0$ so that $\norm{F'_{\pm}(\sigma,\cdot)}_{H^N\cap L^2_M} \leq C' \norm{\sigma}_{\Hz \cap L^2_{N}} $, where $C'$ depends locally uniformly on  $\sigma \in \S^{M,N}$. 
\end{enumerate}
\end{proposition}
\begin{proof} 
By Proposition  \ref{prop:inv.r},   the map $\S^{4,0}\to H^3_\C \cap L^2_1$, $\sigma \to \rho_\pm(\sigma, \cdot)$ is real analytic. Thus item $(i)$  follows by  the properties of the Fourier transform. By Proposition \ref{prop:inv.r} $(ii)$,   $F_{\pm}(\sigma,\cdot) = \F_\pm^{-1}(\rho_\pm)$ is real valued. 
Item $(ii)$ follows from  \eqref{FR2} and the characterizations
\begin{equation}
\label{cor.R.F.equiv}
R_\pm \in H^{M}_\C \ \Longleftrightarrow \ \F_{\pm}^{-1}(R_\pm) \in L^2_M  \qquad \mbox{and} \qquad R_\pm \in L^2_{N} \ \Longleftrightarrow \ \F_{\pm}^{-1}(R_\pm) \in H^N_\C \ .
\end{equation}
The claimed estimates follow from  the properties of the Fourier transform, Proposition \ref{prop:inv.r} and Proposition \ref{prop:inv.R}.
\end{proof}

We are finally able to prove that there exists a potential $q \in \class$ with prescribed scattering coefficient $\sigma \in \S^{M,N}$. More precisely the following theorem holds.
\begin{theorem}
\label{S.onto}
Let $N \in \Z_{\geq 0}$, $M \in \Z_{\geq 4}$ and  $\sigma \in \S^{M,N}$ be fixed. Then there exists a potential $q \in \class$ such that $S(q, \cdot) = \sigma$.
\end{theorem}
\begin{proof}
Let $\rho_\pm:=\rho_\pm(\sigma, \cdot)$ and $\tau:= \tau(\sigma, \cdot)$ be given by formula \eqref{inv.scatt.elem}. Let  $F_\pm(\sigma, \cdot)$ be defined as in \eqref{F.four}. By Proposition \ref{rem:dec_rel}  it follows that  $F_\pm(\sigma, \cdot)$  are absolutely continuous and $F'_\pm(\sigma, \cdot) \in H^N \cap L^2_M$. As $M \geq 4$ it follows that 
\begin{equation}
\label{F'.L1}
\int_{-\infty}^\infty (1+x^2) |F'_\pm (\sigma, x)| \, dx  <\infty \ .
\end{equation}
 The main theorem in inverse scattering \cite{faddeev} assures that if \eqref{F'.L1} and item $(i)$--$(v)$ of Proposition \ref{prop:inv.R} hold, then  there exists a potential $q \in \class$ such that $r_\pm(q,\cdot) = \rho_\pm$ and $t(q,\cdot) = \tau$, where $r_\pm$ and $t$ are the reflection  respectively transmission coefficients defined in \eqref{r.S.rel}. From the  formulas \eqref{inv.scatt.elem} it follows that $S(q,\cdot) = \sigma$. 
\end{proof}

It remains to show that $q \in \class^{N,M}$ and that the map $S^{-1}: \S^{M,N} \to \class^{N,M}$ is real analytic. We take here a different approach then \cite{kapptrub}. In \cite{kapptrub} the authors show that the map $S$ is complex differentiable and its differential $d_qS$ is bounded invertible.
 Here instead  we reconstruct $q$ by solving the Gelfand-Levitan-Marchenko equations and we show that the inverse map $\S^{M,N} \to \class^{N,M} $,  $\sigma \mapsto q$ is real analytic. We outline briefly the procedure.  Given two reflection coefficients $\rho_\pm$ satisfying items $(i)$--$(v)$ of Proposition \ref{prop:inv.r} and arbitrary real numbers $c_+ \leq  c_-$, it is possible to construct a  potential $q_+$  on  $[c_+, \infty)$ using $\rho_+$ and a potential $q_-$  on  $(-\infty, c_-]$ using $\rho_-$, such that $q_+$ and $q_-$  coincide on the intersection of their domains, i.e.,  $\left.q_+\right|_{[c_+, c_-]}= \left.q_-\right|_{[c_+, c_-]}$. 
 Hence $q$ defined on $\R$ by  $q\vert_{[c_+, + \infty)} = q_+$ and $q\vert_{(-\infty, c_-]} = q_-$ is well defined,  $q \in \class$ and  $r_\pm (q, \cdot) = \rho_\pm$,  i.e., $\rho_+$ and $\rho_-$ are the reflection coefficients of the  potential $q$ \cite{faddeev,Masturm,deift}. 
We postpone the details of this procedure to the next section.

\subsection{  Gelfand-Levitan-Marchenko equation}
In this section  we  prove how to construct for any $\sigma \in \S^{M,N}$ two potentials $q_+$ and $q_-$  with  $q_+ \in H^N_{x\geq c} \cap L^2_{M, x \geq c}$ respectively $q_- \in H^N_{x \leq c} \cap L^2_{M, x \leq c}$, where for any $c \in \R$ and $1 \leq p \leq \infty$ 
\begin{equation}
L^p_{x \geq c}:= \left\{f : [c, +\infty) \to \C: \, \norm{f}_{L^p_{x\geq c}}< \infty \right\} \ ,
\end{equation}
where $\norm{f}_{L^p_{x\geq c}}:= \left(\int_c^{+\infty} |f(x)|^p \, dx \right)^{1/p} $ for $1 \leq p < \infty$ and  $\norm{f}_{\Lic} := \esssup_{x \geq c} |f(x)|$. For any integer $N \geq 1$ define 
\begin{equation}
\label{H^N_c} 
H^N_{x \geq c }:= \left\{ f: [c, +\infty) \to \R: \,  \norm{f}_{H^N_{x \geq c }} < \infty \right\}, \qquad   \norm{f}_{H^N_{x \geq c }}^2 := \sum_{j=0}^N \norm{\derx^n f}_{L^2_{x \geq c}}^2,
\end{equation}
 and for any real number $M \geq 1$ define
\begin{equation}
\label{L^M_c} 
L^2_{M, x \geq c }:= \left\{ f: [c, +\infty) \to \C: \,   \norm{f}_{L^2_{M,x  \geq c }} < \infty \right\}, \qquad
\norm{f}_{L^2_{M, x \geq c }}= \norm{\la x \ra^M  f}_{L^2_{x \geq c}},
\end{equation}
where $\x := (1 + x^2)^{1/2}$.
We will write $H^N_{\C, x \geq c}$  for the complexification of  $H^N_{x \geq c }$. 
For $1\leq \alpha, \beta \leq \infty$, we define  
$$L^\alpha_{x \geq c}L^\beta_{y \geq 0}:=\left\{ f: [c, +\infty) \times [0, +\infty) \rightarrow \C: \,  \norm{f}_{L^\alpha_{x \geq c} L^\beta_{y \geq 0}} < \infty \right\} \ ,
$$ 
where  $\norm{f}_{L^\alpha_{x \geq c}L^\beta_{y \geq 0}} := \Big(  \int_{c}^{+\infty} \norm{f(x, \cdot)}^\alpha_{L^\beta_{y \geq 0}} \,dx \Big)^{1/\alpha}.$
Analogously one defines the spaces $L^p_{x \leq c}$, $H^N_{x\leq c}$, $L^2_{M, x\leq c}$ and $L^\alpha_{x \leq c} L^\beta_{y \leq 0}$, \textit{mutatis mutandis}. \\
Let us denote  by $\Cyz := C^0([0, \infty), \C)$ and by $\Cxy := C^0([c, \infty)\times [0, \infty), \C)$. Finally we denote by $\CxLy := C^0([c, \infty), \Lyz)$ the set of continuous functions on $[c, \infty)$ taking  value in $\Lyz$.
\vspace{1em}\\
The potentials  $q_+$ and $q_-$ mentioned at  the beginning of this section are constructed by  solving an integral equation, known in literature as the \textit{Gelfand-Levitan-Marchenko equation}, which  we are now going to described in more detail. 

Given $\sigma \in \S$,  define the functions $F_\pm(\sigma, \cdot)$ as in \eqref{F.four}. See Proposition \ref{rem:dec_rel} for the analytical properties of the maps $\sigma \to F_\pm(\sigma, \cdot)$. To have a more compact notation, in the following we will  denote $F_{\pm, \sigma} := F_\pm(\sigma, \cdot)$.  
\begin{remark} From  the decay properties of $F'_{\pm, \sigma}$ one deduces corresponding  decay properties of $F_{\pm, \sigma}$. Indeed one has 
\begin{equation}
\label{dec_rel}
\begin{aligned}
\x^{m} \, F_\pm' \in L^2_{x \geq c} \Rightarrow \x^{m-1}  F_\pm' \in L^1_{ x \geq c} \Rightarrow x^{m-2} F_\pm \in L^1_{x \geq c} \ , \quad \forall \, m \geq 2 \ .
\end{aligned}
\end{equation}
\end{remark}
\vspace{1.5em}
The  Gelfand-Levitan-Marchenko equations  are the integral equations given by
\begin{align}
 \label{ME}
F_{+, \sigma}(x+y)+E_{+, \sigma}(x,y) +  \intzi F_{+, \sigma}(x+y+z)E_{+, \sigma}(x,z) dz &= 0, \qquad y \geq 0 \\
 \label{ME-}F_{-, \sigma}(x+y)+ E_{-, \sigma}(x,y) +  \intiz F_{-, \sigma}(x+y+z) E_{-, \sigma}(x,z) dz &= 0, \qquad  y \leq 0 
 \end{align}
where $ E_{\pm, \sigma}(x,y)$ are the unknown functions and $F_{\pm, \sigma}$ are given and  uniquely determined by $\sigma$ through formula \eqref{F.four}.  If \eqref{ME} and \eqref{ME-} have  solutions with enough regularity, then one defines the potentials $q_+$ and $q_-$ through the well-known formula --  \cite{faddeev}
\begin{equation}
\label{ainversedefinition}
\begin{aligned}
&q_+(x) = - \partial_xE_{+, \sigma}(x,0), \quad \forall \, c_+ \leq x <  \infty \ , \qquad q_-(x) =  \partial_x E_{-, \sigma}(x,0), \quad \forall \, -\infty < x \leq c_- \ .
\end{aligned}
\end{equation}
The main purpose of this section is to study the maps $\scat_{\pm, c} $ defined by
 \begin{align}
 \label{ainversedefinition1}
\sigma \mapsto \scat_{\pm,c}(\sigma), \qquad \scat_{\pm,c}(\sigma)(x):= \mp\partial_x E_{\pm, \sigma}(x,0), \qquad x \in [c, \pm \infty) \ .
\end{align} 
\begin{theorem}
\label{inversemain1}
  Fix $N \in \Z_{\geq 0}$, $M\in \Z_{\geq 4}$ and $c\in \mathbb{R}$. Then the  maps   $\scat_{+,c}$ $\, [\scat_{-,c}]$ are well defined on $  \S^{M, N}$ and take values in  $H^N_{x \geq c} \cap L^2_{M, x \geq c}$ $\, [ H^N_{x \leq c} \cap L^2_{M, x \leq c} ]$. As such they  are real analytic.
\end{theorem} 
In order to prove Theorem  \ref{inversemain1}  we look for  solutions of \eqref{ME} and \eqref{ME-} of the form  
\begin{equation}
\label{decomp}
E_{\pm, \sigma}(x,y) \equiv - F_{\pm, \sigma}(x+y) + B_{\pm, \sigma}(x,y)
\end{equation}
where $ B_{\pm, \sigma}(x,y)$ are  to be determined. Inserting the ansatz \eqref{decomp} into the Gelfand-Levitan-Marchenko equations \eqref{ME}, \eqref{ME-}, one  gets 
\begin{align}
\label{bGLM}
& B_{+, \sigma}(x,y) +  \intzi F_{+, \sigma}(x+y+z) B_{+, \sigma}(x,z) dz = \intzi F_{+, \sigma}(x+y+z)F_{+, \sigma}(x+z) \, dz, \qquad y \geq 0 \ ,\\
\label{bGLM2}
&  B_{-, \sigma}(x,y) +  \intiz F_{-, \sigma}(x+y+z) B_{-, \sigma}(x,z) dz = \intiz    F_{-, \sigma}(x+y+z) F_{-, \sigma}(x+z)\, dz, \qquad  y \leq 0 .
 \end{align}
We will prove in Lemma \ref{prop:AnBn} below  that  there exists  a solution $B_{+, \sigma}$ of  \eqref{bGLM} and a solution $B_{-, \sigma}$ of  \eqref{bGLM2} with $\derx B_{+, \sigma}(\cdot, 0) \in H^{1}_{x \geq c}$ respectively $\derx B_{-, \sigma}(\cdot, 0) \in H^{1}_{x \leq c}$. By   \eqref{ainversedefinition} we get therefore 
\begin{equation}
\label{structureS^-1}
q_{+} =  \derx F_{+, \sigma} - \derx B_{+, \sigma}(\cdot,0) \ \ \  \forall \, c \leq x < \infty ,  \qquad   q_{-} =  -\derx F_{-, \sigma} + \derx B_{-, \sigma}(\cdot,0) \ \ \  \forall \, -\infty < x \leq c \ .
\end{equation}
Define the maps  $$\mathcal{B}_{\pm,c}: \sigma \mapsto \B_{\pm,c}(\sigma)$$ as 
\begin{equation}
\B_{+, c}(\sigma)(x) := - \derx B_{+, \sigma}(x, 0) \quad \forall \, x \geq c  \qquad \mbox{and} \qquad  \B_{-, c}(\sigma)(x) :=  \derx B_{-, \sigma}(x, 0) \quad \forall \, x \leq c \ ,
\label{def_Binv}
\end{equation}
with $B_{\pm, \sigma}(x,y):= E_{\pm, \sigma}(x,y) + F_{\pm,\sigma}(x,y)$ as in \eqref{decomp}.
Now we study  analytic properties of the maps $\B_{\pm , c}$ in case the scattering coefficient $\sigma$ belongs to  $\S^{4,N}$ with arbitrary $N \in \Z_{\geq 0}$. Later we will treat the case where $\sigma \in \S^{M,0}$, $M \in \Z_{\geq 4}$.
\begin{proposition}
\label{Bsmoothing} Fix $N \in \Z_{\geq 0}$ and $c \in \R$. Then  $\B_{+, c}$ $\, [\B_{-, c}]$ is real analytic as a map from 
$ \S^{4, N}$ to $H^{N}_{ x \geq c}$ $\, [H^{N}_{ x \leq c}]$. Moreover
$$\norm{\B_{+, c}(\sigma)}_{H^{N}_{ x \geq c}} \ , \ \norm{\B_{-, c}(\sigma)}_{H^{N}_{ x \leq c}} \ \leq K \norm{\sigma}_{H^4_{\zeta,\C} \cap L^2_{N}}^2  $$ 
where $K>0$ is a constant which can be chosen locally uniformly in $\sigma \in \S^{4,N}$.
\end{proposition}
 The main ingredient of the proof of Proposition \ref{Bsmoothing} is a detailed analysis of the solutions of the integral equations  \eqref{bGLM}-\eqref{bGLM2}, which we rewrite  as 
\begin{align}
\label{opma1} 
& \left(Id + \K_{x,\sigma}^{\pm} \right)[B_{\pm, \sigma}(x,\cdot)](y) = f_{\pm, \sigma}(x,y)
\end{align}
where for every $x \in \R$ fixed, the  two operators $\K_{x, \sigma}^+ : L^2_{y \geq 0} \to L^2_{y \geq 0}$ and $\K_{x, \sigma}^- : L^2_{y \leq 0} \to L^2_{y \leq 0}$ are defined by
\begin{align}
\label{kxdef}
  \K_{x, \sigma}^{+} \, [f](y):=& \intzi F_{+, \sigma}(x+y+z) f(z) \,dz \ , &f\in L^2_{y \geq 0} \ ,\\
  \K_{x,\sigma}^{-} \, [f](y) := &\int\limits_{-\infty}^0 F_{-, \sigma}(x+y+z) f(z) \,dz \ , & f\in L^2_{y \leq 0} \ ,
 \end{align}
 and the functions $f_{\pm, \sigma}$ are defined by 
\begin{equation}
\label{f.def}
f_{\pm, \sigma}(x,y):= \pm \int\limits_0^{\pm \infty} F_{\pm, \sigma}(x+y+z) F_{\pm, \sigma}(x+z) \, dz \ .
\end{equation}

 As the claimed statements for  $\B_{+, c}$ and  $\B_{-,c}$ can be proved in a similar way we consider $\B_{+, c}$ only.  To simplify notation, in the following we  omit  the subscript  $"+"$. In particular we write  $B_{\sigma} \equiv B_{+, \sigma}$, $F_{\sigma} \equiv F_{+, \sigma}$, $f_{\sigma} \equiv f_{+, \sigma}$ and $\K_{x,\sigma}\equiv \K_{x,\sigma}^+$.
\vspace{1em}\\
We give the following definition: a function $h_\sigma:[c, \infty) \times [0, \infty) \to \R$, which depends on $\sigma  \in \S^{4,N}$,  will be said to satisfy $(P)$ if the following holds true:
\begin{enumerate}
\item[$(P1)$] $ h_\sigma \in \CxLy \cap \Lxyc \cap  \Cxy$. Finally $h_\sigma (\cdot,0) \in \Lc$.
 \item[$(P2)$] There exists a  constant $K_c>0$, which  depends locally uniformly on $\sigma\in  H^4_{\zeta,\C} \cap L^2_{N} $, such that 
 \begin{align}
 \label{prop:AnBn:est}
\norm{ h_\sigma}_{\Lxyc} + \norm{ h_\sigma(\cdot, 0)}_{\Lc} \leq K_c \norm{\sigma}_{H^4_{\zeta,\C} \cap L^2_{N}}^2 . 
\end{align}
 \item[$(P3)$] $\sigma \mapsto  h_\sigma$ $\, [\sigma \mapsto  h_\sigma(\cdot,0)]$ is real analytic as a map from $\S^{4,N}$ to $\Lxyc$  $[\Lc]$. 
\end{enumerate}

We have the following lemma:
\begin{lemma}
\label{prop:AnBn}
 Fix $N \geq 0$ and $c \in \R$. For every $ \sigma \in \S^{4,N}$ 
equation \eqref{bGLM} has a unique solution  $B_\sigma \in \CxLy \cap \Lxyc$. Moreover for all integers $n_1,  n_2\geq 0$ with   $  n_1 + n_2 \leq N+1$ , the function $ \derx^{n_1}\dery^{n_2}B_\sigma$ satisfies $(P)$. 
 \end{lemma}
{\em Proof.}
 Let $N \in \Z_{\geq 0}$ and $c \in \R$ be fixed. The proof is by induction on $j_1 + j_2 = n$, $0 \leq n \leq N$. 
For each $n$ we prove that $\derx^{j_1}\dery^{j_2}B_\sigma$ and its derivatives $\derx^{j_1+1}\dery^{j_2}B_\sigma$, $\derx^{j_1}\dery^{j_2+1}B_\sigma$ satisfy $(P)$. Thus the claim follows.

{\em Case $n=0$. } Then $j_1 = j_2 = 0$. We need to prove existence and uniqueness of the solution of  equation \eqref{opma1}. 
By Lemma \ref{f.prop} [Proposition \ref{rem:dec_rel} and Lemma \ref{lem:Fx}] the function $f_\sigma$ and its derivatives $\derx f_\sigma,$ $\, \dery f_\sigma$ $\, [F_\sigma]$ satisfy assumption $(P)$  [$(H)$-- cf  Appendix \ref{app:ab.eq}]. 
 Thus  by  Lemma \ref{lem:ab.eq2} $(i)$ it follows that $B_\sigma =\left(Id + \K_\sigma \right)^{-1}f_\sigma$ and its derivatives $\derx B_\sigma$, $\dery B_\sigma$ satisfy $(P)$.
%

Note that if $N=0$ the lemma is proved.  Thus in the following we assume  $N \geq 1$.

{\em Case $n-1 \leadsto n $.} Let $j_1 + j_2 = n$.  By the induction assumption we already know that $\derx^{j_1}\dery^{j_2} B_\sigma $ satisfies $(P)$.
By  Lemma \ref{lem:ab.eq2} it follows that  $\derx^{j_1}\dery^{j_2} B_\sigma $ satisfies
\begin{equation}
\begin{cases}
\label{int.eq.1}
(Id + \K_{x, \sigma})[ \derx^n B_\sigma(x,\cdot)](y) = f_\sigma^{n,0} (x,y)  & \mbox{ if } j_2 = 0 \ , \\
\derx^{j_1} \dery^{j_2} B_\sigma(x,y) = f_\sigma^{j_1, j_2} (x,y)  & \mbox{ if } j_2 > 0 \ ,
\end{cases}
\end{equation}
where 
\begin{equation}
\begin{aligned}
\label{f_n0}
& f_\sigma^{n,0} (x,y) := \derx^{n}f_\sigma(x,y) -\sum_{l=1}^{n}\binom{n}{l} \intzi \derx^l F_\sigma(x+y+z) \, \derx^{n-l}B_\sigma(x,z) \, dz \ , \\
& f_\sigma^{j_1, j_2} (x,y) := \derx^{j_1}\dery^{j_2} f_\sigma(x,y) -\sum_{l=0}^{j_1}\binom{j_1}{l} \intzi \derz^{j_2 + l} F_\sigma(x+y+z) \, \derx^{j_1-l}B_\sigma(x,z) \, dz \ .
\end{aligned}
\end{equation}
In order to prove the induction step, we show in Lemma \ref{f_sigma.P} that  for any $j_1+j_2 = n$, $0 \leq n \leq N$, $f_\sigma^{j_1, j_2}$ and its derivatives $\dery f_\sigma^{j_1, j_2}$, $\derx f_\sigma^{j_1, j_2}$ satisfy $(P)$. In view of identities  \eqref{int.eq.1} and Lemma  \ref{lem:ab.eq2} $(i)$, it follows  that $\derx^{j_1} \dery^{j_2} B_\sigma$ and its  derivatives $\derx^{j_1+1} \dery^{j_2} B_\sigma$ and $\derx^{j_1} \dery^{j_2+1} B_\sigma$ satisfy $(P)$, thus proving the induction step.
\qed

\vspace{.5em}
Lemma \ref{prop:AnBn} implies in a straightforward way Proposition \ref{Bsmoothing}.
\begin{proof}[Proof of Proposition \ref{Bsmoothing}]
By Lemma \ref{prop:AnBn}, $\derx^{n} B_\sigma$ satisfies $(P)$ for every $1 \leq n \leq N+1$. In particular for every $1 \leq n \leq N+1$, $\sigma \mapsto \derx^{n} B_\sigma(\cdot,0)$ is real analytic as a map from $\S^{4,N}$ to $\Lc$ and $\norm{ \derx^{n} B_\sigma(\cdot, 0)}_{\Lc} \leq K_c \norm{\sigma}_{H^4_{\zeta,\C} \cap L^2_{N}}^2$. Thus 
the map $\sigma \to -\derx B_\sigma(\cdot,0)$ is real analytic as a map from $\S^{4,N}$ to $H^N_{x \geq c}$. The claimed estimate follows in a straightforward way.
\end{proof}

\vspace{2em}

In the next result we study the case $\sigma \in \S^{M, 0}$ for arbitrary $M \geq 4$.
\begin{proposition}
\label{decayinverse}
Fix $M \in \Z_{\geq 4}$ and    $c \in \R$. For any  $\sigma \in  \S^{M, 0}$ the equations \eqref{ME} and \eqref{ME-} admit solutions $E_{\pm, \sigma}$. The maps $\scat_{+, c}$ $\, [\scat_{-, c}]$, defined by \eqref{ainversedefinition1},  are real analytic as maps from $\S^{M,0}$ to $L^2_{M, x \geq c}$ $\,[L^2_{M, x \leq c}]$. Moreover
$\norm{\scat_{+, c}(\sigma)}_{ L^2_{M, x \geq c}}\ , \ \norm{\scat_{-, c}(\sigma)}_{ L^2_{M, x \leq c}} \leq K_c \norm{\sigma}_{H^M_{\zeta,\C}},$ where $K_c>0$ can be chosen locally uniformly in $\sigma \in \S^{M,0}$.
\end{proposition}

\begin{proof} We prove the result just for  $\scat_{+,c}$, since for  $\scat_{-,c}$ the proof is analogous. As before, we suppress the subscript ''+'' from the various objects. \\
Consider the Gelfand-Levitan-Marchenko equation \eqref{ME}.  Multiply it  by $\x^{M-3/2}$ to obtain
\begin{equation}
\left(Id + \K_{x,\sigma} \right) \left[ \x^{M-3/2} E_\sigma(x,y)\right] = - \x^{M-3/2} F_\sigma(x+y).
\label{lineD.1}
\end{equation}
The function  
$$
h_\sigma(x,y) := - \x^{M-3/2} F_\sigma(x+y) \ ,
$$
satisfies $h_\sigma(x, \cdot) \in \Lyz$ and one checks that  $h_\sigma \in \CxLy \cap \Cxy$. We show now that $h_\sigma \in \Lxyc$.  By Lemma \ref{lem:techLemma} $(A3)$ and Proposition  \ref{rem:dec_rel} for $N=0$ it follows that 
 $$\norm{\x^{M-3/2}h_\sigma}_{\Lxyc}^2\leq K_c \intx{c} \x^{2M-2} |F_\sigma(x)|^2 \, dx \leq K_c  \norm{\x^M F_\sigma'}_{\Lc}^2 \leq K_c \norm{\sigma}_{H^M_{\zeta,\C}}^2 \ .$$ 
Consider now $h_\sigma(x, 0) =  - \x^{M-3/2} F_\sigma(x)$. By \eqref{dec_rel} it follows that $h_\sigma(\cdot, 0) \in \Lc$.
Finally the map $\sigma \mapsto h_\sigma$ $\,[\sigma \mapsto h_\sigma(\cdot, 0)]$ is real analytic as a map from $\S^{M,0}$ to $\Lxyc$ $[L^2_{M-3/2, x\geq c}]$. 
Proceeding as in the proof of Lemma \ref{lem:ab.eq}, one shows that
  there exists a solution $E_\sigma$ of equation \eqref{ME} which satisfies 
$(i)$ $\x^{M-3/2} E_\sigma \in \CxLy \cap \Lxyc$, $\x^{M-3/2} E_\sigma(x, \cdot) \in \Cyz$, $\langle \cdot \rangle^{M-3/2}E_\sigma(\cdot, 0) \in \Lc$, $(ii)$ $ \norm{\x^{M-3/2} E_\sigma}_{\Lxyc} \leq K_c\norm{\sigma}_{H^M_{\zeta,\C}}$, $(iii)$ $ \sigma \mapsto \x^{M-3/2} E_\sigma$ $\,[ \sigma \mapsto   E_\sigma(\cdot,0)]$ is real analytic as a map from $ \S^{M,0}$ to  $\Lxyc$ $[L^2_{M-3/2, x\geq c}]$.
Furthermore its derivative  $\derx E_\sigma$  satisfies the integral equation
\begin{equation}
\label{derxB1}
(Id + \K_{x,\sigma})\left(\derx E_\sigma(x,y)\right) = -   F_\sigma'(x+y) -  \intzi  F_\sigma'(x+y+z)\; E_\sigma(x,z)\; dz .
\end{equation}
Multiply the equation above by $\x^{M-3/2}$, to obtain $(Id + \K_\sigma)\left(\x^{M-3/2}\derx E_\sigma \right) = \tilde h_\sigma$, where
\begin{equation}
\label{tilde.h.R} 
\tilde h_\sigma(x,y) := -\x^{M-3/2}h_\sigma'(x,y) - \intzi F_\sigma'(x+y+z)\, \x^{M-3/2} E_\sigma(x,z) \, dz  \ .
\end{equation}
where $h_\sigma'(x,y):= F_\sigma'(x+y)$.
We claim that $\tilde h_\sigma \in \Lxyc$ and $\sigma \mapsto \tilde h_\sigma$ is real analytic as a map $\S^{M,0} \to \Lxyc$. By Lemma \ref{lem:techLemma} $(A0)$  the first term of \eqref{tilde.h.R} satisfies 
$$\norm{\x^{M-3/2} h_\sigma'}_{\Lxyc} \leq K_c  \norm{\x^{M-1} F_\sigma'}_{\Lc} \leq K_c \norm{\sigma}_{H^M_{\zeta,\C}} \ , $$
and  by Lemma \ref{lem:techLemma} $(A1)$ the second term of \eqref{tilde.h.R} satisfies
$$
\norm{\intzi F_\sigma'(x+y+z)\, \x^{M-3/2} E_\sigma(x,z) \, dz}_{\Lxyc} \leq \norm{F_\sigma'}_{L^1} \norm{\x^{M-3/2} E_\sigma}_{\Lxyc} \leq K_c \norm{\sigma}_{H^M_{\zeta,\C}}^2 \ .
$$
Moreover $\sigma \mapsto \tilde h_\sigma$ is real analytic as a map from $\S^{M,0}$ to $\Lxyc$,  being composition of real analytic maps.

Thus, by  Lemma \ref{lem:ab.eq},  it follows that $\x^{M-3/2}\derx E_\sigma \in \Lxyc$,  $\norm{\x^{M-3/2}\derx E_\sigma}_{\Lxyc} \leq K_c \norm{\sigma}_{H^M_{\zeta,\C}}$ and $\sigma \mapsto \langle \cdot \rangle^{M-3/2}\derx E_\sigma$ is real analytic as a map from $\S^{M,0}$ to $\Lxyc$ .\\
Consider now equation \eqref{ME}. Evaluate it at $y=0$ to get 
$$
E_\sigma(x,0) = -F_\sigma(x) - \intzi F_\sigma(x+z) E_\sigma(x,z) \, dz \ .
$$
Take the $x$-derivative of the equation above and multiply it by $\x^{M}$ to obtain 
\begin{align*}
\x^M \derx E_\sigma(x,0)  = & - \x^M F_\sigma'(x) - \intzi \x^{3/2} F_\sigma'(x+z)\x^{M-3/2} E_\sigma(x,z) \, dz \\
& \quad - \intzi \x^{3/2} F_\sigma(x+z) \x^{M-3/2}\derx E_\sigma(x,z) \, dz \ .
\end{align*}
We prove now that  $\derx E_\sigma(\cdot,0) \in L^2_{M, x \geq c}$ and $\sigma \mapsto \derx E_\sigma(\cdot,0)$ is real analytic as a map from $\S^{M,0}$ to $L^2_{M, x \geq c}$. 
The result follows by Proposition \ref{rem:dec_rel} and Lemma \ref{lem:techLemma} $(A2)$. Indeed 
one has that
$\sigma \mapsto F_\sigma'$ $\, [\sigma \mapsto F_\sigma]$ is real analytic as a map from $\S^{M,0}$ to $L^2_M$ $\,[L^2_{3/2} ]$, and we proved above that  
 $\sigma \mapsto \langle \cdot \rangle^{M-3/2} E_\sigma$ and $\sigma \mapsto  \langle \cdot \rangle^{M-3/2}\derx E_\sigma $   are real analytic as maps from $\S^{M,0}$ to $\Lxyc$.
\end{proof}

Combining the results of Proposition \ref{Bsmoothing} and Proposition \ref{decayinverse},  we can prove  Theorem \ref{inversemain1}.
\vspace{.5em}\\
\noindent{\em Proof of Theorem \ref{inversemain1}.} 
It follows from Proposition \ref{rem:dec_rel}, Proposition \ref{Bsmoothing} and Proposition \ref{decayinverse} by restricting the scattering maps $\scat_{\pm,c}$ to the spaces $\S^{M,N} = \S^{4,N }\cap \S^{M,0}$.
\qed

\vspace{.5em}
Using the results of Theorem \ref{inversemain1} and Theorem \ref{S.onto} we can prove Theorem \ref{thm:inv.scat}, showing that $S^{-1}: \S^{N,M} \to \class^{N,M}$ is real analytic.

\begin{proof}[Proof of Theorem \ref{thm:inv.scat}]
Let $\sigma \in \S^{M,N}$. By Theorem \ref{S.onto} there exists $q \in \class$ with $S(q,\cdot) = \sigma$.  Now let  $c_{+}\leq c \leq c_{-}$ be arbitrary real numbers and consider $\scat_{+,c_+}(\sigma)$ and $\scat_{-,c_-}(\sigma)$, where $\scat_{\pm, c_\pm}$ are defined in \eqref{ainversedefinition1}. By classical inverse scattering theory \cite{faddeev}, \cite{Masturm} the following holds:
\begin{enumerate}[(i)]
\item $\left.\scat_{+,c_+}(\sigma)\right|_{x \in [c_+, c]} = \left.\scat_{-,c_-}(\sigma)\right|_{x \in [c, c_-]} \ ,$
\item the potential $q_c$ defined by 
\begin{equation}
\label{q_a}
q_c:= \scat_{+,c_+}(\sigma) \mathbbm{1}_{[c,\infty)}+\scat_{-,c_-}(\sigma) \mathbbm{1}_{(-\infty, c]}
\end{equation}
is in  $\class $ and  satisfies  
$r_+(q_c,\cdot)=\rho_{+}(\sigma,\cdot) $, $r_-(q_c,\cdot)=\rho_-(\sigma,\cdot)$ and $t(q_c,\cdot) = \tau(\sigma,\cdot)$. Thus by formulas \eqref{r.S.rel} and \eqref{inv.scatt.elem} it follows that $S(q_c,\cdot) = \sigma$.
\end{enumerate}
Since $S$ is 1-1  it follows that $q_c \equiv q$.
Finally, by  Theorem \ref{inversemain1}, $\S^{M,N} \to H^{N}_{x \geq c_+} \cap L^2_{M, x \geq c_+},$ $\sigma \mapsto \scat_{+,c_+}(\sigma)$ and $\S^{M,N} \to  H^{N}_{x \leq c_-}\cap L^2_{M, x \leq c_-},$ $\sigma \mapsto \scat_{-,c_-}(\sigma)$ are real analytic. It follows that $q \in H^N \cap L^2_M $ and the map $S^{-1}: \sigma \to q$ is real analytic. 
\end{proof}

 \section{Proof of Corollary \ref{thm:actions} and Theorem \ref{firstapprox}} 
 This section is devoted to the proof of  Corollary \ref{thm:actions} and Theorem \ref{firstapprox}. Both results are  easy applications of Theorem \ref{reflthm}.

\vspace{1em}
\noindent {\em Proof of Corollary \ref{thm:actions}}. Let $N \geq 0$, $M \geq 4$ be fixed integers.  Fix $q \in \class^{N,M}$. 
By Theorem \ref{reflthm} the scattering map $S(q,\cdot)$ is in $\S^{M,N}$. Furthermore by the definition \eqref{action_angle} of $I(q,k)$ there exists a constant $C >0$ such that for any $|k| \geq 1$
$$
\mmod{I(q,k)} \leq \frac{C |S(q,k)|^2}{|k|} \ .
$$
In particular $I(q,\cdot) \in L^1_{2N+1}([1,\infty), \R)$. By the real analyticity of the map $q \mapsto S(q,\cdot)$, it follows that $\class^{N,M} \to L^1_{2N+1}([1,\infty),\R) $, $q \mapsto \left.I(q,\cdot)\right|_{[1,\infty)}$ is real analytic. 

Now let us analyze $I(q,k)$ for $0 \leq k \leq 1$. By the definition \eqref{action_angle} of $I(q,k)$ one has 
$$
I(q,k) +  \frac{k}{\pi} \log \left( \frac{4k^2}{4(k^2+1)}\right) = - \frac{k}{\pi} \log \left( \frac{4(k^2 + 1)}{4k^2 + S(q,k) S(q,-k)} \right)  \ .
$$
By  Proposition \ref{prop:inv.l}, the map
$\S^{M,N} \to \Hz ([0,1],\R)$, $\sigma \to l(\sigma,k):= \log \left( \frac{4(k^2 + 1)}{4k^2 + \sigma(k) \sigma(-k)} \right)$ is real analytic.\\
Thus also the map $\class^{N,M} \to \Hz([0,1],\R)$, $q \to l(S(q),\cdot)$ is real analytic, being composition of real analytic maps. Since the interval $[0,1]$ is bounded, the   map $f \mapsto kf $, which multiplies a function by $k$, is analytic as  a map $\Hz([0,1],\R) \to \Hz([0,1],\R)$.
It follows that the map $q \mapsto - \frac{k}{\pi} l(S(q), k)$ is real analytic as a map from $\class^{N,M}$ to $H^M([0,1],\R)$.
\qed
\vspace{1em}\\
For $t \in \R$ and $\sigma \in H^1_\C$, let us denote by 
\begin{equation}
\label{def.rot} 
\rot^t(\sigma)(k):= e^{- i8 k^3 t} \sigma(k) \ . 
\end{equation}
 We prove the following lemma.
\begin{lemma}
\label{lem:airy.flow.invariant}
Let $N, M$ be integers with $N \geq 2M \geq 2$. Let $\sigma \in  \S^{M,N}$. Then $\rot^t(\sigma) \in  \S^{M,N},$ $\, \forall t \geq 0$.
\end{lemma}
\begin{proof}
As a first step we show that $\rot^t(\sigma) \in \S$ for every $t \geq 0$. Since  $\rot^t(\sigma)(0) = \sigma(0) >0$ and $\overline{\rot^t(\sigma)(k)} = \rot^t(\sigma)(-k)$,  $\rot^t(\sigma) $ satisfies (S1) and (S2) for every $t \geq 0$. Thus $\rot^t(\sigma) \in \S$, $\forall t \geq 0$. Next we show that $\rot^t(\sigma) \in \Hzc \cap L^2_N$.  Clearly $|\rot^t(\sigma)(k)| \leq |\sigma(k)|$, thus   $\rot^t (\sigma) \in   L^2_N,$ $\, \forall t \geq 0$. Now we show  that $\rot^t (\sigma) \in  \Hzc, $ $\, \forall t \geq 0$.
In particular we prove  that $\zeta \derk^M \rot^t (\sigma) \in L^2$, the other cases being analogous. Using the expression \eqref{def.rot} one gets that $\zeta(k) \derk^M \rot^t (\sigma)(k)$ equals 
$$
e^{- i8 k^3 t}\left( \zeta (k) \derk^{M}\sigma(k) + \sum_{j=1}^{M-1} \binom{M}{j} \left(-i 24 t k^2  \right)^j \zeta (k) \derk^{M-j}\sigma(k) + \left(-i 24 t k^2  \right)^M \zeta (k) \sigma(k) \right) \ .
$$
As $\sigma \in \S^{M,N}$,   the first and  last term  above are in $L^2$. Now we show that for $1 \leq j \leq M-1$, $|k|^{2j} \zeta  \derk^{M-j}\sigma  \in L^2$. We will  use the following   interpolating estimate, proved in  \cite[Lemma 4]{nahas_ponce}. Assume that $J^a f := (1-\derk^2)^{a/2} f \in L^2$ and $\langle k \rangle^b f := (1+|k|^2)^{b/2} f \in L^2$. Then for any $\theta \in (0,1)$
\begin{equation}
\label{inter.est}
\norm{\langle k \rangle^{\theta b} J^{(1-\theta)a} f}_{L^2} \leq c \norm{f}_{L^2_b}^{\theta} \norm{f}_{H^a_\C}^{1-\theta} \ .
\end{equation}
Note that $\zeta \sigma \in H^M_\C \cap L^2_N$, thus we can apply estimate \eqref{inter.est} with $f=\zeta \sigma$, $b=N$, $a=M$,  $\theta = \frac{j}{M}$, to obtain that $\langle k \rangle^{\frac{Nj}{M}} \derk^{M-j} (\zeta\sigma) \in L^2 $. Since $N \geq 2M$, we have  $\langle k \rangle^{2j} \derk^{M-j} (\zeta \sigma) \in L^2 $. 
By integration by parts 
$$
\langle k \rangle^{2j} \zeta(k) \derk^{M-j}  \sigma(k) = \langle k \rangle^{2j} \derk^{M-j} (\zeta \sigma) - \sum_{l=1}^{M-j}\binom{M-j}{l} \langle k \rangle^{2j} \derk^l \zeta(k) \, \derk^{M-j-l}\sigma(k) \ .
$$
Since for any $l \geq 1$ the function $\derk^l \zeta$ has compact support, it follows that the r.h.s. above is in $L^2$. 
Thus for every $1 \leq j \leq M-1$ we have $\langle k \rangle^{2j} \zeta(k) \derk^{M-j}  \sigma \in L^2$ and  it follows that $\zeta \derk^M \rot^t (\sigma) \in L^2$ for every $t \geq 0$.
\end{proof}
\begin{remark}
\label{rem.airy.flow}
One can adapt the proof above, putting $\zeta(k)\equiv 1$, to  shows that  the spaces $H^N \cap L^2_M$, with integers $N \geq 2M \geq 2$, are invariant by the Airy flow. Indeed the Fourier transform $\F_-$ conjugates the Airy flow with the linear flow $\rot^t$, i.e., $\airy^t =\mathcal{F}^{-1}_{-}\circ \rot^t \circ \mathcal{F}_-$.
\end{remark}

\vspace{1em}
\noindent {\em Proof of Theorem \ref{firstapprox}}.
Recall that by \cite{gardner6} the scattering map $S $ conjugate the KdV 
flow with the linear flow $\rot^t(\sigma)(k):= e^{- i8 k^3 t} \sigma(k)$, i.e., 
\begin{equation}
\kdv^t =S^{-1}\circ \rot^t \circ S \ ,
\label{kdv_scat}
\end{equation}
whereas $\airy^t =\mathcal{F}^{-1}_{-}\circ \rot^t \circ \mathcal{F}_-$.
Take now $q \in \class^{N, M}$, where $N, M$ are integers  with $N \geq 2M \geq 8$. By Theorem \ref{reflthm}, $S(q)\equiv S(q,\cdot) \in \S^{M,N}$. 
By Lemma \ref{lem:airy.flow.invariant}  the flow $\rot^t$ preserves the space $\S^{M,N}$ for every $t \geq 0$. Thus $\rot^t \circ S(q) \in \S^{M,N}$, $\ \forall t \geq 0$. 
By the bijectivity of $S$ it follows that $S^{-1}\circ \rot^t \circ S(q) \in \class^{N,M}$ $\, \forall t \geq 0$. Thus item $(i)$ is proved.

 We prove now item $(ii)$. Remark that by item $(i)$, $\kdv^t(q) \in L^2_M$ for any $t \geq 0$. 
Since  $\airy^t$ preserves the space $H^{N} \cap L^2_M$ ($N \geq 2M \geq 8$), it follows that for $q \in \class^{N, M}$  the difference $\kdv^t(q) - \airy^t(q) \in H^{N} \cap L^2_M$, $\forall t \geq 0$. We prove now  the smoothing property of the difference $\kdv^t(q) - \airy^t(q)$. 
Since $S^{-1} = \F_-^{-1} + B$,  
\begin{align}
\kdv^t(q)= &\mathcal{F}_{-}^{-1}\circ \rot^t \circ S (q)+ B \circ \rot^t \circ S (q)\label{secondairy}
\end{align}
and since  $S = \F_- + A$, 
$$\mathcal{F}_{-}^{-1}\circ \rot^t \circ S (q) =  \mathcal{F}_{-}^{-1}\circ \rot^t \circ \mathcal{F}_- (q) + \mathcal{F}_{-}^{-1}\circ \rot^t \circ A (q) \ . $$
Hence
\begin{align}
\kdv^t(q) = & \airy^t(q) + \mathcal{F}_{-}^{-1}\circ \rot^t \circ A (q) + B \circ \rot^t \circ S (q) \label{firstairy}.
\end{align}
The $1$-smoothing property of the difference
$\kdv^t(q) - \airy^t(q)$ follows now from the smoothing properties of $A$ and $B$ described in item (ii) 
of Theorem \ref{reflthm}. The real analyticity of the map $q \mapsto \kdv^t(q) - \airy^t(q)$ follows from formula \eqref{firstairy} and the real analyticity of the maps $A$, $B$ and $S$.
\qed
\appendix
\section{Auxiliary results. }
\label{techLemma}
For the convenience of the reader in this appendix we collect various known estimates used throughout the paper.
\begin{lemma} 
\label{lem:techLemma}
Fix an arbitrary real number $c$.
Then the following holds:
\begin{enumerate}
\item[(A0)]  The linear map $T_0: L^2_{1/2, x \geq c}\to \Lxyc$ defined by
\begin{equation}
\label{T4.def}
g \mapsto T_0(g)(x,y) := g(x+y)
\end{equation}
is continuous, and there exists a constant $K_c>0$, depending on $c$, such that
\begin{equation}
\label{T4.norm}
\norm{T_0(g)}_{\Lxyc} \leq K_c \norm{g}_{L^2_{1/2, x \geq c}}.
\end{equation} 
\item[(A1)] The bilinear map $T_1:\Lc \times \Lc \to \Lxyc$ defined by
\begin{equation}
\label{T2.def}
 (g,h) \mapsto 
T_1(g,h)(x,y):= g(x+y) h(x)
\end{equation}
is continuous,  and 
 \begin{equation}
\label{T2.norm}
\norm{T_1(g,h)}_{\Lxyc} \leq \norm{g}_{\Lc} \norm{h}_{\Lc}.
\end{equation}
\item[(A2)]The bilinear map $T_2: \Lc \times \Lxyc \to \Lc$ defined by
\begin{equation}
\label{T3.def}
(g,h) \mapsto 
T_2(g,h)(x):= \intzi g(x+z) h(x, z) \, dz
\end{equation}
is continuous, and  there exists a constant $K_c>0$, depending on $c$, such that
\begin{equation}
\label{T3.norm}
\norm{T_2(g,h)}_{\Lc} \leq K_c \norm{g}_{\Lc} \norm{h}_{\Lxyc}.
\end{equation}
\item[(A3)](Hardy inequality) The linear map $T_3: L^2_{m+1, x \geq c} \to L^2_{m, x \geq c}$ defined by 
$$
g \mapsto T_3(g) (x):= \intx{x} g(z) dz
$$
is continuous, and there exists a constant $K_c >0$, depending on $c$, such that 
$$\norm{T_3(g)}_{L^2_{m, x \geq c}}\leq K_c \norm{ g}_{L^2_{m+1, x \geq c}} \ . $$
\item[(A4)] The bilinear map $T_4: \Loc \times \Lxyc \to \Lxyc$ defined by
\begin{equation}
 (g, h) \mapsto T_4(g,h)(x,y) := \intzi g(x+y+z) h(x,z) dz
\end{equation}
 is  continuous,  and  there exists a constant $K_c>0$, depending on $c$, such that
 \begin{equation}
 \label{T1.norm}
 \norm{T_4(g,h)}_{\Lxyc} \leq K_c \norm{g}_{\Loc} \norm{h}_{\Lxyc}.
 \end{equation}
\item[(A5)]The bilinear map $T_5: \Lc \times L^2_{1, x \geq c} \to \Lxyc$
defined by
\begin{equation}
(g, h) \mapsto T_5(g, h)(x,y):= \intzi  g(x+y+z)  h(x+z) dz
\end{equation}
is bounded and  satisfies
\begin{equation}
\label{T1.norm.2}
\norm{T_5(g,h)}_{\Lxyc} \leq K_c \norm{g}_{\Lc} \norm{h}_{L^2_{1, x \geq c}}.
\end{equation}
\end{enumerate}
\end{lemma}
\begin{proof}
Inequality $(A1),  (A4)$ can be verified in a straightforward way.  
To prove  $(A0)$ make  the change of variable $\xi=x+y$ and remark that
$$
\intx{c} \intzi |g(x+y)|^2 \,dx \,dy \leq K_c \intzi |\xi-c| \, |g(\xi)|^2 d\xi \ .
$$
 We prove now $(A2)$: using Cauchy-Schwartz, one gets 
\begin{align*}
\norm{\intx{0} g(x+z)h(x,z) \, dz}_{\Lc}^2 & \leq \int\limits_c^{+\infty} \left( \intx{x}|g(z)|^2 \, dz \right)
\, \left( \intzi |h(x,z)|^2 \, dz \right) \, dx  \leq \norm{g}_{\Lc}^2 \norm{h}_{\Lxyc}^2 \ .
\end{align*}
In order to prove $(A3)$ take a function $h \in \Lc$ and remark that
\begin{align*}
\mmod{\intx{c} dx\, h(x)\, \x^m \intx{x} g(z)\, dz} & = \mmod{\intx{c} dz\, g(z) \int_{c}^z \x^m h(x) dx} \leq \tilde K_c  \intx{c} dz\, \langle z \rangle^m |g(z)| \int_{c}^z  |h(x)| \,dx \\
& \leq K_c \intx{c} dz\,\langle z \rangle^{m+1} |g(z)| \, \frac{\int_{c}^z  |h(x)| dx}{|z-c|} \leq K_c
\norm{\langle z \rangle^{m+1} g}_{\Lc} \norm{h}_{\Lc}
\end{align*}
where for the last inequality we used the Hardy-Littlewood inequality. \\
To prove $(A4)$ take a function $f \in \Lxyc$, define $\Omega_c= [c, \infty)\times \R^+ \times \R^+$ and remark that
\begin{align*}
&\int\limits_{\Omega_c } |g(x+y+z)| \, |h(x,z)|  \, |f(x,y)| \, dx \, dy \, dz \leq \\
& \qquad \qquad \leq \Big(\int\limits_{\Omega_c} |g(x+y+z)| \, |h(x,z)|^2 \, dx \, dy \, dz  \Big)^{1/2} \Big(\int\limits_{\Omega_c} |g(x+y+z)| \, |f(x,y)|^2 \, dx \, dy \, dz  \Big)^{1/2}  \\
& \qquad \qquad \leq
\norm{g}_{\Loc} \norm{h}_{L^2_{x\geq c, z \geq 0}} \norm{f}_{\Lxyc},
 \end{align*}
 where the first inequality follows by writing $ |g|= |g|^{1/2}\cdot |g|^{1/2}$ and applying Cauchy-Schwartz.\\
To prove $(A5)$ note that 
$$\norm{\intx{0} g(x+y+z) h(x+z) \, dz}_{L^2_{y \geq 0}} \leq \norm{g}_{\Lc} \intx{x}\mmod{ h(z)} \, dz \ .$$
By $(A3)$ one has that $\norm{\intx{x}\mmod{ h(z)} \, dz}_{\Lc} \leq K_c \norm{\x h }_{\Lc}$, then  $(A5)$ follows.
\end{proof}

\section{Analytic maps in complex Banach spaces}
\label{analytic_map}

In this appendix we recall the definition of an analytic map from \cite{mujica}. 

Let $E$ and $F$ be complex Banach spaces. 
A map $\tilde P^k:E^k\to F$ is said to be $k$-multilinear if $\tilde P^k(u^1,\ldots,u^k)$ is linear in
each variable $u^j$; a multilinear map is said to be bounded if there
exist a constant $C$ such that
$$\norm{\tilde P^k(u^1,\cdots,u^k)}\leq C \norm{u^1}\cdots \norm{u^k} \quad
\forall u^1,\ldots, u^k \in E.
$$ 
Its norm is defined by
$$\norm{\tilde P^k}:=\sup_{u^j \in E,\;\norm{u^j}\leq 1}{\norm{\tilde P^k(u^1,\ldots,u^k)}}.$$
A map $P^k: E \to F $ is said to be a polynomial of order $k$ if there exists
a $k$-multilinear map $\tilde{P}^k: E \to F$ such that
$$P^k(u)=\tilde{P}^k(u,\ldots,u)\quad \forall u\in E.$$
The polynomial is bounded if it has  finite norm
$$\norm{P^k}:=\sup_{\norm{u}\leq 1}\norm{P^k(u)}.$$
We  denote with $\mathcal{P}^k(E, F)$ the vector space of all bounded polynomials of order $k$ from $E$ into $F$.
\begin{definition}
\label{def.analytic_map}
Let $E$ and $F$ be complex Banach spaces. Let $U$ be a open subset of $E$. A mapping $f: U \to F$ is said to be \textit{analytic} if for each $a \in U$ there exists a ball $B_r(a)\subset U$ with center $a$ and radius $r>0$ and a sequence of polynomials $P^k \in \mathcal{P}^k(E, F)$, $k \geq 0$,  such that 
$$ f(u) = \sum_{k=0}^{\infty} P^k(u-a)$$
is convergent uniformly for $u \in B_r(a)$; i.e., for any $\epsilon >0$ there exists $K >0$ so that $$\norm{f(u) - \sum_{k=0}^K P^k(u-a)} \leq \epsilon$$ for any $u \in B_r(a)$.
\end{definition}

Finally let us recall the notion of real analytic map. 
\begin{definition}Let  $E,\; F$ be real Banach spaces and denote by $ E_{\C}$ and $F_{\C}$ their complexifications. Let $U \subset E$ be open.
A map $f: U \to F$ is called \textit{real analytic} on $U$ if for each point $u \in U$ there exists a neighborhood $V$ of $u$ in $E_{\C}$ and an analytic map $g: V \to F_{\C}$ such that $f = g$ on $U \cap V$.
\end{definition}
\begin{remark}
The notion of an analytic map in Definition \ref{def.analytic_map} is equivalent to the notion of a $\C$-differentiable map. Recall that a map $f: U \to F$, where $U$, $E$ and $F$ are given  as in Definition \ref{def.analytic_map}, is said to be \textit{$\C$-differentiable} if for each point $a \in U$ there exists a linear, bounded operator $A: E \to F$  such that
$$
\lim_{u \to a}\frac{\norm{f(u) - f(a) - A(u-a)}_F}{\norm{u-a}_E} = 0.
$$
Therefore analytic maps inherit the properties of $\C$-differentiable maps; in particular the composition of analytic maps is analytic.
For a proof of the equivalence of the two notions see \cite{mujica}, Theorem 14.7.
\end{remark}
\begin{remark}
Any $P^k \in \mathcal{P}^k(E, F)$ is an analytic map.
Let $f(u)=\sum_{m=0}^{\infty} P^m(u)$ be a power series from $E$ into $F$ with infinite radius of convergence with $P^m \in \mathcal{P}^m(E,F)$.  Then $f$ is analytic  (\cite{mujica}, example 5.3, 5.4).
\label{entire.func}
\end{remark}


\section{Properties of the solutions of  integral equation \eqref{opma1}}
\label{app:ab.eq}
 In this section we discuss some properties of the solution of equation  \eqref{opma1} which we rewrite as
\begin{equation}
\label{int.eq}
g(x,y) + \intx{0} F_\sigma(x+y+z) \, g(x,z) \, dz = h_\sigma(x,y) \ .
\end{equation}
Here  $\sigma \in \S^{4,N}$, $N \geq 0$, $h_\sigma$ is a function $h_\sigma: [c, +\infty) \times [0, +\infty) \to \R$, with $c$ arbitrary,  which satisfies $(P)$.
We denote by
\begin{equation}
\label{norm_0}
\norm{h}_0 := \norm{h}_{\Lxyc} + \norm{h(\cdot, 0)}_{\Lc} \ .
\end{equation}

Furthermore  $F_\sigma:\R \to \R$ is a function that satisfies
\begin{enumerate}
\item[$(H)$] The map $\sigma \mapsto F_\sigma $ $\,[ \sigma \mapsto F_\sigma']$ is real analytic as a map from $\S^{4,N}$ to $H^1 \cap L^2_3$ $\,[L^2_4]$. Moreover the operators $Id \pm \K_{x,\sigma}: L^2_{y \geq 0} \to L^2_{y \geq 0}$ with $\K_{x,\sigma}$ defined as
\begin{equation}
\label{oper.K.app}
\K_{x,\sigma}[f](y) := \intx{0} F_\sigma(x+y+z) \, f(z) \, dz \ 
\end{equation}
are invertible for any $x \geq c$, and there exists a constant $C_\sigma>0$, depending  locally uniformly on $\sigma \in H^4_{\zeta,\C}\cap L^2_N$, such that
\begin{equation}
\label{est.inv}
\sup_{x \geq c} \norm{(Id \pm \K_{x,\sigma})^{-1}}_{\L(L^2_{y \geq 0})} \leq C_\sigma \ .
\end{equation}
Finally  $\sigma \mapsto (Id \pm \K_{x,\sigma})^{-1}$ are real analytic as maps from $\S^{4,N}$ to $\L(\Lxyc)$.
\end{enumerate}
\begin{remark}
 \label{pairing}
The pairing
  \begin{align*} 
  \L(\Lxyc) \times  \Lxyc  \to  \Lxyc, \qquad   (H, f) \mapsto  H[f]
 \end{align*} 
 is a bounded bilinear map and hence analytic. Let now $\sigma \mapsto h_\sigma$ be a real analytic  map from $\S^{4,0}$ to $\Lxyc$ and let $\K_\sigma$ as in $(H)$. Then by Lemma \ref{lem:Fx} $(iii)$ it follows that $\sigma \mapsto \left( Id + \K_\sigma\right)^{-1}[ h_\sigma]$ is real analytic as a map from $\S^{4,0}$ to $\Lxyc$ as well.
 \end{remark}
\begin{remark}
By the Sobolev embedding theorem,    assumption $(H)$ implies that  $F_\sigma \in C^{0, \gamma}(\R, \C)$, $\gamma < \frac{1}{2}$. 
\end{remark}

By assumption $(H)$ the map  
$(c, \infty) \to \L(L^2_{y \geq 0}), $  $\, x \mapsto \K_{x,\sigma}$ is differentiable and its derivative  is the operator 
\begin{equation}
\label{K'.def} 
 \K_{x,\sigma}'[f](y) = \intx{0} F_\sigma'(x+y+z)\, f(z) \, dz \ , 
 \end{equation}
 as one  verifies using that for $x > c$ and  $\epsilon \neq 0$ sufficiently small
\begin{equation}
\label{op.K.diff} 
\begin{aligned}
\norm{\frac{\K_{x+\epsilon,\sigma} - \K_{x,\sigma}}{\epsilon}  - \K_{x,\sigma}' }_{\L(L^2_{y \geq 0})}  \leq 
\intx{x} \mmod{\frac{F_\sigma(z+\epsilon) - F_\sigma(z)}{\epsilon} - F_\sigma'(z)} dz \\
\leq \frac{1}{|\epsilon|} \mmod{ \int_0^{\epsilon} \intx{x}\mmod{F_\sigma'(z+s) - F_\sigma'(z)} dz\, ds } \leq \sup_{   |s| \leq |\epsilon | } \intx{x}\mmod{F_\sigma'(z+s) - F_\sigma'(z)} dz 
\end{aligned}
\end{equation}
and the fact that the  translations are continuous in $L^1$.
Therefore  the following lemma holds
\begin{lemma}
\label{lem:der.K}
$\K_{x,\sigma}$ and thus $(Id + \K_{x,\sigma})^{-1}$ is a family of operators from $\Lyz$ to $\Lyz$ which depends continuously on the parameter $x$. Moreover the map $(c, \infty) \to \L(\Lyz)$,  $x \mapsto \K_{x,\sigma}$ is differentiable  and its derivative is the operator $\K'_{x,\sigma}$ defined in \eqref{K'.def}.
\end{lemma}

\begin{lemma}
\label{lem:g.p}
Let $F_\sigma$ satisfy assumption $(H)$, and $g_\sigma \in \CxLy \cap \Lxyc$ be such that  $\norm{g_\sigma}_{\Lxyc} \leq K_c \norm{\sigma}_{H^4_{\zeta,\C} \cap L^2_N}^2$ and $\S^{4,N} \to \Lxyc,$ $\, \sigma \mapsto g_\sigma$ be real analytic. Then
$$
\bF (x,y) := \intzi F_\sigma(x+y+z) \, g_\sigma(x,z) \, dz
$$
satisfies $(P)$.
\end{lemma}
\begin{proof} $(P1)$ For $\epsilon \neq 0$ sufficiently small
\begin{align*}
\norm{\bF(x+\epsilon, \cdot) - \bF(x,\cdot)}_{\Lyz} \leq & \norm{F_\sigma(x+\epsilon, \cdot) - F_\sigma(x,\cdot)}_{L^1} \norm{g_\sigma(x+\epsilon,\cdot)}_{\Lyz} \\
& \quad  + \norm{F_\sigma}_{L^1}\norm{g_\sigma(x+\epsilon,\cdot) - g_\sigma(x,\cdot)}_{\Lyz}
\end{align*}
which goes to $0$ as $\epsilon \to 0$, proving that $\bF \in \CxLy$.
Furthermore, by Lemma \ref{lem:techLemma} $(A4)$, $\bF \in \Lxyc$ and fulfills
\begin{equation}
\label{lem:g.p1}
\norm{\bF}_{\Lxyc} \leq \norm{F_\sigma}_{L^1} \norm{g_\sigma}_{\Lxyc} \leq K_c \norm{\sigma}_{H^4_{\zeta,\C} \cap L^2_N}^2 \ .
\end{equation}
Now we show that $\bF \in \Cxy$. Let $(x_n)_{n \geq 1} \subseteq [c,\infty)$ and $(y_n)_{n \geq 1}\subseteq [0,\infty)$ be two sequences such that 
$x_n \to x_0$, $y_n \to y_0$. Then $F_\sigma(x_n + y_n +\cdot) g_\sigma(x_n, \cdot) \to F_\sigma(x_0 + y_0 +\cdot) g_\sigma(x_0, \cdot) $ in $L^1_{z \geq 0}$ as $n \to \infty$. Indeed
\begin{align*}
& \norm{F_\sigma(x_n + y_n +\cdot) g_\sigma(x_n, \cdot) - F_\sigma(x_0 + y_0 +\cdot) g_\sigma(x_0, \cdot)}_{L^1_{z \geq 0}}  \leq \\
& \qquad \qquad \qquad \leq  \norm{F_\sigma(x_n+y_n+ \cdot) - F_\sigma(x_0 + y_0 +\cdot)}_{L^2_{z\geq 0}} \norm{g_\sigma(x_n,\cdot)}_{\Lyz} \\
&\qquad \qquad \qquad \quad + \norm{F_\sigma(x_0 + y_0 +\cdot)}_{L^2_{z\geq 0}}\norm{g_\sigma(x_n,\cdot) - g_\sigma(x_0,\cdot)}_{\Lyz} , 
\end{align*}
and the r.h.s. of the inequality above goes to $0$ as $(x_n, y_n) \to (x_0,y_0)$, by the continuity of the translations in $L^2$ and the fact that $g_\sigma \in \CxLy$. Thus it follows that $\bF(x_n, y_n) \to \bF(x_0, y_0)$ as $n \to \infty$, i.e., $\bF \in \Cxy$.\\
We evaluate $\bF$ at $y=0$, getting
$$
\bF(x,0) = \intzi F_\sigma(x+z) g_\sigma(x,z) \, dz \ .
$$
By Lemma \ref{lem:techLemma} $(A2)$, $\bF(\cdot,0) \in \Lc$ and fulfills
\begin{equation}
\label{lem:g.p2} 
\norm{\bF (\cdot,0)}_{\Lc} \leq \norm{F_\sigma}_{L^2} \norm{g_\sigma}_{\Lxyc} \leq K_c \norm{\sigma}_{H^4_{\zeta,\C} \cap L^2_N}^2 \ .
\end{equation}

$(P2)$ It follows from \eqref{lem:g.p1} and \eqref{lem:g.p2}.

$(P3)$ It follows by Lemma \ref{lem:techLemma} $(A2)$ and the fact that  $\bF$ and $\bF(\cdot,0)$ are composition of real analytic maps.
\end{proof}

We study now the solution of equation \eqref{int.eq}.

\begin{lemma}
\label{lem:ab.eq}
Assume that $h_\sigma$ satisfies $(P)$ and $F_\sigma$ satisfies $(H)$. Then equation \eqref{int.eq} has a unique  solution $g_\sigma$ in $\CxLy\cap \Lxyc$ which satisfies $(P)$.
\end{lemma}

\noindent {\sl Proof.} We start to show that $g_\sigma$ exists and satisfies $(P1)$. Since $h_\sigma$ satisfies $(P)$ and $F_\sigma$ satisfies $(H)$,  it follows that for any $x \geq c$, $g_\sigma(x, \cdot):= (Id + \K_{x,\sigma})^{-1}[h_\sigma(x, \cdot)]$  
is the unique solution in   $\Lyz $ of the integral equation \eqref{int.eq}. Furthermore, by \eqref{est.inv}, $\norm{g_\sigma(x, \cdot)}_{\Lyz} \leq C_\sigma \norm{h_\sigma(x, \cdot)}_{\Lyz}$, which implies
\begin{equation}
\label{g.eq.10}
\norm{g_\sigma}_{\Lxyc} \leq C_\sigma \norm{h_\sigma}_{\Lxyc} \ . 
\end{equation}
Since $h_\sigma \in \CxLy$,  Lemma \ref{lem:der.K} implies  that $g_\sigma \in \CxLy$ as well. Thus we have proved that $g_\sigma \in \CxLy \cap \Lxyc$.
Now write 
\begin{equation}
\label{g.eq.1}
g_\sigma(x, y) =  h_\sigma(x,  y) - \intx{0}  F_\sigma(x+y+z) g_\sigma(x,z) \, dz \ .
\end{equation}
By Lemma \ref{lem:g.p} and  the assumption that $h_\sigma$ satisfies $(P)$, it follows that  the r.h.s. of formula \eqref{g.eq.1}  satisfies $(P)$.
\qed
\vspace{1em}\\

The following lemma will be useful in the following:
\begin{lemma}
\label{lem:Phi}
 \begin{enumerate}
 \item[$(i)$]
 Let $F_\sigma$ satisfy $(H)$, and  $g_\sigma \in \CxLy \cap \Lxyc$ be such that  $\norm{g_\sigma}_{\Lxyc} \leq K_c \norm{\sigma}_{H^4_{\zeta,\C} \cap L^2_N}^2$ and $\S^{4,N} \to \Lxyc,$ $\, \sigma \mapsto g_\sigma$ be real analytic. Denote 
\begin{equation}
\label{Phi.def}
\Phi_\sigma(x,y) := \intzi F_\sigma'(x+y+z) g_\sigma(x,z) \, dz \ .
\end{equation}
Then $\Phi_\sigma \in \CxLy \cap \Lxyc$, the map $\S^{4,N} \to \Lxyc,$ $\, \sigma \mapsto \Phi_\sigma$ is real analytic and 
\begin{align}
\label{eq:Phi.est1}
\norm{\Phi_\sigma}_{\Lxyc}\leq K_c \norm{\sigma}_{H^4_{\zeta,\C} \cap L^2_N}^2 \ ,
\end{align}
where $K_c >0$ depends locally uniformly on $ \sigma \in H^4_{\zeta,\C}\cap L^2_N$.
\item[$(ii)$] Let $g_\sigma$ as in item $(i)$, and furthermore let $g_\sigma \in \Cxy$ and $g_\sigma(\cdot, 0) \in \Lc$. Assume furthermore that  $\dery g_\sigma$ satisfies the same assumptions as $g_\sigma$ in item $(i)$. Then  $\Phi_\sigma$, defined in \eqref{Phi.def}, satisfies $(P)$.
\item[$(iii)$] Assume that $F_\sigma$ satisfies $(H)$ and that the map $\S^{4,N} \to \Hoc$, $\, \sigma \mapsto b_\sigma$ is real analytic with $\norm{b_\sigma}_{\Hoc} \leq K_c \norm{\sigma}_{H^4_{\zeta,\C} \cap L^2_N}^2 $. Then the function
$$
\phi_\sigma(x,y) : = F_\sigma(x+y) b_\sigma(x)
$$
satisfies $(P)$.
\end{enumerate}
\end{lemma}
\begin{proof}
$(i)$ Clearly $\norm{\Phi_\sigma(x, \cdot)}_{\Lyz} \leq \norm{F_\sigma'}_{L^1}\norm{g_\sigma(x, \cdot)}_{\Lyz}$, and since $g_\sigma \in \Lxyc$ it follows that $\Phi_\sigma \in \Lxyc$ with $\norm{\Phi_\sigma}_{\Lxyc} \leq \norm{F_\sigma'}_{L^2_4} \norm{g_\sigma}_{\Lxyc}$, which implies \eqref{Phi.def}.
We show now that $\Phi_\sigma \in \CxLy$. For $\epsilon \neq 0$ one has
\begin{align*}
\norm{\Phi_\sigma(x+\epsilon, \cdot) - \Phi_\sigma(x, \cdot)}_{\Lyz} \leq \norm{ F_\sigma'(\cdot + \epsilon) - F_\sigma'}_{L^1} \norm{g_\sigma(x, \cdot)}_{\Lyz} + \norm{F_\sigma'}_{L^1}\norm{g_\sigma(x+\epsilon, \cdot) - g_\sigma(x, \cdot)}_{\Lyz}  \ .
\end{align*}
The continuity of the translation in $L^1$ and  the assumption  $g_\sigma \in \CxLy$ imply that $\norm{\Phi_\sigma(x+\epsilon, \cdot) - \Phi_\sigma(x, \cdot)}_{\Lyz} \to 0$ as $\epsilon \to 0$, thus $\Phi_\sigma \in \CxLy$. The real analyticity of  $\sigma \mapsto \Phi_\sigma$  follows from Lemma \ref{lem:techLemma} $(A4)$ and the fact that $\Phi_\sigma$ is composition of real analytic maps.

$(ii)$ Fix $x \geq c$ and use integration by parts to   write
\begin{equation}
\label{K'.parts} 
\Phi_\sigma(x,y)  = - F_\sigma(x+y)g_\sigma(x,0) - \intzi F_\sigma(x+y+z) \derz g_\sigma(x,z) \, dz \ ,
\end{equation}
where we used   that since $F_\sigma \in H^1$ $\, [g(x, \cdot) \in H^1_{y \geq 0}]$,  $\lim_{x \to \infty} F_\sigma(x) = 0$  $\, [\lim_{y \to \infty} g_\sigma(x,y) = 0]$. 
By the assumption and the proof of Lemma \ref{lem:g.p} $(P1)$, $\Phi_\sigma \in \Cxy$. We evaluate \eqref{K'.parts} at $y=0$ to get the formula
$$
\Phi_\sigma(x,0) = - F_\sigma(x)g_\sigma(x,0) - \intzi F_\sigma(x+z) \derz g_\sigma(x,z) \, dz \ .
$$
Together with Lemma \ref{lem:techLemma} $(A2)$ we have the estimate
\begin{equation}
\label{eq:Phi.est2}
\norm{\Phi_\sigma(\cdot, 0)}_{\Lc} \leq \norm{F_\sigma}_{H^1} \left( \norm{g_\sigma(\cdot, 0)}_{\Lc} + \norm{\dery g_\sigma}_{\Lxyc} \right)\leq K_c \norm{\sigma}_{H^4_{\zeta,\C} \cap L^2_N}^2 \ .
\end{equation}
Estimate  \eqref{eq:Phi.est2} together with estimate \eqref{eq:Phi.est1} imply that $\Phi_\sigma$ satisfies $(P2)$. Finally $\sigma \mapsto \Phi_\sigma(\cdot,0)$ is real analytic, being a composition of real analytic maps.

$(iii)$ We skip an easy proof.
\end{proof}

If the function  $h_\sigma$ is  more regular one  deduces better regularity properties of the corresponding solution of \eqref{int.eq}. 

\begin{lemma}
\label{lem:ab.eq2} Consider the integral equation \eqref{int.eq}  and assume  that $F_\sigma$ satisfies $ (H)$.
Assume that $h_\sigma$, $\derx h_\sigma$, $\dery h_\sigma$ satisfy $(P)$. Then  $g_\sigma$ solution of \eqref{int.eq} satisfies $(P)$. Its derivatives $\derx g_\sigma$ and $\dery g_\sigma$  satisfy $(P)$ and solve the equations 
\begin{align}
\label{x.eq}
&\left( Id + \K_{x,\sigma} \right) [\derx g_\sigma] = \derx h_\sigma - \K_{x,\sigma}'[g_\sigma] \ ,\\
\label{y.eq}
& \dery g_\sigma = \dery h_\sigma - \K_{x,\sigma}'[g_\sigma] \ .
\end{align}
\end{lemma}

\noindent {\em Proof.} 
By Lemma \ref{lem:ab.eq}, $g_\sigma$ satisfies $(P)$.\\

{\em $\dery g_\sigma$ satisfies $(P)$.}
 For $\epsilon \neq 0$ sufficiently small, we have in $\Lyz$  
 $$
 \frac{g_\sigma(x, y+\epsilon) - g_\sigma(x, y)}{\epsilon} = \Psi^\epsilon_\sigma(x,y)
 $$
 where 
 \begin{align}
\label{tras.y}
& \Psi^\epsilon_\sigma(x,y) :=  \frac{h_\sigma(x, y+\epsilon) - h_\sigma(x,y) }{\epsilon}-  \intzi \frac{F_\sigma(x+y+\epsilon+z) - F_\sigma(x+y+z)}{\epsilon} g_\sigma(x,z) \, dz \ . 
\end{align}
Define
 $$
 \Psi^0_\sigma(x,y):= \dery h_\sigma(x,y) - \intzi F_\sigma'(x+y+z) g_\sigma(x,z) \, dz \ .
 $$
Since  $\dery h_\sigma$ and $g_\sigma$ satisfy  $(P)$,  by Lemma \ref{lem:Phi} $(i)$ it follows that   $\Psi^0_\sigma \in \CxLy \cap \Lxyc$, the map $\S^{4,N} \to \Lxyc,$ $\, \sigma \mapsto \Psi^0_\sigma$ is real analytic and 
$\norm{\Psi^0_\sigma}_{\Lxyc}\leq K_c \norm{\sigma}_{H^4_{\zeta,\C} \cap L^2_N}^2$. Furthermore one verifies   that 
 $$
\dery g_\sigma(x,\cdot)= \lim_{\epsilon \to 0} \frac{g_\sigma(x, \cdot+\epsilon) - g_\sigma(x, \cdot)}{\epsilon} =  \lim_{\epsilon \to 0} \Psi^\epsilon_\sigma(x, \cdot) = \Psi^0_\sigma(x, \cdot) \qquad \mbox{in } \Lyz \ .
 $$
Thus $\dery g_\sigma$ fulfills
\begin{equation}
\label{dery.g.eq} 
\dery g_\sigma(x,y) = \dery h_\sigma(x,y) - \intzi F_\sigma'(x+y+z) g_\sigma(x,z) dz \ , 
\end{equation}
i.e.,  $\dery g_\sigma$ satisfies equation \eqref{y.eq}.
Since $\dery g_\sigma = \Psi^0_\sigma$, $g_\sigma$ satisfies the assumptions of Lemma \ref{lem:Phi} $(ii)$. Since $\dery h_\sigma$ satisfies $(P)$ as well, it follows that $\dery g_\sigma$ satisfies $(P)$.\\

{\em $\derx g_\sigma$ satisfies $(P)$}.  For $\epsilon \neq 0$ small enough we have in $\Lyz$
 $$
\left( Id + \K_{x+\epsilon,\sigma}\right) \left[\frac{g_\sigma(x+\epsilon, \cdot) - g_\sigma(x, \cdot)}{\epsilon}\right] = \Phi^\epsilon_\sigma(x, \cdot)
 $$
 where
 $$
 \Phi^\epsilon_\sigma(x,y):= \frac{h_\sigma(x+\epsilon, y) - h_\sigma(x,y)}{\epsilon} - \intzi \frac{F_\sigma(x+y+\epsilon+z) - F_\sigma(x+y+z)}{\epsilon} g_\sigma(x,z) \, dz \ .
 $$
 Define
 $$
 \Phi^0_\sigma(x,y):= \derx h_\sigma(x,y) - \intzi F_\sigma'(x+y+z) g_\sigma(x,z) \, dz \ .
 $$
Proceeding as above, one proves that  $\Phi^0_\sigma$ satisfies $(P)$, and
 $$
 \lim_{\epsilon \to 0} \Phi^\epsilon_\sigma(x, \cdot) = \Phi^0_\sigma(x, \cdot) \qquad \mbox{in } \Lyz \ .
 $$
 Together with Lemma \ref{lem:der.K} we get for $x >c$ in $\Lyz$
 \begin{equation}
 \label{derx.g.limit} 
 \derx g_\sigma(x, \cdot) = \lim_{\epsilon \to 0} \frac{g_\sigma(x+\epsilon, \cdot) - g_\sigma(x,\cdot)}{\epsilon} = \lim_{\epsilon \to 0} \left( Id + \K_{x+\epsilon,\sigma}\right)^{-1} \Phi^\epsilon_\sigma(x, \cdot) = \left( Id + \K_{x,\sigma}\right)^{-1} \Phi^0_\sigma(x, \cdot) \ .
 \end{equation}
In particular $(Id + \K_\sigma) (\derx g_\sigma(x,\cdot)) = \Phi^0_\sigma(x,\cdot)$. Since $\Phi^0_\sigma$ satisfies $(P)$, by Lemma \ref{lem:ab.eq},   $\derx g_\sigma$ satisfies $(P)$. 
  Formula \eqref{derx.g.limit} implies that 
  \begin{equation}
  \label{derx.g.eq} 
  \derx g_\sigma(x,y) + \intzi F_\sigma(x+y+z) \derx g_\sigma(x,z) \, dz = \derx h_\sigma(x,y) - \intzi F_\sigma'(x+y+z) g_\sigma(x,z) \, dz \ ,
  \end{equation}
  namely $\derx g_\sigma$ satisfies equation \eqref{x.eq}.\\

\qed


\section{Proof from Section \ref{sec:inv.scat}}
\label{proof.prop.B}

\subsection{Properties of $\K_{x,\sigma}^\pm$  and $f_{\pm, \sigma}$.}
We begin with proving  some properties of  $\K_{x,\sigma}^\pm$  and $f_{\pm, \sigma}$, defined in  \eqref{kxdef} and \eqref{f.def},  which will be needed later. 
\vspace{1em}\\
\vspace{1em}
{\em \underline{Properties of $Id + \K^{\pm}_{x,\sigma}$}}.  In order to solve the integral equations  \eqref{opma1} we need  the operator  $Id + \K^{+}_{x,\sigma}$ to be invertible on $L^2_{y \geq 0}$ (respectively   $Id + \K^{-}_{x,\sigma}$ to be invertible on $L^2_{y \leq 0}$). The following result is well known: 
\begin{lemma}[\cite{deift,KaCo}] Let $\sigma \in \S^{4, 0}$ and fix $c \in \R$. Then the following holds:
\label{lem:Fx}
\begin{enumerate}[(i)]
 \item For every $x\geq c$,   $\K_{x,\sigma}^{+} : L^2_{y_\geq 0} \rightarrow L^2_{y_\geq 0}$ is a bounded linear operator; moreover 
\begin{equation}
\sup_{x \geq c} \norm{\K_{x,\sigma}^{+}}_{\L(L^2_{y\geq 0})} < 1, \quad \mbox{ and } \quad 
\norm{\K_{x, \sigma}^+}_{\L(L^2_{y\geq 0})} \leq   \int\limits_x^{+ \infty} |F_{+, \sigma}(\xi)| \; d \xi \rightarrow 0 \quad \mbox{ if } \quad x \rightarrow + \infty.
\label{Kx_norm1}
\end{equation}
 \item The map $\K_\sigma^{+}: \Lxyc \to \Lxyc$,  $f \mapsto \K_\sigma^+[f]$, where $\K_\sigma^+[f](x,y) :=  K_{x, \sigma}^+[ f](y) $, is linear and bounded. Moreover the operators $Id\pm \K_\sigma^{+}$ are invertible on $\Lxyc$  and
 there exists a constant $K_c>0$, which depends locally uniformly on $\sigma \in \S^{4,0}$,  such that
\begin{equation} 
  \norm{\left( Id \pm \K_\sigma^{+}\right)^{-1}}_{\L(\Lxyc)} \leq K_c.
  \label{Kx_norm}
  \end{equation}
   \item $\sigma \mapsto \left( Id \pm \K_\sigma^{+}\right)^{-1}$ are real analytic as maps from $\S^{4, 0}$ to $\L(\Lxyc)$.
\end{enumerate}
Analogous results hold also for $\K_{x,\sigma}^-$ replacing $\Lxyc$ by $L^2_{x \leq c}L^2_{y \leq 0}$.
\end{lemma}

\vspace{1em}
{\em \underline{Properties of $f_{\pm, \sigma}$}}. First   note that $f_{\pm, \sigma}$, defined by \eqref{f.def}, are well defined. Indeed for any $\sigma \in \S^{4, 0}$,  Proposition \ref{rem:dec_rel} implies that $F_{\pm, \sigma} \in H^1 \cap L^2_3 \subset L^2$. Hence for any $x \geq c$, $\, y \geq 0$ the map given by $z \mapsto F_{+, \sigma}(x+y+z)F_{+, \sigma}(x+z) $ is in $L^1_{z \geq 0}$. Similarly, for any $x \geq c$, $\, y \geq 0$, the map given by $z \mapsto F_{-, \sigma}(x+y+z) F_{-, \sigma}(x+z) $ is in $L^1_{z \leq 0}$.

In the following we will use repeatedly  the Hardy inequality \cite{hardy}
\begin{equation}
\label{hardy.ineq}
\norm{\x^m \intx{x} g(z) dz}_{\Lc} \leq K_c \norm{\langle x \rangle^{m+1} g}_{\Lc}, \qquad \forall \, m \geq 0 \ .
\end{equation}
The inequality is well known, but for sake of completeness we give a proof of it in  Lemma \ref{lem:techLemma} $(A3)$. \\
We analyze now the maps $\sigma \mapsto f_{\pm, \sigma}$. Since the analysis of $f_{+, \sigma}$ and the one of $f_{-, \sigma}$ are similar, we will consider $f_{+, \sigma}$ only. To shorten the notation we will suppress the subscript '' + '' in what follows. 
\begin{lemma}
\label{f.prop}
Fix $N \in \Z_{\geq 0}$ and let $\sigma \in \S^{4,N}$. Let $f_{\sigma}\equiv f_{+, \sigma}$ be given as in \eqref{f.def}. Then for every $j_1, j_2\in \Z_{\geq 0}$ with $0 \leq j_1 + j_2 \leq N+1$, the function $\derx^{j_1}\dery^{j_2}f_\sigma$ satisfies $(P)$.
\end{lemma}
\begin{proof}
We prove at the same time $(P1), (P2)$ and $(P3)$ for any $j_1, j_2 \geq 0$ with $j_1 + j_2 = n$ for any $0 \leq n \leq N+1$.

\noindent{\em Case $n=0$.} Then $j_1 = j_2 = 0$. By Proposition \ref{rem:dec_rel}, for any $N \in \Z_{\geq 0}$ one has   $F_\sigma \equiv  F_{+, \sigma} \in H^1 \cap L^2_3$. 
\begin{enumerate}
\item[$(P1)$]  
We show  that $f_\sigma \in \CxLy$. For any $x \geq c$ fixed one has  $\norm{f_\sigma(x, \cdot)}_{\Lyz} \leq \norm{F_\sigma}_{L^1} \norm{F_\sigma(x+\cdot)}_{\Lyz}$, which shows that $f_\sigma(x, \cdot) \in \Lyz$. For $\epsilon \neq 0$ sufficiently small one has
\begin{align*}
\norm{f_\sigma(x+\epsilon, \cdot) - f_\sigma(x,\cdot)}_{\Lc} \leq & \norm{F_\sigma}_{L^1} \norm{F_\sigma(x+\epsilon+\cdot) -F_\sigma(x+\cdot) }_{\Lyz} \\
& +  \norm{F_\sigma(\epsilon + \cdot) - F_\sigma}_{L^1}\norm{F_\sigma(x+\cdot)}_{\Lyz}
\end{align*}
which goes to $0$ as $\epsilon \to 0$, due to the continuity of the translations in $L^p$-space, $1 \leq p < \infty$. Thus $f_\sigma \in \CxLy$.

We show now that $f_\sigma \in \Lxyc$.  Introduce $h_\sigma(x,y):= F_{\sigma}(x+y)$. Then $h_\sigma \in \Lxyc$, since for some $C, C' >0$
 \begin{equation}
 \label{F(x+y).est}
 \norm{h_\sigma}_{\Lxyc} \leq C \norm{ F_\sigma}_{L^2_{1/2, x \geq c}} \leq C' \norm{\sigma}_{H^4_{\zeta,\C}} 
 \end{equation}
where for the first [second] inequality we used Lemma \ref{lem:techLemma} $(A0)$ [Proposition \ref{rem:dec_rel} $(i)$].
By  Lemma \ref{lem:techLemma}$(A4)$ and using once more Proposition \ref{rem:dec_rel} $(i)$, one  gets
\begin{equation}
\label{f_sigma.norm}
\norm{f_{\sigma}}_{\Lxyc} \leq C'' \norm{F_{\sigma}}_{\Loc} \norm{h_\sigma}_{\Lxyc}  \leq C''' \norm{F_{\sigma}}_{L^2_1} \norm{h_\sigma}_{\Lxyc} \leq C'''' \norm{\sigma}_{H^4_{\zeta,\C}}^2 \ ,
\end{equation}
for some $ C'', C''', C'''' >0 $.  Thus $f_\sigma \in \Lxyc$.

To show that  $f_\sigma \in \Cxy$ proceed as in Lemma \ref{lem:g.p}. 

Finally we show that $f_\sigma(\cdot,0) \in \Lc$.  Evaluate \eqref{f.def} at $y=0$ to get $f_{\sigma}(x,0)=\intx{x} F_{\sigma}^2(z)\, dz$. Using the Hardy inequality \eqref{hardy.ineq},   $F_{\sigma}(x) = -\int_{x}^{+ \infty} F_{\sigma}'(s) \, ds$ and Proposition \ref{rem:dec_rel} one obtains 
\begin{align}
\nonumber
\norm{f_{\sigma}(\cdot, 0)}_{\Lc} &\leq \norm{\x F_{\sigma}^2}_{\Lc} \leq \norm{\x F_{\sigma}}_{\Lic} \norm{F_{\sigma}}_{\Lc} \leq K_c \norm{\x F_{\sigma}'}_{\Loc} \norm{F_{\sigma}}_{\Lc}\\ 
\label{f.0.T}
&\leq K_c' \norm{\la x \ra^2 F_{\sigma}'}_{\Lc} \norm{F_{\sigma}}_{\Lc} \leq K_c'' \norm{\sigma}_{H^4_{\zeta,\C} \cap L^2_{N}}^2 \ ,
\end{align}
for some constants $K_c, K_c', K_c'' >0$. Thus  $f_\sigma(\cdot,0) \in \Lc$.
 
\item[$(P2)$] It follows from \eqref{f_sigma.norm} and \eqref{f.0.T}.

 \item[$(P3)$] By Proposition \ref{rem:dec_rel} $(i)$, $\S^{4,0} \to H^1_\C \cap L^2_3, $ $\, \sigma \mapsto F_\sigma$ is real analytic and by Lemma \ref{lem:techLemma} $(A0)$ so is $\S^{4,0} \to \Lxyc, $ $\, \sigma \mapsto h_\sigma$. 
 By Lemma \ref{lem:techLemma} $(A4)$ it  follows that $\S^{4,0} \to \Lxyc, $ $\, \sigma \mapsto f_\sigma$ is real analytic.
 Since the map $\sigma \mapsto f_{\sigma}(\cdot,0)$ is a composition of real analytic maps, it is real analytic as a map from $\S^{4,N}$ to $\Lc$.
\end{enumerate}

\noindent{\em Case $n \geq 1$.}   By Proposition \ref{rem:dec_rel},   $F_{\sigma} \in H^{N+1}$ and $\norm{F_{\sigma}}_{H^{N+1}} \leq C' \norm{\sigma}_{H^4_{\zeta,\C} \cap L^2_{N}}$. By Sobolev embedding theorem, it follows that $F_\sigma \in C^{N, \gamma}(\R, \R)$, $\gamma < \tfrac{1}{2}$.   Moreover since $\lim_{x \to +\infty} F_{\sigma}(x) = 0$, one has
\begin{equation}
\label{derxfR}
 \derx f_{\sigma}(x,y) = \derx \intx{x} F_{\sigma}(y+z) F_{\sigma}(z) \, dz = - F_{\sigma}(x+y) F_{\sigma}(x) \ . \end{equation}
Consider first the case  $j_1 \geq 1$. Then $j_2 \leq N$. By \eqref{derxfR} it follows that 
\begin{align}
\label{derf(x,y).1}
\derx^{j_1} \dery^{j_2} f_{\sigma}(x,y) &= -\sum_{l=0}^{j_1-1}\binom{j_1-1}{l} F_{\sigma}^{(j_2 + l)}(x+y) F_{\sigma}^{(j_1-1-l)}(x)  \ ,
\end{align}
where  $F^{(l)}_{\sigma}\equiv \derx^l F_{\sigma}$. Thus $\derx^{j_1} \dery^{j_2} f_{\sigma}$ is a linear combination of terms of the form \eqref{bH.def0}, with $b_\sigma = F_\sigma^{(j_1-1-l)}$ satisfying the assumption of Lemma \ref{lem:P} $(i)$,  thus $\derx^{j_1} \dery^{j_2} f_{\sigma}$, with $j_1 \geq 1$, satisfies $(P)$. 
\vspace{1em}\\
Consider now the case $j_1 =0$. Then  $1 \leq j_2 \leq n \leq  N+1$. Since  $ \dery F_{\sigma}(x+y+z) = \derz F_{\sigma}(x+y+z) = F_{\sigma}'(x+y+z) $, by integration by parts one obtains
\begin{align}
\label{derf(x,y).2}
\dery^{j_2} f_{\sigma}(x,y)  &= - F_{\sigma}^{(j_2-1)}(x+y) F_{\sigma}(x) - \intzi  F_{\sigma}^{(j_2-1)}(x+y+z)  F_{\sigma}'(x+z) dz \ .
\end{align}
Then, by Lemma \ref{lem:P} $(i)$ and $(ii)$, $\dery^{j_2} f_{\sigma}$ is the sum of two terms which satisfy $(P)$, thus it satisfies $(P)$ as well.
\end{proof}

\begin{lemma}
\label{lem:P}
 Fix $ c \in \R$,  $N \in \Z_{\geq 0}$ and let  $\sigma \in \S^{4, N}$. Let $F_\sigma$ be given as in \eqref{F.four}. 
 Then the following holds true:
\begin{enumerate}[(i)]
\item Let $\sigma \mapsto b_\sigma$ be  real analytic as a map from $\S^{4,N}$ to $H^1_{x \geq c}$, satisfying $\norm{b_\sigma}_{H^1_{x \geq c}} \leq K_c \norm{\sigma}_{H^4_{\zeta,\C} \cap L^2_N}$, where $K_c >0$ depends locally uniformly on $\sigma \in H^4_{\zeta,\C}\cap L^2_N$. Then for every integer $k$  with  $0 \leq k \leq N$,  the function
\begin{equation}
\label{bH.def0}
\bH(x,y) := F_\sigma^{(k)}(x+y)\, b_\sigma(x) 
\end{equation}
satisfies $(P)$.
\item For every integer $0 \leq k \leq N$, the function
\begin{equation}
\label{bG.def0} 
\bG(x,y) = \intzi  F_{\sigma}^{(k)}(x+y+z)  F_{\sigma}'(x+z) dz
\end{equation}
satisfies $(P)$.
\item Let $N \geq 1$ and let $G_\sigma$ be a function satisfying $(P)$. Then the function
\begin{equation}
\label{bF.def0}
\bF(x,y) := \intzi F_\sigma'(x+y+z) G_\sigma(x,z) \, dz
\end{equation}
satisfies $(P)$.
\end{enumerate}
\end{lemma}
\begin{proof}
$(i)$ {\em $\bH$ satisfies $(P1)$.} Clearly  $\bH(x, \cdot) \in \Lyz$ and by the continuity of the translations in $L^2$ one verifies that $\norm{\bH(x+\epsilon, \cdot) - \bH(x,\cdot)}_{\Lyz} \to 0 $ as $\epsilon \to 0$, thus proving that $ \bH \in \CxLy$. \\
We show now that  $\bH \in \Lxyc$. By Lemma \ref{lem:techLemma} $(A1)$,  Proposition \ref{rem:dec_rel} and the assumption on $b_\sigma$, one has that
\begin{equation}
\label{derf.norm1}
 \norm{\bH }_{\Lxyc} \leq C \norm{F_\sigma}_{H^{N+1}} \norm{b_\sigma}_{\Lc} \leq K_c \norm{\sigma}_{H^4_{\zeta,\C} \cap L^2_{N}}^2 \ , 
\end{equation}
 where $K_c >0$ can be chosen locally uniformly for $\sigma \in H^4_{\zeta,\C} \cap L^2_{N}$.
\\
For $0 \leq k \leq N$, $F_\sigma^{(k)} \in C^0(\R, \R)$ by the Sobolev embedding theorem. Thus $\bH \in \Cxy$. 
\\
Finally we show that $\bH(\cdot,0) \in \Lc$. We evaluate the r.h.s. of formula  \eqref{bH.def0} at $y=0$, getting
$$
\bH(x,0)  =  F_{\sigma}^{(k)}(x) b_\sigma(x) \ .$$
It follows that there exists $C >0$ and $K_c >0$, depending  locally uniformly on $\sigma \in H^4_{\zeta,\C} \cap L^2_{N}$, such that
\begin{equation}
\label{norm_derf(x,0)31}
 \norm{\bH(\cdot,0)}_{\Lc}  \leq C \norm{F_\sigma}_{H^{N+1}}\norm{b_\sigma}_{H^1_{x\geq c}} \leq K_c \norm{\sigma}_{H^4_{\zeta,\C} \cap L^2_{N}}^2 \ ,
\end{equation}
where we used that both $F_{\sigma}^{(k)}$ and $b_\sigma$ are in $H^1_{x \geq c}$. 

{\em $\bH$ satisfies $(P2)$.} It follows from \eqref{derf.norm1} and \eqref{norm_derf(x,0)31}. 

{\em $\bH$ satisfies $(P3)$.} The real analyticity property follows from Lemma \ref{lem:techLemma} and Proposition \ref{rem:dec_rel}, since for every $0 \leq k \leq N$, $\bH$ is product of real analytic maps. \\

$(ii)$ {\em $\bG$ satisfies $(P1)$.}
We show  that $\bG \in \Lxyc$. By Lemma \ref{lem:techLemma} $(A5)$ and Proposition \ref{rem:dec_rel} it follows that
\begin{equation}
\label{deryf.norm2}
\norm{\bG}_{\Lxyc} \leq  \norm{F_\sigma}_{H^{N+1}} \norm{F'_\sigma}_{L^1} \leq K_c \norm{\sigma}_{H^4_{\zeta,\C} \cap L^2_{N}}^2 \ , 
\end{equation}
 where $K_c >0$ depends  locally uniformly on $\sigma \in H^4_{\zeta,\C} \cap L^2_{N}$.  One verifies easily that $\bG \in \CxLy$.
 \\
In order to prove that $\bG \in \Cxy, $ proceed as in Lemma \ref{lem:g.p}. 
 \\
Now  we show that $\bG(\cdot, 0) \in \Lc$. We  evaluate formula \eqref{bG.def0} at $y=0$ getting that
$$
\bG(x,0) =  \int_0^{\infty} F_{\sigma}^{(k)}(x+z) F_{\sigma}'(x+z) dz  \ .
$$
Let $h_\sigma'(x,z) := F_{\sigma}'(x+z)$. By Lemma \ref{lem:techLemma} $(A0)$ and Proposition \ref{rem:dec_rel} one has 
$$
\norm{h_\sigma'}_{\Lxyc} \leq \norm{\langle x \rangle^{1/2} F_{\sigma}'}_{\Lc} \leq K_c \norm{\sigma}_{H^4_{\zeta,\C} \cap L^2_{N}} \ ,
$$
where $K_c >0$ can be chosen locally uniformly for $\sigma \in H^4_{\zeta,\C} \cap L^2_{N}$. Thus by Lemma \ref{lem:techLemma} $(A2)$ one gets   
\begin{equation}
\label{dery.f.norm3}
\norm{\bG(\cdot,0)}_{\Lc} \leq  K_c \norm{F_{\sigma}^{(k)}}_{\Lc}\norm{h_\sigma'}_{\Lxyc} \leq K_c \norm{\sigma}_{H^4_{\zeta,\C} \cap L^2_{N}}^2 \ ,
\end{equation}
where $K_c >0$ can be chosen locally uniformly for $\sigma \in H^4_{\zeta,\C} \cap L^2_{N}$.

{\em $\bG$ satisfies $(P2)$.} It follows from \eqref{deryf.norm2} and \eqref{dery.f.norm3}. 

{\em $\bG$ satisfies $(P3)$.} The real analyticity property follows from Lemma \ref{lem:techLemma} and Proposition \ref{rem:dec_rel}, since for every $0 \leq k\leq N$, $\bG$ is composition of real analytic maps.\\

$(iii)$ {\em $\bF$ satisfies $(P1)$.} By Lemma \ref{lem:Phi} $(i)$, $\bF \in \CxLy \cap \Lxyc$ and 
$$\norm{\bF}_{\Lxyc} \leq \norm{F_\sigma'}_{L^2_4} \norm{G_\sigma}_{\Lxyc}
 \leq K_c \norm{\sigma}_{H^4_{\zeta,\C} \cap L^2_{N}}^2 \ . $$
Proceeding as in the proof of Lemma \ref{f.prop} $(P1)$ one shows that $\bF \in \Cxy$.
Since $F_\sigma' \in H^N$, $N \geq 1$, $F_\sigma'$ is a continuous function. Thus we can  evaluate $\bF$ at $y=0$, obtaining $\bF(x,0) = \intzi F_\sigma'(x+z) G_\sigma(x,z) \, dz$. By Lemma \ref{lem:techLemma} $(A2)$ we have that 
 \begin{equation*}
\norm{\bF(\cdot,0)}_{\Lc} \leq \norm{F_\sigma'}_{\Lc} \norm{G_\sigma}_{\Lxyc} \leq   K_c \norm{\sigma}_{H^4_{\zeta,\C} \cap L^2_{N}}^2 \ .
\end{equation*}
The proof that $\bF$ satisfies $(P2)$ and $(P3)$ follows as in the previous items. We omit the details.
\end{proof}

\begin{lemma}
\label{f_sigma.P}
Let $N \geq 1$ be fixed. For every $j_1, j_2 \geq 0$ with $1 \leq j_1 + j_2 \leq N$, the function $f_\sigma^{j_1, j_2}$ defined in \eqref{f_n0} and its  derivatives $\dery f_\sigma^{j_1, j_2}$, $\derx f_\sigma^{j_1, j_2}$ satisfy $(P)$.
\end{lemma}
\begin{proof}
First note that by Lemma \ref{f.prop} the terms  $\derx^{j_1}\dery^{j_2} f_\sigma$ and its  derivatives $\derx^{j_1+1}\dery^{j_2} f_\sigma $, $\derx^{j_1}\dery^{j_2+1} f_\sigma$  satisfy $(P)$. It thus remains to show that 
\begin{equation}
\label{mix.term}
\bF^{k_1, k_2}(x,y) :=   \intzi \derx^{k_1} F_\sigma(x+y+z) \, \derx^{k_2}B_\sigma(x,z) \, dz \ ,  \qquad k_1 \geq 1, \ k_2  \geq 0, \quad k_1 + k_2 = n \leq N 
\end{equation}
and its derivatives $\dery \bF^{k_1,k_2}, \, \derx \bF^{k_1,k_2} $ satisfy $(P)$. 
Remark that, by the induction assumption in the proof of Lemma \ref{prop:AnBn}, for every integers $k_1, k_2 \geq 0$ with $k_1+k_2 \leq n$, $\derx^{k_1} \dery^{k_2} B_\sigma$ satisfies $(P)$.\\

{\em $\bF^{k_1, k_2}$ satisfies $(P)$}. If $k_1=1$, it follows by Lemma \ref{lem:P} $(iii)$. Let $k_1 >1$. By integration by parts $k_1 -1 $ times we obtain
\begin{equation}
\label{bF.R}
\begin{aligned}
\bF^{k_1, k_2}(x,y) = & \sum_{l=1}^{k_1-1} (-1)^l \derx^{k_1 - l} F_\sigma(x+y) (\derx^{k_2}\derz^{l-1} B_\sigma)(x,0) \\
& + (-1)^{k_1-1} \intzi  F_\sigma'(x+y+z)  \, \derx^{k_2} \derz^{k_1-1} B_\sigma(x,z) \, dz \ ,
\end{aligned}
\end{equation}
where we used that for $1 \leq l \leq k_1 -1$ one has $F_\sigma^{(k_1 - l)} \in H^1$ $\, [(\derx^{k_2}\dery^{l-1} B_\sigma)(x,\cdot) \in H^1_{y \geq 0}]$, thus $\lim_{x \to \infty} F_\sigma^{(k_1 - l)}(x) = 0$ $\, [\lim_{y \to \infty} \derx^{k_2}\dery^{l-1} B_\sigma)(x,y) = 0]$.
Consider the  r.h.s. of \eqref{bF.R}. It is a linear combinations of terms of the form \eqref{bH.def0} and \eqref{bF.def0}. By the induction assumption, these terms satisfy the hypothesis of Lemma \ref{lem:P} $(i)$ and $(iii)$. It follows that
$\bF^{k_1, k_2}$ satisfies $(P)$, and in particular there exists a constant $K_c>0$, depending locally uniformly on $\sigma \in H^4_{\zeta,\C} \cap L^2_{N}$, such that
\begin{equation}
\label{F.norm}
\norm{ \bF^{k_1, k_2}}_{\Lxyc} + \norm{ \bF^{k_1, k_2}(\cdot,0)}_{\Lc} \leq K_c \norm{\sigma}_{H^4_{\zeta,\C} \cap L^2_{N}}^2  \ .
\end{equation}

{\em  $\dery \bF^{k_1, k_2}$ satisfies $(P)$.}
For $\epsilon \neq 0$ sufficiently small, by integration by parts $k_1$-times we obtain
\begin{align*}
\frac{\bF^{k_1, k_2}(x,y + \epsilon)  - \bF^{k_1, k_2}(x,y)}{\epsilon}  = & \sum_{l=1}^{k_1}(-1)^{l}  \frac{\derx^{k_1-l}F_\sigma(x+y + \epsilon) - \derx^{k_1-l} F_\sigma(x+y) }{\epsilon} (\derx^{k_2}\derz^{l-1} B_\sigma)(x,0) \\
 & + (-1)^{k_1} \intzi  \frac{ F_\sigma(x+y+\epsilon +z) - F_\sigma(x+y+z) }{\epsilon} \, \derx^{k_2} \derz^{k_1} B_\sigma(x,z) \, dz \ ,
\end{align*}
where once again we used that for $1 \leq l \leq k_1$ one has $F_\sigma^{(k_1 - l)} \in H^1$ $\, [(\derx^{k_2}\dery^{l-1} B_\sigma)(x,\cdot) \in H^1_{y \geq 0}]$, thus $\lim_{x \to \infty} F_\sigma^{(k_1 - l)}(x) = 0$ $\, [\lim_{y \to \infty} \derx^{k_2}\dery^{l-1} B_\sigma)(x,y) = 0]$.
Define also 
\begin{equation}
\begin{aligned}
\label{dery.bF}
\dery \bF^{k_1, k_2}(x,y)  :=  & \sum_{l=1}^{k_1}(-1)^{l} \, \derx^{k_1-l + 1} F_\sigma(x+y) \,  (\derx^{k_2}\derz^{l-1} B_\sigma)(x,0) \\
& +  (-1)^{k_1} \intzi  F_\sigma'(x+y+z) \, \derx^{k_2} \derz^{k_1} B_\sigma(x,z) \, dz \ .
\end{aligned}
\end{equation}
Consider the r.h.s. of equation \eqref{dery.bF}. It is a linear combinations of terms of the form \eqref{bH.def0} and \eqref{bF.def0}. By the induction assumption, these terms satisfy the hypothesis of Lemma \ref{lem:P} $(i)$ and $(iii)$. It follows that
$\dery \bF^{k_1, k_2}$ satisfies $(P)$ and one has
\begin{equation}
 \label{deryF.norm}
 \norm{\dery \bF^{k_1, k_2}}_{\Lxyc} + \norm{\dery \bF^{k_1, k_2}(\cdot, 0)}_{\Lc} \leq K_c' \norm{\sigma}_{H^4_{\zeta,\C} \cap L^2_{N}}^2
 \end{equation}
  for some constant $K_c'>0$,  depending  locally uniformly on $\sigma \in H^4_{\zeta,\C} \cap L^2_{N}$. 
 Furthermore one verifies that
$$
\lim_{\epsilon \to 0} \frac{\bF^{k_1, k_2}(x,\cdot + \epsilon)  - \bF^{k_1, k_2}(x,\cdot)}{\epsilon} = \dery \bF^{k_1, k_2}(x,\cdot) \quad \mbox{ in } \Lyz \ .
$$

{\em $\derx\bF^{k_1, k_2}$ satisfies $(P)$.} The proof is similar to the previous case, and the details are omitted.

This conclude the proof of the inductive step.
\end{proof}

\section{Hilbert transform}
\label{Hilbert.transf}
Define $\HT : L^2(\R, \C) \to L^2(\R, \C)$ as the Fourier multiplier operator
\[\widehat{(\HT (v))}(\xi)= -i\operatorname{sign}(\xi)\; \hat{v}(\xi) \ .\]
Thus $\HT$ is an isometry on $L^2(\R, \C)$. It is easy to see that $\left.\HT\right|_{H^N_\C}: H^N_\C \to H^N_\C$ is an isometry for any $N\geq 1$ -- cf. \cite{duo}.
In case $v\in C^1(\R, \C)$ with $\|v'\|_{L^\infty},\|xv(x)\|_{L^\infty}<\infty$, one has   \[\HT(v)(k)=-\frac{1}{\pi} \lim_{\epsilon \to 0^+} \int_{|k'-k|\geq \epsilon}\frac{v(k')}{k'-k} dk' \]
and obtains the estimate $\mmod{\HT(v)(k)} \leq C (\|v'\|_{\infty}+\|xv(x)\|_{\infty})$, where  $C>0$ is a constant independent of $v$ and $k$.

Let $g \in  C^1(\R, \R)$ with $\|g'\|_{L^\infty},\|xg(x)\|_{L^\infty}<\infty$. 
Then define for $z\in \C^+:= \{z\in \C: \Im(z)> 0\} $  
the function
\begin{align}\nonumber
f(z):=  \frac{1}{\pi i}\int_{-\infty}^\infty \frac{g(s)}{s-z}ds \ .
\end{align}
Decompose $\frac{1}{s-z}$ into real and imaginary part
\[\frac{1}{s-z}= \frac{1}{s-a- ib}= \frac{s-a}{(s-a)^2+b^2}+i \frac{b}{(s-a)^2+b^2}\]
to get the formulas for the real and imaginary part of $f(z)$
 \begin{align}\label{realpartformula}
 \operatorname{Re} f(z) =&\frac{1}{\pi }\int_{-\infty}^\infty \ 
 \frac{b}{(s-a)^2+b^2} g(s) ds \ , \\
 \label{imaginarypartformula}
 \operatorname{Im} f(z)=& \frac{-1}{\pi }\int_{-\infty}^\infty \frac{s-a}{(s-a)^2+b^2 }g(s)ds \ .
\end{align}

The following Lemma is well known and can be found in \cite{duo}.
\begin{lemma}\label{hilbertdeltausw}
The function $f$ is analytic and admits a continuous extension  to the real line. Furthermore it has the following properties for any $a \in \R$:
\begin{enumerate}[(i)]
\item  $ \lim_{b\to 0^+}\operatorname{Im} f(a+bi)=\mathcal{H}(g)(a)$.
 \item $ \lim_{b\to 0^+}\operatorname{Re} f(a+bi) = g(a)$.
 \item There exists $C>0$ such that $|f(z)| \leq \frac{C}{1+|z|}, $ $\, \forall \, z \in \{z: \Im z \geq 0 \}$.
\item Let $\tilde f(z)$ be a continuous function on $\Im z \geq 0$ which is analytic on $\Im z > 0$ and satisfies  $\operatorname{Re} \tilde f_{|\R}= g$ and $|\tilde f(z)| = O(\frac{1}{|z|})$ as $|z|\to \infty$, then 
$\tilde f = f$.
 \end{enumerate}
\end{lemma}

The next lemma follows from the  commutator estimates due to  Calder\'on \cite{calderon}:
\begin{lemma}[\cite{calderon}]
\label{lem:comm.est}
Let $b:\R \to \R$ have first-order derivative in $L^\infty$. 
For any $p \in (1, \infty)$ there exists $C>0$, such that
$$
\norm{\left[ \H, b \right] \derx g}_{L^p} \leq C \norm{g}_{L^p} \ . 
$$
\end{lemma}

We apply this lemma to prove the following result:
\begin{lemma}
\label{lem:hilb.zeta}
Let $M \in \Z_{\geq 1}$ be fixed. Then $\H: \Hzc \to \Hzc$ is a bounded linear operator.
\end{lemma}
\begin{proof}
Let $f \in \Hzc$. As the Hilbert transform commutes with the derivatives, we have that $\H(f) \in H^{M-1}_\C$. Next we show that if $\zeta \derk^M f \in L^2$, then $\zeta \derk^M \H(f) = \zeta \H(\derk^M f) \in L^2$. By Lemma \ref{lem:comm.est} with $p=2$,   $g = \derk^{M-1}f$ and $b = \zeta$, we have that
$$
\norm{\zeta \H(\derk^M f)}_{L^2} \leq \norm{ \H(\zeta \derk^M f)}_{L^2} + \norm{\left[ \H, \zeta \right] \derk^M f }_{L^2} \leq \norm{f}_{\Hzc} + C \norm{\derk^{M-1}f}_{L^2} < \infty \ .
$$
\end{proof}

{\bf Acknowledgments} We are particularly grateful to Thomas Kappeler for his continued support and the numerous discussions about the paper. The first author is supported by the Swiss National Science Foundation. This paper is part of the first author's PhD thesis.

\end{document}